\documentclass[reqno]{amsart}

\usepackage{etoolbox}

\patchcmd{\abstract}{\null\vfil}{}{}{}

\usepackage{hyperref}
\usepackage[dvipsnames]{xcolor}
\usepackage{smartdiagram}
\usepackage{tikz}
\usepackage{amsmath,bm,bbm,amsthm, amssymb}
\usepackage[scale=0.8,twoside=false]{geometry}
\usepackage{color}
\usepackage{dsfont}
\usepackage{animate}
\usepackage{etoolbox}
\usepackage{comment}
\usetikzlibrary{calc}
\usetikzlibrary{patterns}
\usetikzlibrary{arrows}
\usetikzlibrary{decorations.pathreplacing}
\usepackage[utf8]{inputenc}

\usepackage[most]{tcolorbox}
\usepackage{adjustbox}
\usepackage[final]{pdfpages}
\usepackage[ruled,vlined,linesnumbered,algosection,resetcount]{algorithm2e}
\usepackage{blkarray}
\usepackage{wrapfig}
\usepackage{ marvosym }
\usepackage{xfrac,mathrsfs}
\usepackage{mathtools}

\usepackage[foot]{amsaddr}
\usepackage{float}
\theoremstyle{plain}
\newtheorem{theorem}{Theorem}
\newtheorem{proposition}[theorem]{Proposition}
\newtheorem{corollary}[theorem]{Corollary}
\newtheorem{lemma}[theorem]{Lemma}

\newtheorem{claim}[theorem]{Claim}

\theoremstyle{definition}
\newtheorem{definition}[theorem]{Definition}

\newtheorem{example}[theorem]{Example}

\theoremstyle{remark}
\newtheorem{remark}[theorem]{Remark}

\definecolor{wiasblue}   {cmyk}{1.0, 0.60, 0, 0}

\def\ni{\noindent}

\def\bep{\begin{proof}}
\def\enp{\end{proof}}
\def\bepr{\begin{proposition}}
\def\enpr{\end{proposition}}
\def\bec{\begin{corollary}}
\def\enc{\end{corollary}}
\def\bea{\begin{align}}
\newcommand\eea{\end{align}}
\def\beas{\begin{align*}}
\def\eeas{\end{align*}}
\def\bet{\begin{theorem}}
\def\ent{\end{theorem}}
\def\bee{\begin{example}}
\def\ene{\end{example}}

\def\bede{\begin{definition}}
\def\ende{\end{definition}}
\def\ber{\begin{remark}}
\def\enr{\end{remark}}
\def\beca{\begin{cases}}
\def\enca{\end{cases}}
\def\bel{\begin{lemma}}
\def\enl{\end{lemma}}
\def\been{\begin{enumerate}}
\def\enen{\end{enumerate}}

\def\beit{\begin{itemize}}
\def\enit{\end{itemize}}
\def\befr{\begin{frame}}
\def\enfr{\end{frame}}

\renewcommand\le{\leqslant}
\renewcommand\ge{\geqslant}

\def\becbb{\begin{center}\begin{tcolorbox}[{colback=Dandelion!20}]}
\def\encbb{\end{tcolorbox}\end{center}}
\def\beccb{\begin{center}\begin{tcolorbox}[{colback=Dandelion!20}]}
\def\enccb{\end{tcolorbox}\end{center}}
\def\becb{\begin{center}\begin{tcbox}[{colback=Dandelion!20}]}
\def\encb{\end{tcbox}\end{center}}

\def\bef{\begin{figure}[!h]}
\def\enf{\end{figure}}
\def\bels{\begin{lstlisting}}
\def\enls{\end{lstlisting}}

\def\betp{\begin{tikzpicture}}
\def\entp{\end{tikzpicture}}

\def\endo{\end{document}}

\usepackage{todonotes}

\usepackage{comment}

\begin{document}

\title{The contact process on a bipartite spatial network}
\author{John Fernley}
\author{Christian Hirsch}
\author{Daniel Valesin}
\address[John Fernley]{Department of Statistics (CR{\rm i}SM), University of Warwick, Coventry CV4 7AL, United Kingdom}
\address[Christian Hirsch]{Department of Mathematics, Aarhus University, Ny Munkegade, 119, 8000, Aarhus C, Denmark}
\address{\emph{and} DIGIT Center, Aarhus University, Finlandsgade 22, 8200 Aarhus N, Denmark}
\address[Daniel Valesin]{Department of Statistics, University of Warwick, Coventry CV4 7AL, United Kingdom}
\email{hirsch@math.au.dk, john.fernley@warwick.ac.uk, daniel.valesin@warwick.ac.uk}

\begin{abstract}
    We study the contact process on a random bipartite connection hypergraph generated from two Poisson point processes, with mark-dependent connection thresholds. For asymmetric infection rates and asymmetric power law tail decays of the two degree distributions, we determine the dominant survival strategies in all parameter regimes and provide asymptotics for the epidemic probability up to logarithmic factors. 
\end{abstract}


\maketitle
\thispagestyle{empty}

\frenchspacing

\section{Introduction}

\begin{wrapfigure}{R}{0.5\textwidth}
    \centering
    \fbox{\begin{tikzpicture}[scale=1.5]
  \draw[black!60] (0.856,0.685) -- (1.284,1.341);
  \draw[black!60] (0.709,3.136) -- (0.663,2.058);
  \draw[black!60] (0.709,3.136) -- (2.428,4.626);
  \draw[black!60] (0.709,3.136) -- (1.284,1.341);
  \draw[black!60] (0.709,3.136) -- (0.251,3.582);
  \draw[black!60] (0.709,3.136) -- (0.915,2.390);
  \draw[black!60] (0.709,3.136) -- (0.344,4.201);
  \draw[black!60] (0.709,3.136) -- (0.916,3.634);
  \draw[black!60] (0.709,3.136) -- (0.937,1.315);
  \draw[black!60] (0.709,3.136) -- (4.066,0.217);
  \draw[black!60] (0.709,3.136) -- (0.583,2.756);
  \draw[black!60] (3.056,3.554) -- (3.888,3.176);
  \draw[black!60] (3.056,3.554) -- (2.428,4.626);
  \draw[black!60] (3.056,3.554) -- (1.284,1.341);
  \draw[black!60] (3.056,3.554) -- (3.694,4.563);
  \draw[black!60] (3.056,3.554) -- (4.066,0.217);
  \draw[black!60] (3.056,3.554) -- (2.901,2.191);
  \draw[black!60] (3.056,3.554) -- (3.159,3.248);
  \draw[black!60] (3.056,3.554) -- (0.583,2.756);
  \draw[black!60] (3.056,3.554) -- (4.657,2.123);
  \draw[black!60] (0.954,1.555) -- (1.284,1.341);
  \draw[black!60] (0.954,1.555) -- (1.057,1.252);
  \draw[black!60] (0.954,1.555) -- (0.937,1.315);
  \draw[black!60] (2.787,0.855) -- (2.637,0.146);
  \draw[black!60] (2.787,0.855) -- (4.066,0.217);
  \draw[black!60] (3.060,3.118) -- (3.159,3.248);
  \draw[black!60] (2.565,4.239) -- (2.428,4.626);
  \draw[black!60] (3.191,4.837) -- (2.428,4.626);
  \draw[black!60] (2.412,2.618) -- (1.284,1.341);
  \draw[black!60] (2.412,2.618) -- (2.318,1.480);
  \draw[black!60] (2.412,2.618) -- (4.066,0.217);
  \draw[black!60] (2.412,2.618) -- (2.901,2.191);
  \draw[black!60] (2.412,2.618) -- (3.159,3.248);
  \draw[black!60] (2.412,2.618) -- (0.583,2.756);
  \draw[black!60] (1.444,0.485) -- (1.284,1.341);
  \draw[black!60] (1.444,0.485) -- (1.729,0.059);
  \draw[black!60] (1.444,0.485) -- (2.318,1.480);
  \draw[black!60] (1.444,0.485) -- (1.057,1.252);
  \draw[black!60] (1.444,0.485) -- (2.637,0.146);
  \draw[black!60] (1.444,0.485) -- (0.937,1.315);
  \draw[black!60] (1.444,0.485) -- (4.066,0.217);
  \draw[black!60] (1.444,0.485) -- (2.264,1.009);
  \draw[black!60] (1.444,0.485) -- (1.628,1.208);
  \draw[black!60] (0.206,4.771) -- (0.344,4.201);
  \draw[black!60] (0.206,4.771) -- (0.916,3.634);
  \draw[black!60] (0.935,4.082) -- (0.916,3.634);
  \draw[black!60] (1.465,2.800) -- (1.284,1.341);
  \draw[black!60] (1.465,2.800) -- (0.915,2.390);
  \draw[black!60] (1.465,2.800) -- (0.583,2.756);
  \draw[black!60] (3.643,0.160) -- (4.066,0.217);
  \draw[black!60] (2.473,3.800) -- (2.428,4.626);
  \draw[black!60] (2.473,3.800) -- (1.284,1.341);
  \draw[black!60] (2.473,3.800) -- (0.915,2.390);
  \draw[black!60] (2.473,3.800) -- (3.694,4.563);
  \draw[black!60] (2.473,3.800) -- (0.916,3.634);
  \draw[black!60] (2.473,3.800) -- (4.066,0.217);
  \draw[black!60] (2.473,3.800) -- (2.901,2.191);
  \draw[black!60] (2.473,3.800) -- (3.159,3.248);
  \draw[black!60] (2.473,3.800) -- (0.583,2.756);
  \draw[black!60] (2.473,3.800) -- (4.657,2.123);
  \draw[black!60] (0.793,0.992) -- (1.284,1.341);
  \draw[black!60] (0.793,0.992) -- (1.729,0.059);
  \draw[black!60] (0.793,0.992) -- (0.915,2.390);
  \draw[black!60] (0.793,0.992) -- (1.057,1.252);
  \draw[black!60] (0.793,0.992) -- (0.937,1.315);
  \draw[black!60] (0.793,0.992) -- (0.583,2.756);
  \draw[black!60] (0.793,0.992) -- (1.628,1.208);
  \draw[black!60] (2.219,4.306) -- (2.428,4.626);
  \draw[black!60] (3.358,2.053) -- (3.119,2.117);
  \draw[black!60] (1.492,2.442) -- (1.284,1.341);
  \draw[black!60] (1.492,2.442) -- (0.915,2.390);
  \draw[black!60] (1.492,2.442) -- (0.583,2.756);
  \draw[black!60] (2.827,4.456) -- (2.428,4.626);
  \draw[black!60] (0.256,4.252) -- (0.344,4.201);
  \draw[black!60] (2.030,0.883) -- (1.284,1.341);
  \draw[black!60] (2.030,0.883) -- (1.729,0.059);
  \draw[black!60] (2.030,0.883) -- (2.318,1.480);
  \draw[black!60] (2.030,0.883) -- (2.637,0.146);
  \draw[black!60] (2.030,0.883) -- (4.066,0.217);
  \draw[black!60] (2.030,0.883) -- (2.264,1.009);
  \draw[black!60] (2.030,0.883) -- (1.628,1.208);
  \draw[black!60] (3.761,3.589) -- (3.888,3.176);
  \draw[black!60] (3.761,3.589) -- (2.428,4.626);
  \draw[black!60] (3.761,3.589) -- (3.159,3.248);
  \draw[black!60] (3.761,3.589) -- (4.657,2.123);
  \draw[black!60] (4.720,1.252) -- (4.066,0.217);
  \fill[blue] (0.856,0.685) circle (0.032);
  \fill[blue] (4.510,2.866) circle (0.031);
  \fill[blue] (0.709,3.136) circle (0.149);
  \fill[blue] (3.056,3.554) circle (0.118);
  \fill[blue] (4.555,4.787) circle (0.047);
  \fill[blue] (0.954,1.555) circle (0.031);
  \fill[blue] (2.787,0.855) circle (0.049);
  \fill[blue] (3.060,3.118) circle (0.046);
  \fill[blue] (2.565,4.239) circle (0.037);
  \fill[blue] (3.191,4.837) circle (0.047);
  \fill[blue] (2.412,2.618) circle (0.086);
  \fill[blue] (1.444,0.485) circle (0.108);
  \fill[blue] (0.206,4.771) circle (0.079);
  \fill[blue] (0.935,4.082) circle (0.042);
  \fill[blue] (1.465,2.800) circle (0.050);
  \fill[blue] (3.643,0.160) circle (0.031);
  \fill[blue] (2.473,3.800) circle (0.154);
  \fill[blue] (0.793,0.992) circle (0.096);
  \fill[blue] (2.164,3.594) circle (0.045);
  \fill[blue] (2.027,3.160) circle (0.032);
  \fill[blue] (2.219,4.306) circle (0.053);
  \fill[blue] (3.358,2.053) circle (0.037);
  \fill[blue] (1.492,2.442) circle (0.063);
  \fill[blue] (1.847,3.045) circle (0.039);
  \fill[blue] (2.827,4.456) circle (0.042);
  \fill[blue] (0.256,4.252) circle (0.034);
  \fill[blue] (2.030,0.883) circle (0.068);
  \fill[blue] (3.761,3.589) circle (0.083);
  \fill[blue] (4.720,1.252) circle (0.036);
  \fill[red] (0.663,2.058) circle (0.032);
  \fill[red] (3.888,3.176) circle (0.034);
  \fill[red] (2.428,4.626) circle (0.078);
  \fill[red] (4.735,0.310) circle (0.078);
  \fill[red] (1.284,1.341) circle (0.112);
  \fill[red] (0.251,3.582) circle (0.038);
  \fill[red] (3.740,2.344) circle (0.030);
  \fill[red] (1.729,0.059) circle (0.056);
  \fill[red] (2.318,1.480) circle (0.050);
  \fill[red] (0.915,2.390) circle (0.054);
  \fill[red] (3.484,1.110) circle (0.035);
  \fill[red] (1.057,1.252) circle (0.043);
  \fill[red] (3.694,4.563) circle (0.040);
  \fill[red] (0.344,4.201) circle (0.032);
  \fill[red] (0.916,3.634) circle (0.063);
  \fill[red] (2.637,0.146) circle (0.060);
  \fill[red] (0.937,1.315) circle (0.048);
  \fill[red] (4.066,0.217) circle (0.124);
  \fill[red] (3.817,1.858) circle (0.032);
  \fill[red] (3.119,2.117) circle (0.031);
  \fill[red] (3.630,0.652) circle (0.055);
  \fill[red] (3.273,0.629) circle (0.035);
  \fill[red] (2.901,2.191) circle (0.047);
  \fill[red] (3.159,3.248) circle (0.046);
  \fill[red] (0.583,2.756) circle (0.079);
  \fill[red] (2.264,1.009) circle (0.036);
  \fill[red] (4.657,2.123) circle (0.075);
  \fill[red] (1.628,1.208) circle (0.045);
\end{tikzpicture}
}\\[0.5\baselineskip]
    \caption{Random connection hypergraph with two types of vertices (red and blue). Each vertex has a mark (represented by the size of the circle) that determines its connection reach and so its expected degree.}
    \label{fig:rch}
\end{wrapfigure}
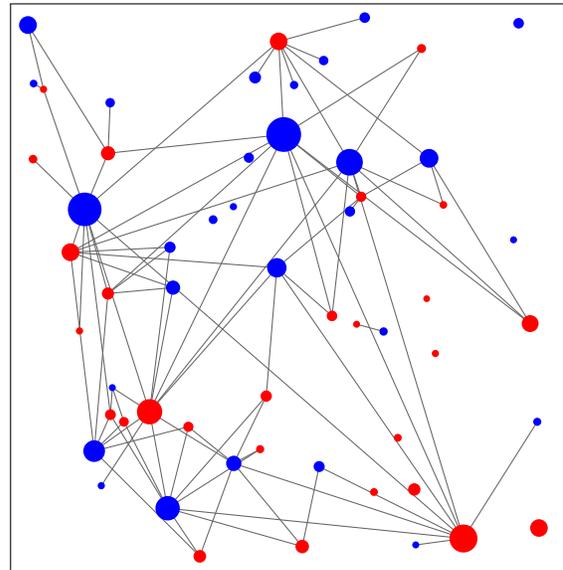

The contact process is one of the most fundamental stochastic processes for modelling the spread of an infection on a network \cite{liggett1985,liggett2013stochastic}. In the present work, we study the contact process on a complex spatial network based on a bipartite vertex set. 

Before defining the model, we give two motivations for bipartite structures in epidemic modelling. The first example is the spreading of sexually transmitted infections on heterosexual contact networks. In this situation, physiological differences produce a difference in infection rate between the two parts \cite{wong2004gender}. Moreover, \cite{gomez2008spreading} show that both parts exhibit power-law behaviour, with different tail exponents.
Our results identify the epidemic probability in this context, and show that it can differ by orders of magnitude between the two sexes.

The second example comes from higher-order infection processes. For instance, \cite{simp_cont} report an experiment in which social contagion is shown to depend on group-level reinforcement rather than purely pairwise transmission. Hypergraph models therefore offer a natural setting in which local groups of vertices may jointly support infection, even when individual pairwise connections are too weak. Bipartite graphs naturally represent such hypergraphs: one part corresponds to individuals, while the other represents groups. An individual is connected to a group if they belong to it, and infection can spread from groups to individuals and vice versa. This interpretation is natural for gathering places, where an airborne infection can be viewed as first infecting the place and then being transmitted from the place to other occupants.

In this work we study the contact process on the random connection hypergraph introduced in \cite{juh} and illustrated in Figure \ref{fig:rch}. This model can be thought of as a bipartite variant of the weight-dependent random connection model studied in \cite{glm2,komjathy}. The construction starts from two independent Poisson point processes on the real line, each equipped with a mark in the unit interval. Points from the first process are connected to points from the second whenever their spatial distance is at most a mark-dependent threshold. For suitable choices of the mark exponents, this produces a bipartite graph whose connectivity depends delicately on the balance of the two mark distributions. A further motivation for considering spatial models is that they capture structural features of real networks that are often absent in purely combinatorial models.  More precisely, in spatial models vertices that are close to the same region of space are likely to have common neighbours, leading naturally to nontrivial clustering, a feature that is frequently observed in real-world networks.

Conditioned on this environment, we consider the contact process with infection rates \(\lambda_1\) from the first type to the second and \(\lambda_2\) in the opposite direction. We start the infection from a typical point and write \(\theta(\lambda_1,\lambda_2)\) for the survival probability. Our focus is the regime where \(\lambda_1=\lambda\) and \(\lambda_2=\lambda^a\) with \(\lambda\) tending to zero. 
This question is analogous to the questions studied in \cite{grac,link} for age-dependent random connection models and random hyperbolic  graphs. However, the bipartite structure makes the present model more challenging, and the resulting phase diagram exhibits substantially more complicated behaviour.

The main new feature is that the bipartite structure allows survival strategies based on strongly unbalanced degrees, and the infection may exploit several qualitatively different routes to long-time survival, compared to those that appear for the classical contact process on power-law random graphs \cite{grac,link,jacob,jacob2024metastability}.

Our main theorem, Theorem~\ref{thm_main}, identifies the dominant survival strategy in every parameter regime and describes the resulting asymptotics of \(\theta(\lambda_1,\lambda_2)\) up to logarithmic factors. The outcome is a phase diagram with several distinct regions, each associated with a different optimal route to survival.

The rest of this paper is organised as follows. Section~\ref{s_model} contains a precise description of the model and the statement of the main theorem.  The proof of the theorem is split into a lower bound, which is proved in Section~\ref{s_lower_bound}, and an upper bound, proved in Section~\ref{s_upper_bound}. Finally, in the \hyperref[appn]{Appendix}, we carry out several computations needed to give an explicit  description of the phase diagram.


\section{Background, model definition and main results}\label{s_model}

\subsection{Classical contact process}
Although our focus is the contact process with asymmetric infection rates on bipartite graphs, we begin with some background on the classical contact process, i.e. the $\lambda_1=\lambda_2$ case. 

The \emph{contact process} (introduced by Harris~\cite{harris}) with infection rate~$\lambda > 0$ on a graph~$G=(V,E)$ is a Markov process on~$\{0,1\}^V$ where: vertices in state~$0$ are said to be \emph{healthy}, and vertices in state~$1$ are said to be \emph{infected}; infected vertices recover with rate~1, and infected vertices transmit the infection to each neighbour with rate~$\lambda$. We denote the process (with some arbitrary initial configuration) by~$(\xi_t)_{t \ge 0} = (\xi_t(x):x \in V)_{t \ge 0}$. We adopt the usual identification of~$\{0,1\}^V$ with the power set of~$V$, via~$\{0,1\}^V \ni \xi \mapsto \{x\in V: \xi(x)=1\} \subseteq V$.

The identically-zero, ``all-healthy'' configuration (denoted by~$\varnothing$, via the above identification) is an absorbing state. We define the probability of survival of the infection as
\[
	\theta(\lambda) := \mathbb P( \forall t, \;\; \xi_t \neq \varnothing );
\]
this of course depends on~$G$,~$\lambda$, and~$\xi_0$. If one restricts attention to initial configurations with finitely many infected vertices, then it holds that either~$\theta(\lambda)=0$ for all~$\xi_0$, or~$\theta(\lambda) > 0$ for all~$\xi_0$. We say that the process \emph{dies out} in the former case, and \emph{survives} in the latter one.

It is known since~\cite{harris} that, on the Euclidean lattice, there is a phase transition: there is a positive threshold~$\lambda_c$ such that the process dies out if~$\lambda < \lambda_c$ and survives if~$\lambda > \lambda_c$ (later, in~\cite{bezuidenhout1990critical}, Bezuidenhout and Grimmett proved that the process dies out at~$\lambda=\lambda_c$).

 In contrast, on graphs where vertex degrees are unbounded, it can happen that the infection survives \emph{regardless of~$\lambda$} (so that~$\lambda_c = 0$). This is known to hold for Bienaym\'e--Galton--Watson (BGW) trees where the offspring distribution has no finite exponential moment, see Huang and Durrett~\cite{huang2020contact}; Bhamidi, Nam, Nguyen, and Sly~\cite{bhamidi2021survival} proved that the moment assumption is sharp.

Studying the asymptotics of~$\theta(\lambda)$ as~$\lambda \downarrow \lambda_c$ (with a fixed initial configuration, say, consisting of a single infected site) is illuminating: it reveals the ``best strategy'' adopted by the infection, in the event that it survives (which becomes rare when~$\lambda \to 0$). This is mathematically more feasible in settings where~$\lambda_c=0$, because the process with near-zero infection rate is easier to study.

This kind of analysis was first carried out in~\cite{dch} and~\cite{mountford2013metastable} for BGW trees, with implications for the so-called \emph{metastable density} of the contact process on the configuration model, a random graph that converges locally to a two-stage BGW tree. See also~\cite{van2015metastability}, and~\cite[Chapter 5]{valesin2024survey} for an overview of this topic.
In that setting, depending on the degree distribution, one of the following two survival strategies is adopted:
\begin{itemize}
	\item \emph{direct propagation}: the infection survives by propagating outwards from the root (so that re-infections in the direction of the root are not needed), and travelling through vertices of increasingly high degrees. This strategy is optimal in cases where the degree distribution has heavier tail.
	\item \emph{propagation through stars}: the infection reaches a vertex with degree much larger than~$1/\lambda^2$ (``star''), which sustains it for a long time; while being active there, the infection manages to find another nearby star, which also sustains it for long enough to reach another star etc.
\end{itemize}

In further works, it has been found that these strategies (and corresponding expressions for the survival probability) reappear in other families of random graphs. These include the preferential attachment graph~\cite{vanhaocan17}, random hyperbolic graphs~\cite{link}, and scale-free geometric random graphs~\cite{grac}; the latter two are random graphs embedded in space, with local clustering. We should also mention that similar analyses have been carried out for the contact process on \emph{dynamic} random graphs, with very rich phase diagrams; see~\cite{jacob2024metastability,jacob}.

\subsection{Random connection hypergraph}
We now define the \emph{random connection hypergraph}, a random bipartite graph~$\mathcal G$ embedded in space, that was introduced in \cite{juh}. We take two independent point processes ~$\mathcal P_1,\mathcal P_2 \subset \mathbb R \times [0,1]$, where:
\begin{itemize}
	\item $\mathcal P_1$ is a Poisson point process with intensity equal to the Lebesgue measure, with one extra element (called the \emph{root}, and denoted by~$o$) placed at~$(0,U)$, where~$U \sim \mathrm{Unif}(0,1)$, independently of the Poisson points;
	\item $\mathcal P_2$ is a Poisson point process with intensity equal to the Lebesgue measure.
\end{itemize}
The vertex set of~$\mathcal G$ is~$\mathcal P_1 \cup \mathcal P_2$; elements of~$\mathcal P_1$ and of~$\mathcal P_2$ are called \emph{type-1 vertices} and \emph{type-2 vertices}, respectively. For any vertex~$(x,u) \in\mathcal P_1 \cup \mathcal P_2$, we refer to the second coordinate~$u$ as the \emph{height} or~\emph{mark} of the vertex.

We now fix parameters $\gamma_1, \gamma_2 \in (0, 1)$. The edge set of~$\mathcal G$ is then
\[
	\{\{(x,u),(y,v)\}: (x,u) \in \mathcal P_1,\; (y,v) \in \mathcal P_2,\; |x-y| \le u^{-\gamma_1}v^{-\gamma_2}\}.
\]

This produces a random graph where, roughly speaking, the degree distributions of type-1 and type-2 vertices, denoted~$p_1$ and~$p_2$ respectively, satisfy as $m\rightarrow\infty$
\begin{equation*}
	p_1((m,\infty))	\sim \left(\frac{\bar{\gamma}_2}{2} m \right)^{-1/\gamma_1},\qquad p_2((m,\infty))\sim \left(\frac{\bar{\gamma}_1}{2} m \right)^{-1/\gamma_2}.
\end{equation*}

\begin{remark}[Metastable density]
By adding a point at $(0,U)$, by Slivnyak's theorem, we are conditioning that there is a point at spatial position $0$. In the finite box version of our model (restricting to points inside some large spatial interval), this is also the local distribution of a uniform point translated to $0$: so it is the correct root for the survival probability  of the infection in our main theorem to correspond also to a \emph{metastable density} for that finite model. Note that we follow \cite{grac,fernley2024targeted} in finding the survival probability for the local limit and omitting the standard duality proof to actually equate this to the density.
\end{remark}

We assume throughout this work that the pair~$(\gamma_1,\gamma_2)$ belongs to the set
\begin{equation*}
	\mathcal S:= \{(\gamma_1,\gamma_2) \in (0,1)^2:\; \gamma_1 + \gamma_2 > 1\}.
\end{equation*}

The reason for this assumption is that we want a graph where the contact process can potentially survive forever; in particular, a graph with at least one infinite connected component. If~$\gamma_1 + \gamma_2 < 1$, then the graph does not have infinite components, a fact that we now prove for completeness.

\begin{proposition}[Sub-critical regime]
\label{pr:subcrit}
If $\gamma_1 + \gamma_2 < 1$, then~$\mathcal G$ has no infinite connected component almost surely.
\end{proposition}
\begin{proof}
	The proof is inspired by the proof of Theorem~3(ii) in~\cite{gracar2023emergence}.
For~$z \in \mathbb R$, define the event
\[
A(z):=\left\{ \text{$\mathcal G$ has no edge $\{(x,u),(y,v)\}$ with $(x,u) \in \mathcal P_1$, $(y,v) \in \mathcal P_2$ and $x \le z \le y$} \right\}.
\]

We will prove that~$A(0)$ has positive probability. By ergodicity, it will then follow that almost surely, there is an unbounded set of~$z$ for which~$A(z)$ occurs, hence there is no infinite cluster.

	Fix~$\kappa \in (1,(\gamma_1+\gamma_2)^{-1})$. Define~$S:=\{(z,t)\in \mathbb R\times [0,1]:\; t \le |z|^{-\kappa}\}$, and let~$B$ be the event that~$\mathcal P_1 \cap S = \mathcal P_2 \cap S = \varnothing$. Since~$S$ has finite Lebesgue measure, we have~$\mathbb P(B) > 0$. We now claim that~$B \subseteq A(0)$. To see this, assume that~$B$ occurs, and let~$(x,u) \in \mathcal P_1$ and~$(y,v) \in \mathcal P_2$ with~$x<0<y$; we then have
\[
u^{-\gamma_1}v^{-\gamma_2} < |x|^{\kappa \gamma_1} y^{\kappa\gamma_2} \le \max\{|x|,y\}^{\kappa(\gamma_1+\gamma_2)} \le \max\{|x|,y\} \le  |x|+y,
\]
so there is no edge between~$(x,u)$ and~$(y,v)$. 
\end{proof}

It will be useful to introduce the quantities
\[
	\bar{\gamma}_i = 1- \gamma_i,\qquad 	\Delta_i := (2\gamma_i - 1) \vee 0,\qquad i = 1,2,
\]
and
\[
\Delta:=\gamma_1+\gamma_2-1.
\]

\subsection{Contact process on~$\mathcal G$ with type-dependent infection rates}

Next, we introduce a version of the {contact process} on~$\mathcal G$ where we allow distinct infection rates for the two types (one could also allow for distinct recovery rates, but we have chosen not to do so for simplicity). More concretely, the dynamics is given by the rules:
\begin{itemize}
	\item an infected vertex of either type recovers with rate~1;
	\item an infected vertex of type 1 transmits the infection to each neighbour with rate~$\lambda_1$;
	\item an infected vertex of type 2 transmits the infection to each neighbour with rate~$\lambda_2$.
\end{itemize}
We will typically be interested in the process started from the configuration where the (type-1) root is infected, and all other vertices are healthy. Moreover, we will study asymptotics when the infection rates are taken to~$0$, and with this in mind, we introduce an additional parameter~$a > 0$ and take infection rates
\[
\lambda_1 = \lambda,\qquad \lambda_2 = \lambda^a.
\]

With these rates, recall that we define
\[
	\theta(\lambda) := \mathbb P( \forall t, \;\; \xi_t \neq \varnothing ),
\]
where the process is started from only the root infected. Note that the above probability is \emph{annealed} (it involves both the randomness in the choice of the graph and in the evolution of the infection).

Our main result, Theorem~\ref{thm_main} below, describes~$\theta(\lambda)$ as~$\lambda \to 0$. Before we state it, we give some heuristic explanations and some definitions.  We begin with three heuristic observations, and throughout this discussion we assume that~$\lambda$ is small.

\begin{itemize}
	\item[$(\mathrm{a})$] \textit{Stars.} As explained earlier, for the classical contact process with small~$\lambda$, a vertex with degree much larger than~$1/\lambda^2$ sustains the infection for a long time. In our bipartite setting, we show that regardless of the type of a vertex (and hence of its neighbours), if its degree is much larger than~$1/(\lambda_1\lambda_2)=1/\lambda^{1+a}$, then it sustains the infection for a long time -- hence, such vertices  are considered ``stars'' in our analysis. 
Here and throughout this section, statements such as `much larger' are understood up to multiplicative factors that are powers of~$\log \tfrac{1}{\lambda}$.
	\item[$(\mathrm{b})$] \textit{Mark-degree conversion.} Given~$u,v \in (0,1)$, in the absence of any other information about the graph, 
		\begin{equation}\begin{split}
			&\text{a type-1 vertex with mark $u$ has, in expectation, }\\
			&u^{-\gamma_1}\cdot \int_0^v t^{-\gamma_2}\;\mathrm{d}t = \frac{u^{-\gamma_1}v^{\bar{\gamma}_2}}{\bar{\gamma}_2}  \text{ neighbours with mark below } v.
		\end{split}
		\label{eq_type1_ne}
		\end{equation}
		
	 In particular, taking~$v=1$,
\begin{equation}\label{eq_conversion}
	\text{a type-1 vertex with mark $u$ has expected degree } \frac{u^{-\gamma_1}}{\bar{\gamma}_2}.
\end{equation}

Symmetric statements hold with types~$1$ and~$2$ reversed.

	\item[$(\mathrm{c})$] \textit{Direct transmission.} A type-1 vertex transmits the infection to a fixed neighbour before recovering with probability~$\frac{\lambda}{1+\lambda}$ (or roughly~$\lambda$, since~$\lambda$ is small). In view of this, and of point~$(\mathrm b)$, if a type-1 vertex with mark~$u$ is infected, then the expected number of neighbours with mark below~$v$ that it infects before recovering is of order~$\lambda u^{-\gamma_1} v^{\bar{\gamma}_2}$. In particular, the lowest mark of a neighbour it will infect directly will likely be of order~$(u^{\gamma_1}/\lambda)^{1/\bar{\gamma}_2}$.

		If we parametrize~$u=\lambda^{\mu}$, where~$\mu > 0$, this says that a vertex with mark~$u$ can hope to infect, by direct transmission, a neighbour with mark as low as~$v=\lambda^{\varphi(\mu)}$, where
\begin{equation}\label{eq_varphi_first}
	\varphi(\mu):=\frac{\gamma_1}{\bar{\gamma}_2} \mu - \frac{1}{\bar{\gamma}_2},
\end{equation}
but not much lower. By similar considerations, a type-2 vertex of mark~$\lambda^\nu$ can hope to infect (by direct transmission) a neighbour of mark roughly~$\lambda^{\psi(\nu)}$, where
		\[
			\psi(\nu):= \frac{\gamma_2}{\bar{\gamma}_1} \nu - \frac{1}{\bar{\gamma}_1}a.
		\]
\end{itemize}

Guided by the above considerations, we now describe three values of~$\mu > 0$ (resp. three values of~$\nu > 0$) such that, if the infection reaches a type-1 vertex of mark much lower than~$\lambda^{\mu}$ (resp. a type-2 vertex of mark much lower than~$\lambda^\nu$), then it is very likely to survive.

\begin{itemize}
	\item \textbf{Survival from stars.} Let
		\[\mu_{\mathsf S} := \frac{1}{\gamma_1}+\frac{1}{\gamma_1}a.\]
		
		By~\eqref{eq_conversion}, if a type-1 vertex has mark much lower than~$u=\lambda^{\mu_{\mathsf S}}$, then it is likely to have degree much larger than~$\lambda^{-1-a}$, in which case it is a star. Moreover, by~\eqref{eq_type1_ne}, such a vertex is likely to have a neighbour with mark roughly~$v = u^{\gamma_1/\bar{\gamma}_2}$, which in turn (applying~\eqref{eq_type1_ne} with the types reversed), is likely to have a neighbour with mark roughly~$u' = v^{\gamma_2/\bar{\gamma}_1} = u^{\gamma_1\gamma_2/(\bar{\gamma}_1\bar{\gamma}_2)}$. Using~$\gamma_1+\gamma_2 > 1$ gives~$\gamma_1\gamma_2/(\bar{\gamma}_1\bar{\gamma}_2) > 1$, so~$u' < u$, which implies that~$u'$ is also likely to be a star. This suggests that there is a good chance that an infinite chain of type-1 stars can be formed, each one at distance~$2$ from the previous. By usual arguments involving propagation of the contact process in stars, this would suffice for the infection to survive. The symmetric argument starting from a type-2 vertex would have initial mark~$\lambda^{\nu_{\mathsf S}}$, where
		\[
			\nu_{\mathsf S}:=\frac{1}{\gamma_2}+\frac{1}{\gamma_2}a.
		\]
		
	\item \textbf{Survival from bridges.} Loosely speaking, a \emph{bridge} is a vertex that is likely to reach a star of the other type by direct transmission. In order for a type-1 vertex with mark expressed as~$\lambda^\mu$ (possibly times a logarithmic factor of~$1/\lambda$) to be a bridge, the value of~$\mu$ should be such that~$\varphi(\mu)=\nu_{\mathsf S}$, where~$\varphi$ is as in~\eqref{eq_varphi_first}. This is achieved by~$\mu$ equal to
		\[
\mu_{\mathsf B} := \frac{1}{\gamma_1\gamma_2} + \frac{\bar{\gamma}_2}{\gamma_1\gamma_2}a.
		\]
		
		The corresponding value for type 2, obtained by solving for~$\nu$ in~$\psi(\nu)=\mu_{\mathsf S}$, is
		\[
\nu_{\mathsf B} := \frac{\bar{\gamma}_1}{\gamma_1 \gamma_2} + \frac{1}{\gamma_1 \gamma_2} a.
		\]

	\item \textbf{Survival from direct spreaders.} Point $\mathrm{(c)}$ above roughly expresses the idea that
		\begin{align*}
			&\text{type 1 with mark $\lambda^\mu$ directly infects type 2 with mark~$\lambda^{\varphi(\mu)}$};\\
			&\text{type 2 with mark $\lambda^\nu$ directly infects type 1 with mark~$\lambda^{\psi(\nu)}$}.
		\end{align*}
		
		This suggests that, if~$\mu$ satisfies~$\psi(\varphi(\mu)) \ge \mu$, then a type-1 vertex with mark much smaller than~$\lambda^{\mu}$ will reach a type-1 vertex with lower mark by direct transmission in two steps, which in turn will do the same etc., causing survival of the infection by direct transmissions. We have
		\[
			\psi(\varphi(\mu)) = \frac{\gamma_1\gamma_2}{\bar{\gamma}_1\bar{\gamma}_2} \mu - \frac{\gamma_2}{\bar{\gamma}_1 \bar{\gamma}_2} - \frac{1}{\bar{\gamma}_1} a,
		\]
		so the smallest value of~$\mu$ for which~$\psi(\varphi(\mu))\ge \mu$ is
		\[
\mu_{\mathsf D} := \frac{\gamma_2}{\Delta} + \frac{\bar{\gamma}_2}{\Delta}a.
		\]
		
		The corresponding value for type 2 is the solution of~$\varphi(\psi(\nu))=\nu$, namely
		\[
\nu_{\mathsf D} :=  \frac{\bar{\gamma}_1}{\Delta} + \frac{\gamma_1}{\Delta}a.
		\]
\end{itemize}

The minimiser among~$\mu_{\mathsf S},\mu_{\mathsf B},\mu_{\mathsf D}$ (resp. among $\nu_{\mathsf S},\nu_{\mathsf B},\nu_{\mathsf D}$) depends on~$\gamma_1,\gamma_2,a$ (the answers are given in Proposition~\ref{prop_target_opt} in the \hyperref[appn]{Appendix}, and summarized in Figure~\ref{fig_first} there). We define
\[
	\mu_\star:= \mu_{\mathsf S}\wedge \mu_{\mathsf B}\wedge \mu_{\mathsf D},\qquad \nu_\star:=  \nu_{\mathsf S}\wedge \nu_{\mathsf B}\wedge\nu_{\mathsf D}.
\]

A type-1 vertex with mark lower than roughly~$\lambda^{\mu_\star}$ (resp. a type-2 vertex with mark lower than roughly~$\lambda^{\nu_\star}$) is called a \emph{target}.

Now that we know that infecting a target of either type essentially guarantees survival, we can turn to the question of the most economic way to reach one from the initial infection at the root. Of course, one possibility is that the root is itself a target, by having mark lower than~$\lambda^{\mu_\star}$; this has probability~$\lambda^{\mu_\star}$. In rough terms:
\begin{itemize}
	\item the probability to reach a target in one step from the root by direct transmission is of order~$\lambda^{1+\bar{\gamma}_2\nu_{\star}}$ (neglecting~$\log$ factors);
	\item the probability to reach a target in two steps from the root by direct transmission is of order~$\lambda^{1+a-\Delta_2 \nu_\star+\bar{\gamma}_1\mu_{\star}}$ (neglecting~$\log$ factors).
\end{itemize}

Moreover, we prove in Section~\ref{s_upper_bound} that reaching a target in more steps has probability of lower order.
This leads us to defining
	\begin{equation*}
		A_\star:=\mu_\star \wedge (1+\bar{\gamma}_2 \nu_\star)\wedge (1+a-\Delta_2 \nu_\star+\bar{\gamma}_1 \mu_\star)
	\end{equation*}
for the exponent of the optimal strategy.

\subsection{Main result}
We are now ready to state our main result. The statement includes a phase diagram with many regions. Figure~\ref{fig_opt} gives a visual aid, and an interactive explorer is available at~\url{https://christian-hirsch.netlify.app/publication/fhv/}.

\begin{theorem}
	\label{thm_main}
	Assume that~$(\gamma_1,\gamma_2) \in \mathcal S$ and~$a > 0$. 	
Let
\begin{align*}
	&\mu_{\mathsf S} := \frac{1}{\gamma_1}+\frac{1}{\gamma_1}a,\qquad \mu_{\mathsf B} := \frac{1}{\gamma_1\gamma_2} + \frac{\bar{\gamma}_2}{\gamma_1\gamma_2}a,\qquad \mu_{\mathsf D} := \frac{\gamma_2}{\Delta} + \frac{\bar{\gamma}_2}{\Delta}a,\\
	&\nu_{\mathsf S} := \frac{1}{\gamma_2} + \frac{1}{\gamma_2}a,\qquad \nu_{\mathsf B} := \frac{\bar{\gamma}_1}{\gamma_1 \gamma_2} + \frac{1}{\gamma_1 \gamma_2} a,\qquad \nu_{\mathsf D} :=  \frac{\bar{\gamma}_1}{\Delta} + \frac{\gamma_1}{\Delta}a.
\end{align*}

Also define
	\begin{equation}
	\mu_\star:= \mu_{\mathsf S}\wedge \mu_{\mathsf B}\wedge \mu_{\mathsf D},\qquad \nu_\star:=  \nu_{\mathsf S}\wedge \nu_{\mathsf B}\wedge\nu_{\mathsf D}
	\end{equation}
	and
	\begin{equation}\label{eq_def_A_star}
		A_\star:=\mu_\star \wedge (1+\bar{\gamma}_2 \nu_\star)\wedge (1+a-\Delta_2 \nu_\star+\bar{\gamma}_1 \mu_\star).
	\end{equation}
	
	Then, there exists~$C > 0$ such that, for~$\lambda$ small enough, the survival probability~$\theta(\lambda)$ of the contact process with infection rates~$(\lambda_1,\lambda_2)=(\lambda,\lambda^a)$ on~$\mathcal G$, started from only the root infected, satisfies:
	\[
		(\log \tfrac{1}{\lambda})^{-C} \cdot \lambda^{A_\star} \le \theta(\lambda) \le (\log \tfrac{1}{\lambda})^{C} \cdot \lambda^{A_\star}.
	\]
	
	Moreover,~$A_\star$ is given explicitly by the following:
	\begin{itemize}
	\item[$\mathrm{(Ia)}$] $\gamma_1 > \frac12$ and~$\gamma_2 \le \frac{\gamma_1}{1+\gamma_1} $, then
		\[
			A_\star = \mu_{\mathsf S}. 
		\]
	\item[$\mathrm{(Ib)}$] If~$\gamma_1 > \frac12$ and~$\frac{\gamma_1}{1+\gamma_1} < \gamma_2 \le \frac12$, then
		\[
			A_\star = \begin{cases} \mu_{\mathsf S} & \text{if }  a \le \frac{\gamma_1-\gamma_2}{\gamma_2-\gamma_1+\gamma_1\gamma_2};\\[.2cm]
				1+\bar{\gamma}_2 \nu_{\mathsf S} & \text{if } a >\frac{\gamma_1-\gamma_2}{\gamma_2-\gamma_1+\gamma_1\gamma_2}. 
			\end{cases}
		\]
	\item[$\mathrm{(II)}$] If~$\gamma_1 \le \frac12$ (and consequently~$\gamma_2 > \frac12$), we have
	\[
			A_\star = 1 + \bar{\gamma}_2 \nu_{\mathsf S}.
		\]
	\item[$\mathrm{(III)}$] If~$\gamma_1,\gamma_2 > \frac12$ and~$\frac{1}{\gamma_1}+\frac{1}{\gamma_2} > 3$, then
		\[
			A_\star = \begin{cases}
				1+\bar{\gamma}_2 \nu_{\mathsf S}&\text{if } a \le \frac{\bar{\gamma}_2}{\Delta_2};\\[.2cm]
				1+\bar{\gamma}_2\nu_{\mathsf B}&\text{if } a > \frac{\bar{\gamma}_2}{\Delta_2}.
			\end{cases}
		\]
	\item[$\mathrm{(IV)}$] If~$\frac{1}{\gamma_1}+\frac{1}{\gamma_2} \le 3$  (and consequently~$\gamma_1, \gamma_2 > \frac12$), then
		\[
			A_\star = \begin{cases}
				1+\bar{\gamma}_2 \nu_{\mathsf S}&\text{if } a \le \frac{\bar{\gamma}_1 \bar{\gamma}_2}{\gamma_1\gamma_2 + \gamma_2 - 1};\\[.2cm]
				1+\bar{\gamma}_2\nu_{\mathsf D}&\text{if } \frac{\bar{\gamma}_1 \bar{\gamma}_2}{\gamma_1\gamma_2 + \gamma_2 - 1} < a \le \frac{\gamma_1\gamma_2 + \gamma_1 - 1}{\bar{\gamma}_1 \bar{\gamma}_2};\\[.2cm]
				1+\bar{\gamma}_2\nu_{\mathsf B}&\text{if } a > \frac{\gamma_1\gamma_2 + \gamma_1 - 1}{\bar{\gamma}_1 \bar{\gamma}_2}.
			\end{cases}
		\]
	\end{itemize}
\end{theorem}

\begin{figure}[H]
	\begin{center}
		\fbox{\includegraphics[width = .8\textwidth]{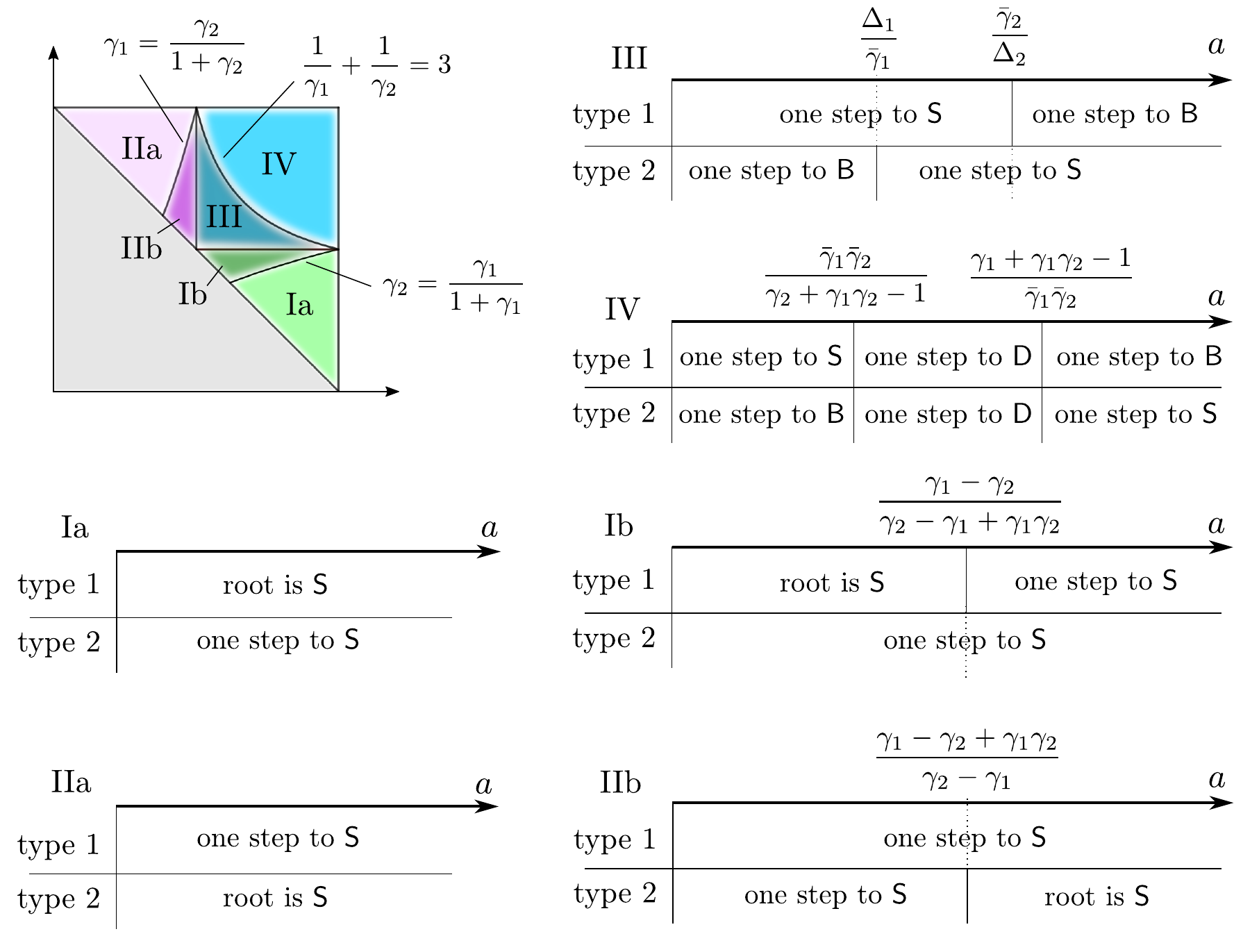}}
		\caption{Strategies that give the exponent in the survival probability for different values of~$(\gamma_1,\gamma_2,a)$. We include both the case where the root is of type~1 (which is given in the statement of Theorem~\ref{thm_main}) and of type~2. For a root of type~1, ``root is $\mathsf S$'' means that the exponent is~$\mu_{\mathsf S}$; ``one step to $\mathsf S$'' means that it is~$1+\bar{\gamma}_2 \nu_{\mathsf S}$; ``one step to $\mathsf B$'' means that is it~$1+\bar{\gamma}_2 \nu_{\mathsf B}$;  ``one step to $\mathsf D$'' means that is it~$1+\bar{\gamma}_2 \nu_{\mathsf D}$.
}\label{fig_opt}
	\end{center}
\end{figure}

\begin{remark}[Implied additional results]
	Analogous formulas are of course available for the case where the root vertex holding the initial infection is of type~2 rather than type~1. We do not include these in the statement of Theorem~\ref{thm_main} for brevity, but the optimal strategy that gives the correct exponent is shown in Figure~\ref{fig_opt}.
	
	Moreover for case where the two types have the same infection rate, i.e. when setting $a=1$, the theorem still says something interesting. Then we have a two-dimensional phase diagram---so it's easy to sketch, as we do in Figure \ref{fig_one_type}.
\end{remark}

\begin{figure}[H]
	\begin{center}
		\fbox{\includegraphics[width = .6\textwidth]{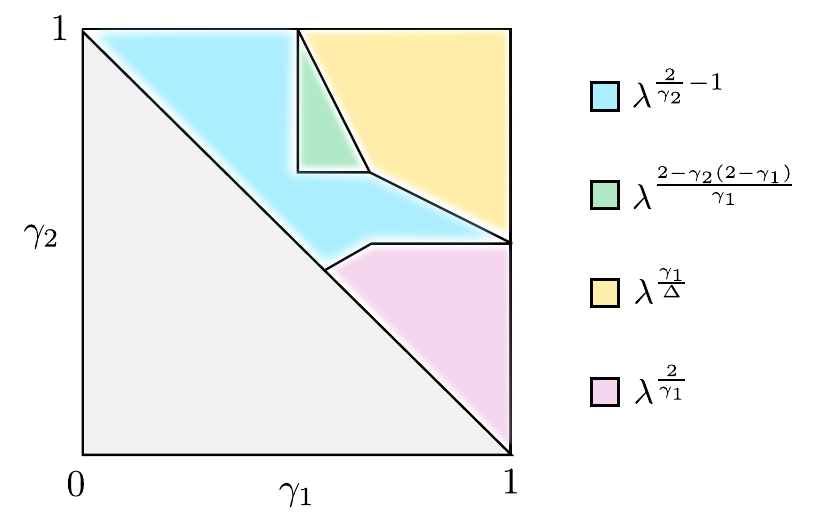}}
		\caption{Setting $a=1$, Theorem~\ref{thm_main} gives a result for the survival probability of the classical contact process on our bipartite spatial model (still of course neglecting polylogarithmic factors). Each of the four depicted regions corresponds to a dominating strategy, and in the above order these are: ``one step to $\mathsf S$'', ``one step to $\mathsf B$''; ``one step to $\mathsf D$''; ``root is $\mathsf S$''. 
}\label{fig_one_type}
	\end{center}
\end{figure}

\begin{remark}[Polylogarithmic corrections]
	In several of the aforementioned works that compute exponents of survival probabilities of the contact process on power law random graphs (configuration model in~\cite{mountford2013metastable,van2015metastability}, preferential attachment graph in~\cite{vanhaocan17}, random hyperbolic graph in~\cite{link}, geometric random graphs~\cite{grac}, dynamic graphs in~\cite{jacob2024metastability,jacob2019metastability}), the expression for the survival probability,  at least in some parameter regimes, is given in the more precise form~$\lambda^\alpha  (\log \tfrac{1}{\lambda})^\beta$ (up to a multiplicative constant not depending on~$\lambda$). We believe that establishing this level of precision would be straightforward in some of the regimes in our phase diagram, but significantly involved in others, potentially revealing even more parameter regimes. Given that the phase diagram without attention to the logarithmic term is already quite intricate, we have opted to leave further refinements for future work.
\end{remark}

Theorem~\ref{thm_main} follows immediately from the following three propositions:
\begin{proposition}[Exponent optimisation] \label{prop_opt}
	For~$(\gamma_1,\gamma_2) \in \mathcal S$ and~$a > 0$, the function~$A_\star$ defined by~\eqref{eq_def_A_star} is given explicitly as described in the statement of Theorem~\ref{thm_main}.
\end{proposition}

\begin{proposition}[Lower bound]\label{prop_lower_bound}
	For all~$(\gamma_1,\gamma_2) \in \mathcal S$ and~$a > 0$, there exists~$C > 0$ such that, if~$\lambda$ is small enough,
	\[
		\theta(\lambda) \ge  (\log \tfrac{1}{\lambda})^{-C} \cdot \lambda^{A_\star}.
	\]
\end{proposition}

\begin{proposition}[Upper bound]\label{prop_upper_bound}
	For all~$(\gamma_1,\gamma_2) \in \mathcal S$ and~$a > 0$, there exists~$C > 0$ such that, if~$\lambda$ is small enough,
	\[
		\theta(\lambda) \le  (\log \tfrac{1}{\lambda})^C \cdot \lambda^{A_\star}.
	\]
\end{proposition}

Proposition~\ref{prop_opt} is proved in the \hyperref[appn]{Appendix}, Proposition~\ref{prop_lower_bound} is proved in Section~\ref{s_lower_bound}, and Proposition~\ref{prop_upper_bound} is proved in Section~\ref{s_upper_bound}. We now say a few words about each of their proofs.

Proposition~\ref{prop_opt} is obtained by solving an optimisation problem using elementary Calculus arguments. 
Proposition~\ref{prop_lower_bound} is proved by giving rigorous versions of the strategies for infection survival that we have roughly outlined above.  We prove survival from the different kinds of target (star, bridge, direct spreader) by controlling the presence of vertices in large space-time boxes and infection events. For the star and bridge cases, a necessary first step is giving a lower bound on the extinction time of the infection on star graphs with infection rates~$\lambda_1$ towards the leaves and~$\lambda_2$ towards the centre.

For Proposition~\ref{prop_upper_bound}, the bulk of the work goes into proving that we can neglect infection paths that reach a target in more than two steps. One of our main contributions is a method to do this that is applicable simultaneously to all our numerous parameter regimes. It builds on the idea of classifying infection paths by their \emph{ordered trace}, as in~\cite{link}, and on the martingale bounds for infection paths in~\cite{mountford2013metastable}.
 
\subsection{Graphical construction of the contact process}
We briefly define the graphical construction of the contact process (paying attention to our slightly unusual setting with two different infection rates).

Let~$G=(V,E)$ be a bipartite graph, with~$V$ being split into disjoint sets~$V_1$ and~$V_2$, and each edge having a vertex of each. We take independent families of Poisson processes on~$[0,\infty)$, as follows:
\begin{itemize}
	\item \textit{recovery marks:} for each~$x \in V$, we take~$\mathcal R^x$ with rate~$1$;
	\item \textit{1-to-2 transmission marks:} for each~$(x,y)$ such that~$x \in V_1$,~$y \in V_2$ and~$\{x,y\} \in E$, we take~$\mathcal T^{(x,y)}$ with rate~$\lambda_1$;
	\item \textit{2-to-1 transmission marks:} for each~$(x,y)$ such that~$x \in V_2$,~$y \in V_1$ and~$\{x,y\} \in E$, we take~$\mathcal T^{(x,y)}$ with rate~$\lambda_2$.
\end{itemize}
We then define an \emph{infection path} in the usual way: it is a c\`adl\`ag function~$g:[s,t] \to V$ (where~$0 \le s < t$) that does not touch any recovery marks (that is: $r \notin \mathcal R^{g(r)}$ for all~$r$), and only jumps by traversing transmission arrows (that is:~$g(r)\neq g(r-)$ implies~$r \in \mathcal T^{(g(r-),g(r))}$). Given an initial configuration~$\xi_0 \in \{0,1\}^V$, we can then construct the contact process by setting
\[
	\xi_t(x) = \mathds{1}\left\{\begin{array}{l}\text{there is $y \in V$ with $\xi_0(y)=1$ and}\\ \text{an infection path~$g:[0,t]\to V$ with~$g(0)=y,\; g(t)=x$}\end{array}\right\}, \qquad x \in V,\; t > 0,
\]
where~$\mathds{1}$ denotes the indicator function.

\section{Lower bound}
\label{s_lower_bound}

To lower bound the survival probability, we construct an event in which there is an infinite infection path. In most cases these paths will use star subgraphs to sustain the infection, so we first analyse the bipartite contact process around a star.

\subsection{Extinction time on stars}
In what follows,~$S_n$ is a star graph with vertex set $\{o,x_1,\ldots,x_n\}$ and edge set~$\{\{o,x_1\},\ldots,\{o,x_n\}\}$. The vertex~$o$ is called the \emph{centre}.

\begin{lemma}\label{lem_first_starr}
	There exists~$c > 0$ such that the following holds. Let~$\lambda_1,\lambda_2 \in (0,1]$ and~$n \ge m \ge 2e/\lambda_2$. Let~$A$ be a set of leaves of~$S_n$ with~$|A|=m$. Let~$(\xi_t)_{t \ge 0}$ be the contact process on~$S_n$ where the centre has infection rate~$\lambda_1$ and the leaves have infection rate~$\lambda_2$, started from~$A$ infected. Then,
	\[
		\mathbb P \left(|\xi_1| \ge \lambda_1 n/(8e) \right) > 1 - e^{-cm} - e^{-c\lambda_2m} - e^{-c\lambda_1n}.
	\]
\end{lemma}
\begin{proof}
	The proof is a line-by-line repetition of the proof of Lemma~2.20 in~\cite{valesin2024survey}, except that each occurrence of~$\lambda$ in that proof has to be replaced by either~$\lambda_1$ or~$\lambda_2$ here. 
Specifically, $\lambda_1$ appears in the bound for the probability of the event $E_1$,
$\lambda_2$ in the bound for the probability of $E_2$, and $\lambda_1$ again in the
definition of $E_3$ and in the corresponding probability bound.
\end{proof}

The next lemma isolates the mechanism behind survival by repeated transfer between well-connected type-1 vertices.

%
%
\begin{lemma}\label{lem_previous_star}
	There exists~$C > 0$ such that the following holds. Let~$\lambda_1,\lambda_2 \in (0,1]$ and~$n > \frac{C}{\lambda_1\lambda_2}$. Let~$(\xi_t)_{t \ge 0}$ be the contact process on~$S_n$ where the centre has infection rate~$\lambda_1$ and the leaves have infection rate~$\lambda_2$, started from only the centre infected. Then,
	\[
		\mathbb P \left( \sup_{ t > 0} |\xi_t| \ge \frac{\lambda_1 n}{8e}\right) \ge 1-\frac{3}{\sqrt{\lambda_1\lambda_2 n}}.
	\]
\end{lemma}
\begin{proof}
	Let~$t_0:=1/\sqrt{\lambda_1\lambda_2 n}$. Let~$E_1$ be the event that the centre has no recovery by time~$t_0$, so that
	\[\mathbb P(E_1) = e^{-t_0} > 1 - t_0 =1 - \frac{1}{\sqrt{\lambda_1\lambda_2 n}}.\]

	Let~$p:=(1-e^{-\lambda_1 t_0})e^{-t_0}$. Since we can take~$n$ much larger than~$1/(\lambda_1\lambda_2)$, we can bound~$e^{-t_0} > 1/2$ and~$1-e^{-\lambda_1 t_0} \ge \lambda_1 t_0/2$, so
	\begin{equation}\label{eq_bound_np} p \ge \tfrac14\lambda_1t_0= \tfrac14 \sqrt{\lambda_1/(\lambda_2n)}\qquad \text{and}\qquad np \ge \tfrac14 \sqrt{\lambda_1 n /\lambda_2}.\end{equation}

	Let~$E_2$ be the event that~$E_1$ occurs, and there are at least~$np/2$ leaves that receive a transmission from the centre before~$t_0$, and have no recovery mark before~$t_0$.
 Then, a Chernoff bound gives
	\begin{align*}
		\mathbb P(E_2 \mid E_1) = \mathbb P\left( \mathrm{Bin}(n,p) > np/2  \right) \ge 1 - e^{-np/8}\stackrel{\eqref{eq_bound_np}}{\ge}  1- e^{-\frac{1}{32}\sqrt{\lambda_1 n /\lambda_2}} > 1-\frac{1}{\sqrt{\lambda_1\lambda_2n}},
	\end{align*}
	where the last inequality follows from~$\lambda_2 \le 1$ and~$n \gg \frac{1}{\lambda_1\lambda_2}$.

	Now, let~$E_3$ be the event that~$E_2$ occurs, and at time~$t_0 + 1$, there are more than~$\lambda_1 n/(8e)$ infected vertices. Noting that~$np/2 > \frac18 \sqrt{\lambda_1 n /\lambda_2} \gg 1/\lambda_2$ (again by~$n \gg 1/(\lambda_1\lambda_2)$),  we can use Lemma~\ref{lem_first_starr}, which gives
	\begin{equation}\begin{split}
		\mathbb P(E_3 \mid E_2) &> 1 - e^{-np/2} - e^{-c\lambda_2np/2} - e^{-c\lambda_1n}\\
		&\stackrel{\eqref{eq_bound_np}}{\ge} 1 - e^{-\frac18 \sqrt{\lambda_1 n /\lambda_2}} - e^{-\frac{c}{8} \sqrt{\lambda_1 \lambda_2 n}} - e^{-c\lambda_1 n} > 1- \frac{1}{\sqrt{\lambda_1 \lambda_2 n}},
	\end{split}\label{eq_three_terms_bound}
	\end{equation}
completing the proof.
\end{proof}

\begin{lemma}\label{lem_main_star}
	There exists~$\varepsilon > 0$ such that the following holds. Let~$\lambda_1,\lambda_2 \in (0,1]$ and~$n > \frac{1}{\varepsilon\lambda_1\lambda_2}$. Let~$(\xi_t)_{t \ge 0}$ be the contact process on~$S_n$ where the centre has infection rate~$\lambda_1$ and the leaves have infection rate~$\lambda_2$, started from only the centre infected. Then,
	\[
		\mathbb P \left( \xi_{\exp\{\varepsilon \lambda_1 \lambda_2 n \}} \neq \varnothing\right) \ge 1-\frac{4}{\sqrt{\lambda_1\lambda_2 n}}.
	\]
\end{lemma}
\begin{proof}
	Fix~$\varepsilon < 1/C$, where~$C$ is the constant of Lemma~\ref{lem_previous_star}. 
	Letting
	\[\sigma:=\inf\{t: |\xi_t|\ge \lambda_1 n/(8e)\},\]
	Lemma~\ref{lem_previous_star} implies that~$\mathbb P(\sigma < \infty) > 1-3/\sqrt{\lambda_1\lambda_2 n}$. Now, for~$k \in \mathbb N_0$, let
	\[
		E_k := \{\sigma < \infty\} \cap \{|\xi_{\sigma + k + 1}| \ge \lambda_1 n/(8e)\}.
	\]
	
	Noting that~$\lambda_1 n /(8e) \gg 2e/\lambda_2$, we can apply Lemma~\ref{lem_first_starr}, which gives
	\[\mathbb P(E_{k+1} \mid E_k) > 1 - e^{-c\frac{\lambda_1 n}{8e}}-e^{-c \frac{\lambda_1\lambda_2 n}{8e}} - e^{-c\lambda_1 n} > 1- e^{-\varepsilon \lambda_1 \lambda_2 n},\]
	if~$\varepsilon$ is chosen small enough. The desired bound now easily follows.
\end{proof}

\subsection{Lower bounds on survival probability for four strategies}

With the lemmas of the previous section, we can now construct the survival events. In the first result we show how the infection can propagate along a path with an increasing sequence of degrees, and in the subsequent result we will use this to say that the infection survives whenever it infects a vertex beyond the star threshold.

\begin{lemma}[Survival along a sequence of stars]\label{lem_line_and_star}
	Let~$G$ be a bipartite graph. Let~$(x_k)_{k\ge 0}$ and~$(y_k)_{k \ge 1}$ be sequences of vertices of~$G$ such that:
	\begin{itemize}
		\item $(x_k)_{k \ge 0}$ have type 1 and~$(y_k)_{k \ge 1}$ have type~2, with
	\[
x_0 \sim y_1 \sim x_1 \sim y_2 \sim x_2 \sim \cdots;
	\]
\item the sequence~$(x_k)_{k\ge 0}$ has infinitely many distinct elements.
	\end{itemize}
	Then, the contact process on~$G$ with rates~$\lambda_1,\lambda_2$ and~$x_0$ initially infected survives forever with probability at least
	\begin{equation*}
		1-\sum_{k=0}^\infty ((\varepsilon \lambda_1 \lambda_2 \deg(x_k))^{-1/2} + \exp\{-c\lambda_1 \lambda_2^2 \lfloor \exp(\varepsilon \lambda_1 \lambda_2 \cdot \deg(x_k)) \rfloor \}),
	\end{equation*}
	where~$\varepsilon > 0$ is the constant of Lemma~\ref{lem_main_star} and~$c > 0$ is a constant not depending on~$\lambda_1,\lambda_2$.
\end{lemma}
\begin{proof}
	Let~$(\xi_t)_t$ denote the contact process on~$G$ with rates~$\lambda_1, \lambda_2$ started from~$x_0$ infected. Define~$\tau_0:=0$ and, recursively,
	\[
		\tau_{k+1}:= \begin{cases} \inf\{t \ge \tau_k:\; x_{k+1} \text{ is infected at time }t\}&\text{if } \tau_k<\infty;\\ \infty&\text{otherwise}.\end{cases}
	\]
	
	Since~$(x_k)_k$  has infinitely many distinct elements, we have
	\[
		\mathbb P((\xi_t)_t \text{ survives forever}) \ge \mathbb P(\tau_k < \infty \text{ for all } k) \ge 1 - \sum_{k=0}^\infty \mathbb P(\tau_{k+1}=\infty \mid \tau_k < \infty).
	\]
	
	Define
\[
	t_k := \exp\left\{\varepsilon \lambda_1 \lambda_2 \cdot  \deg(x_k) \right\},\qquad k \ge 0.
\]

Letting~$B(x_k,1)$ be the star graph induced by~$x_k$ and its neighbours, we bound
\begin{align}
	\nonumber&\mathbb P(\tau_{k+1} = \infty \mid \tau_k < \infty) \le \mathbb P(\tau_{k+1} > \tau_k+t_k \mid \tau_k < \infty)\\[.2cm]
	\label{eq_new_star_lemma1}&\le \mathbb P(\xi_s \cap B(x_k,1) = \varnothing \text{ for some } s \in [\tau_k,\tau_k+t_k] \mid \tau_k < \infty)\\
	\label{eq_new_star_lemma2}&\quad + \mathbb P(\xi_s \cap B(x_k,1) \neq \varnothing  \text{ for all } s \in [\tau_k,\tau_k+t_k], \; \tau_{k+1}> \tau_k + t_k \mid \tau_k < \infty).
\end{align}

	By Lemma~\ref{lem_main_star}, the expression in~\eqref{eq_new_star_lemma1} is smaller than~$(\varepsilon \lambda_1 \lambda_2 \deg(x_k))^{-1/2}$.

	Next, note that if at some time~$t$ the infection is present somewhere in~$B(x_k,1)$, then with probability at least~$c\lambda_1 \lambda_2^2$ (for some universal constant~$c$), it reaches~$x_{k+1}$ by time~$t+1$. To see this, we first consider the case where at time~$t$, $x_k$ and~$y_{k + 1}$ are healthy, but some other neighbour of~$x_k$ is infected. Then, a factor~$\lambda_2$ is included for the infection to reach~$x_k$, then a factor~$\lambda_1$ for it to reach~$y_{k + 1}$, and finally another~$\lambda_1$ to reach~$x_{k+1}$ (the constant~$c>0$ being present to take care of recovery marks and the constraint that everything must happen in one unit of time). Other possible cases are when either~$x_k$ or~$y_{k + 1}$ are infected at time~$t$; in such situations, the expression is either~$c\lambda_2$ or~$c \lambda_1 \lambda_2$.

	Using this consideration, the expression in~\eqref{eq_new_star_lemma2} is smaller than
\begin{align*}
	& \mathbb P(\xi_s \cap B(x_k,1) \neq \varnothing  \text{ for all } s \in [\tau_k,\tau_k+t_k-1], \; \tau_{k+1}> \tau_k + t_k \mid \tau_k < \infty)\\
	&\le (1-c \lambda_1 \lambda_2^2) \cdot  \mathbb P\left(\left. \begin{array}{r}\xi_s \cap B(x_k,1) \neq \varnothing  \text{ for all } s \in [\tau_k,\tau_k+t_k-1], \\ \tau_{k+1}> \tau_k + t_k - 1 \end{array}\;\right|\;\tau_k < \infty\right).
\end{align*}

Iterating this, the right-hand side is smaller than
\begin{equation*}
	(1-c\lambda_1 \lambda_2^2)^{\lfloor t_k \rfloor} \le \exp\{-c\lambda_1 \lambda_2^2 \lfloor t_k \rfloor\} = \exp\{-c\lambda_1 \lambda_2^2 \lfloor \exp\left\{\varepsilon \lambda_1 \lambda_2\cdot  \deg(x_k) \right\} \rfloor\}. \qedhere
\end{equation*}

\end{proof}

Now we use this lemma to find survival from the star threshold. To be specific, define
\begin{equation}\label{eq_def_u0}
u_{0} := \lambda^{\mu_{\mathsf S}} \cdot  (\log \tfrac{1}{\lambda})^{-2/\gamma_1}.\end{equation}

	We always assume that~$\lambda$ is small enough that~$u_0 < 1/4$.
	Given~$h \in [u_0,2u_0]$, let~${\mathcal{G}}^+_{1}(h)$ be the random graph defined as follows:
\begin{itemize}
	\item we take the random bipartite graph~$\mathcal G$ the same way as before, except that we place the (type-1) root deterministically at~$(0,h)$;
	\item we remove all vertices from~$(-\infty,0) \times [0,1]$, and all type-1 vertices from~$\mathbb R \times [u_0,2u_0]$, except for the root.
\end{itemize}

\begin{proposition}[Survival from type-1 root below star threshold]\label{prop_all_lower_bound_new}
	Assume that~$(\gamma_1,\gamma_2) \in \mathcal S$ and~$a > 0$. If~$\lambda$ is small enough, then for all~$h \in [u_0,2u_0]$, the contact process on~${\mathcal G}^+_1(h)$ with rates~$(\lambda_1,\lambda_2)=(\lambda,\lambda^a)$ started with the root infected  survives with probability at least~$1/2$.
\end{proposition}
\begin{proof}
	The proof is an application of Lemma~\ref{lem_line_and_star}. Most of the work will go into finding sequences of vertices with the properties required in that lemma, and so that the degrees of the type-1 vertices are large, so that the bound provided by the lemma is a good one.\\[-.2cm]

	\textit{Graph events.}  We begin by introducing some constants and spatial regions.
	It follows from the assumption~$\gamma_1+\gamma_2 > 1$ that~$\frac{\bar{\gamma}_1}{\gamma_2} < \frac{\gamma_1}{\bar{\gamma}_2}$. We then choose some arbitrary~$\beta \in (\frac{\bar{\gamma}_1}{\gamma_2}, \frac{\gamma_1}{\bar{\gamma}_2})$ and define
	\[
		v_0 := u_0^{\beta}.
	\]
	
	We then also let
	\[
		u_k := 2^{-k}u_0,\qquad v_k := 2^{-k}v_0,\qquad \ell_k := 2^{(k-1)(\gamma_1+\gamma_2)}\cdot u_0^{-\gamma_1} v_0^{-\gamma_2},\qquad k \in \mathbb N.
	\]
	
	Next, define the rectangles
	\[
		R_k := [0,\ell_k] \times [0,u_k],\qquad R_k' := [0,\ell_k] \times [0,v_k],\qquad k \in \mathbb N
	\]
	and the event
	\[
		E_1 := \bigcap_{k \ge 1} \left\{ \mathcal{P}_1 \cap R_k \neq \varnothing,\; \mathcal{P}_2 \cap R_k' \neq \varnothing\right\}.
	\]
	
If~$E_1$ occurs, we can choose a sequence of vertices~$(x_k, y_k)_{k \ge 1}$ with
	\begin{equation}\label{eq_prop_xkyk1} x_k \in \mathcal P_1 \cap R_k,\qquad y_k \in \mathcal P_2 \cap R_k',\qquad  k\ge 1.
	\end{equation}
	
	Note that both sequences may contain repetitions, but the facts that~$x_k \in R_k$ and~$y_k \in R_k'$ for all~$k$ imply that both~$(x_k)_k$ and~$(y_k)_k$ have infinitely many distinct terms.  Inspecting the dimensions of the boxes and the definition of~$u_k,v_k,\ell_k$, it is easy to check that 
	\begin{equation}\label{eq_prop_xkyk2} 
		(0,h) =: x_0 \sim y_1 \sim x_1 \sim y_2 \sim x_2 \sim \cdots.
	\end{equation}
	
Now define the additional rectangles
	\[
		r_0 =  [0,(2u_0)^{-\gamma_1}]\times [\tfrac12, 1]
	\]
	and
	\[
		r_{k,j}:= [\tfrac12 j u_k^{-\gamma_1}, \;\tfrac12 (j+1) u_k^{-\gamma_1}] \times [\tfrac12, 1], \qquad k \in \mathbb N,\;j \in \{0, 1, \ldots,\left\lceil \tfrac12 \ell_k u_k^{-\gamma_1}\right\rceil  \},
	\]
	and define the event
	\[
		E_2:= \{|\mathcal P_2 \cap r_0| \ge \tfrac12\mathrm{Leb}(r_0)\} \cap  \bigcap_{k,j}\{|\mathcal P_2 \cap r_{k,j}| \ge \tfrac12 \mathrm{Leb}(r_{k,j}) \},
	\]
	where~$\mathrm{Leb}$ denotes the Lebesgue measure.

	Assume~$E_1$ and~$E_2$ occur, and fix a sequence~$(x_k,y_k)_k$ satisfying~\eqref{eq_prop_xkyk1} (hence also~\eqref{eq_prop_xkyk2}). Note that all vertices of~$\mathcal P_2 \cap r_0$ are neighbours of~$(0,h)$. Moreover, for each~$k \ge 1$ we can find~$j$ such that all vertices in~$\mathcal P_2 \cap r_{k,j}$ are neighbours of~$x_k$.  This implies that
	\begin{equation}\label{eq_prop_xkyk3}
			\deg((0,h))  \ge \tfrac12 \mathrm{Leb}(r_0) = \tfrac14(2u_0)^{-\gamma_1}\quad \text{and} \quad \deg(x_k) \ge \tfrac12 \mathrm{Leb}(r_{k,j}) = \tfrac18 u_k^{-\gamma_1}.
	\end{equation}
	
	 Let us give upper bounds to~$\mathbb P(E_1^c)$ and~$\mathbb P(E_2^c)$. 
We bound
	\begin{align}
		\nonumber\mathbb P(E_1^c) &\le \sum_k \left(e^{-u_k \ell_k } + e^{-v_k \ell_k } \right) \\
		&= \sum_k \left(\exp\{-2^{\Delta \cdot k + \gamma_1 + \gamma_2} \cdot u_0^{\bar{\gamma}_1} \cdot v_0^{-\gamma_2}\} + \exp\{-2^{\Delta \cdot k + \gamma_1 + \gamma_2} \cdot u_0^{-{\gamma}_1} \cdot v_0^{\bar{\gamma}_2} \} \right). \label{eq_exponentials}
	\end{align}
	We have
	\begin{align*}
		&u_0^{\bar{\gamma}_1} \cdot v_0^{-\gamma_2} = u_0^{\bar{\gamma}_1 - \gamma_2 \beta}, \qquad u_0^{-{\gamma}_1} \cdot v_0^{\bar{\gamma}_2} = u_0^{-{\gamma}_1 + \bar{\gamma}_2 \beta} .
	\end{align*}
	
	The choice of~$\beta$ implies that~$\bar{\gamma}_1 -  \gamma_2 \beta < 0$ and~$-\gamma_1 + \bar{\gamma}_2 \beta < 0$, which implies that both~$u_0^{\bar{\gamma}_1} \cdot v_0^{-\gamma_2}$ and~$u_0^{-{\gamma}_1} \cdot v_0^{\bar{\gamma}_2}$ tend to~$\infty$ as~$\lambda \to 0$, which in turn implies that the expression in~\eqref{eq_exponentials} can be made as small as desired by taking~$\lambda$ small.

Next, using the Chernoff bound~$\mathbb P(\mathrm{Poi}(\lambda) \le \lambda / 2) \le \exp\{-\tfrac12(1-\log 2) \lambda\}$ and a union bound, we have
	\begin{align*}
		\mathbb P(E_2^c) &\le \exp\{-\tfrac12 (1-\log 2)\cdot \mathrm{Leb}(r_0)\} + \sum_{k,j} \exp\{-\tfrac12 (1-\log 2)\cdot \mathrm{Leb}(r_{k,j}) \}\\
		&= \exp\{-\tfrac12 (1-\log 2)\cdot \tfrac12(2u_0)^{-\gamma_1}\} + \sum_{k} \left\lceil \tfrac12 \ell_k u_k^{-\gamma_1}\right\rceil \cdot  \exp\{-\tfrac12 (1-\log 2)\cdot \tfrac14 u_k^{-\gamma_1} \}.
	\end{align*}
	Using the definition of~$u_0$,~$u_k$ and~$\ell_k$, this can be made arbitrarily small by taking~$\lambda \to 0$.\\[-.2cm]

	\textit{Infection survival probability.} 
	Assume that~$E_1$ and~$E_2$ occur. By Lemma~\ref{lem_line_and_star}, the probability that the infection from~$x_0$ lives forever is then at least
\begin{equation*}
	1- \sum_{k=0}^\infty ( (\varepsilon \lambda^{1+a}\cdot  \tfrac18 u_k^{-\gamma_1})^{-1/2} +  \exp\left\{-c\lambda^{1+2a} \left\lfloor \exp\left\{\varepsilon \lambda^{1+a}\cdot \tfrac18 u_k^{-\gamma_1}  \right\} \right\rfloor\right\}).
\end{equation*}

We now bound, for each~$k$,
\[
	(\varepsilon \lambda^{1+a}\cdot  \tfrac18 u_k^{-\gamma_1})^{-1/2} \le {(8/\varepsilon)^{1/2}}\cdot 2^{-\tfrac{\gamma_1}{2} k} \cdot (\log\tfrac{1}{\lambda})^{-1}
\]
and
\[
\exp\left\{-c\lambda^{1+2a} \left\lfloor \exp\left\{\varepsilon \lambda^{1+a}\cdot \tfrac18 u_k^{-\gamma_1}  \right\} \right\rfloor\right\}) \le \exp\left\{-c\lambda^{1+2a} \left\lfloor \exp\left\{ \tfrac{\varepsilon}{8} \cdot 2^{\gamma_1 k} \cdot (\log \tfrac{1}{\lambda})^{2}   \right\} \right\rfloor\right\}.
\]

It is now easy to see that the sum
\[
	\sum_{k=0}^\infty ({(8/\varepsilon)^{1/2}}\cdot 2^{-\tfrac{\gamma_1}{2} k} \cdot (\log\tfrac{1}{\lambda})^{-1} + \exp\left\{-c\lambda^{1+2a} \left\lfloor \exp\left\{ \tfrac{\varepsilon}{8} \cdot 2^{\gamma_1 k} \cdot (\log \tfrac{1}{\lambda})^{2}   \right\} \right\rfloor\right\})
\]
can be made as small as desired by taking~$\lambda$ small.
\end{proof}

Finally, the remaining propositions of this section will construct events corresponding to all possible survival strategies. In the main theorem, each of these strategies will produce the lower bound in its own~$(a,\gamma_1,\gamma_2)$ parameter region.

\begin{proposition}[Lower bound with ``root is star'' strategy] \label{prop_strategy_rott_star}
	Assume that~$(\gamma_1,\gamma_2) \in \mathcal S$ and~$a > 0$. If~$\lambda$ is small enough, the contact process on~$\mathcal G$ started from the root infected survives with probability at least~$\tfrac12 \lambda^{\mu_{\mathsf S}} \cdot  (\log \tfrac{1}{\lambda})^{-2/\gamma_1}$.
\end{proposition}
\begin{proof}
	Let~$E$ be the event that the vertical coordinate of the root is in~$[u_0,2u_0]$. Then, we have~$\mathbb P(E) = u_0$, and Proposition~\ref{prop_all_lower_bound_new} directly implies that conditionally on~$E$, the infection survives with probability at least~$1/2$.
\end{proof}

\begin{proposition}[Lower bound with ``one step to bridge'' strategy]\label{prop_strategy_bridge}
	For any~$(\gamma_1,\gamma_2) \in \mathcal S$ and~$a > 0$, there exists~$c > 0$ such that if~$\lambda$ is small enough, the contact process on~$\mathcal G$ started from the root infected survives with probability at least~$c \lambda^{1+\bar{\gamma}_2\nu_{\mathsf B}} \cdot \left(\log \tfrac{1}{\lambda}\right)^{-2\bar{\gamma}_1\bar{\gamma}_2/(\gamma_1\gamma_2)}$.
\end{proposition}
\begin{proof}
	Define
	\[
		v':= 2^{-\frac{\gamma_1+2}{\gamma_2}}\cdot  \lambda^{\nu_{\mathsf B}}\cdot (\log \tfrac{1}{\lambda})^{-2\bar{\gamma}_1/(\gamma_1 \gamma_2)}.
	\]
	
	This choice is made so that~$2^{-\gamma_1-1} \cdot u_0^{\bar{\gamma}_1}\cdot (v')^{-\gamma_2} = 2 \lambda^{-a}$, which will be used in~\eqref{eq_why_choice} below. Also define the rectangle
	\[
		\tilde R := [-(v')^{-\gamma_2}, 0] \times [0,v'].
	\]
	
	Next, define the event
	\begin{align*}
		E_1 := \{\text{the root of $\mathcal G$ has a neighbour in $\tilde R$}\}.
	\end{align*}
	
	On $E_1$, let $y^*$ denote an arbitrarily chosen neighbour of the root in $\widetilde R$.
	Then, let
	\[
		E_2 := \{\text{$y^*$ has at least~$\lambda^{-a}$ neighbours in~$[0,\infty) \times [u_0,2u_0]$} \},
	\]
	with~$u_0$ as in~\eqref{eq_def_u0}.

	We will give lower bounds for~$\mathbb P( E_1 \cap E_2)$ and for~$\mathbb P(\text{infection from root survives} \mid  E_1 \cap E_2)$. We start with the latter.
	Fix~$\mathcal G \in  E_1 \cap E_2$. Then,
	\begin{itemize}
		\item with probability~$\frac{\lambda}{1+\lambda}$, the root infects~$y^*$ before recovering;
		\item on the event that the above item happens, we then consider the sub-event that~$y^*$ stays infected for one unit of time, during which it transmits the infection to one of its neighbours in~$[0,\infty) \times [u_0,2u_0]$; there are more than~$\lambda^{-a}$ of these neighbours being infected each at rate~$\lambda^a$ and so this happens with (conditional) probability at least
			$e^{-1}\cdot (1 - e^{-1} ).$
		\item on the event that both the above happen, we then ask for the infection from the newly-reached element of~$(0,\infty) \times [u_0,2u_0]$ to live forever. Since at this point the graph is still unexplored in~$(0,\infty) \times ([0,1]\backslash [u_0,2u_0])$, this has (conditional) probability at least~$1/2$, by Proposition~\ref{prop_all_lower_bound_new} (after a spatial shift and re-rooting at~$y^*$).
	\end{itemize}
	
	This gives
	\begin{equation} \label{eq_for_bridge_bound1}
		\mathbb P(\text{infection from root survives} \mid  E_1 \cap E_2) \ge  \frac{\lambda(1+e^{-1})}{2e(1+\lambda)}\ge  \frac{\lambda}{4}.
	\end{equation}

	We now turn to estimating the graph event probabilities. We note that~$E_1$ contains the event that~$\mathcal P_2 \cap \tilde R$ is non-empty. Indeed, denoting the root by~$(0,h_o)$ and letting~$(z,h)$ be an arbitrary element of~$\tilde R$, we have
	\[
		|z-0|=|z|\le (v')^{-\gamma_2} \le h^{-\gamma_2} \le h_o^{-\gamma_1}\cdot  h^{-\gamma_2}.
	\]
	
 This gives
	\begin{equation}\label{eq_for_bridge_bound2}
		\mathbb P(E_1) \ge \mathbb P(\mathcal P_2 \cap \tilde R \neq \varnothing) = 1-\exp\{- (v')^{\bar\gamma_2}\} \ge \tfrac12  (v')^{\bar{\gamma}_2}
	\end{equation}
	if~$v'$ is small enough (which is achieved by taking~$\lambda$ small enough).

	To handle~$E_2$, we note that every type-2 vertex in~$\tilde R$ is a neighbour of every type-1 vertex in the region
	\[
		[0, (2u_0)^{-\gamma_1}\cdot (v')^{-\gamma_2} - (v')^{-\gamma_2}] \times [u_0,2u_0].
	\]
	
Let~$X$ be the number of type-1 vertices in the above region. Then~$X \sim \mathrm{Poi}(\kappa)$, where
	\begin{equation}\begin{split} \label{eq_why_choice}
		\kappa := \left((2u_0)^{-\gamma_1}\cdot (v')^{-\gamma_2} - (v')^{-\gamma_2}\right)\cdot u_0 &\ge \tfrac12 \cdot (2u_0)^{-\gamma_1}\cdot (v')^{-\gamma_2} \cdot u_0 \\
		&= 2^{-\gamma_1-1} \cdot u_0^{\bar{\gamma}_1}\cdot (v')^{-\gamma_2} = 2 \lambda^{-a}.
	\end{split}\end{equation}
	
	Hence,
	\begin{equation}\begin{split}\label{eq_for_bridge_bound3}
		\mathbb P(E_2 \mid E_1) \ge \mathbb P(X \ge \lambda^{-a} \mid E_1) &= \mathbb P( X \ge \lambda^{-a}) \\
		&\ge  \mathbb P(\mathrm{Poi}(2\lambda^{-a}) \ge \lambda^{-a}) \ge 1 - 2\lambda^a \ge 1/2,
	\end{split}
	\end{equation}
where the next-to-last step is Chebyshev's inequality.
	Putting together~\eqref{eq_for_bridge_bound1},~\eqref{eq_for_bridge_bound2}, and~\eqref{eq_for_bridge_bound3}, we have
	\[
		\mathbb P(\text{infection from root survives}) \ge \frac{1}{16}\cdot \lambda \cdot (v')^{\bar{\gamma}_2} = c \lambda^{1+\bar{\gamma}_2\nu_{\mathsf B}} \cdot \left(\log \tfrac{1}{\lambda}\right)^{-2\bar{\gamma}_1\bar{\gamma}_2/(\gamma_1\gamma_2)}. \qedhere
	\]
\end{proof}

We now introduce a quantity~$v_0$ and a random graph~$\mathcal G_2^+(h)$ which are analogous to~$u_0$ and~$\mathcal G_1^+(h)$, respectively, except that we take the viewpoint of type~2 rather than of type~1.
We let
\[
	v_0:= \lambda^{\nu_{\mathsf S}} \cdot  (\log \tfrac{1}{\lambda})^{-2/\gamma_2}.
\]

For~$h \in [v_0,2v_0]$, we define the graph~$\mathcal G_2^+(h)$ as follows:
\begin{itemize}
	\item we take the random bipartite graph~$\mathcal G$ the same way as before, except that we place the (type-2) root deterministically at~$(0,h)$;
	\item we remove all vertices from~$(-\infty,0) \times [0,1]$, and all type-2 vertices from~$\mathbb R \times [v_0,2v_0]$, except for the root.
\end{itemize}

We then have the analogue of Proposition~\ref{prop_all_lower_bound_new}:

%
%
\begin{proposition}[Survival from type-2 root below star threshold]\label{prop_all_lower_bound_new2}
	Assume that~$(\gamma_1,\gamma_2) \in \mathcal S$ and~$a > 0$. If~$\lambda$ is small enough, then for all~$h \in [v_0,2v_0]$, the contact process on~${\mathcal G}^+_2(h)$ with rates~$(\lambda_1,\lambda_2)=(\lambda,\lambda^a)$ started with the (type-2) root infected survives with probability at least~$1/2$.
\end{proposition}
We omit the details, since the proof is the same as that of Proposition~\ref{prop_all_lower_bound_new} (or alternatively, it follows from that proposition, by relabelling the infection rates, so that  type~2 infects with rate~$\eta:=\lambda^a$ and type~1 infects with rate~$\eta^b=\lambda$, with~$b:=1/a$, and then reversing the roles of types~1 and~2).

\begin{proposition}[Lower bound with ``one step to star'' strategy]\label{prop_lower_one_to_star}
	For any~$(\gamma_1,\gamma_2) \in \mathcal S$ and~$a > 0$, there exists~$c > 0$ such that if~$\lambda$ is small enough, the contact process on~$\mathcal G$ started from the root infected survives with probability at least~$c \lambda^{1+\bar{\gamma}_2 \nu_{\mathsf S}}\cdot (\log \tfrac{1}{\lambda})^{-2\bar{\gamma}_2/\gamma_2}$.
\end{proposition}
\begin{proof}
	Let~$\tilde R:=[-v_0^{-\gamma_2},0] \times [0,v_0]$.
	Let~$F$ be the event that the root of~$\mathcal G$ has a neighbour in~$\tilde R$. We have
	\[
		\mathbb P(F) \ge \mathbb P(\mathcal P_2 \cap \tilde R \neq \varnothing) = 1-\exp\left\{-v_0^{\bar{\gamma}_2} \right\} \ge  \tfrac12 v_0^{\bar{\gamma}_2} =\tfrac12 \lambda^{\bar{\gamma}_2 \nu_{\mathsf S}}\cdot (\log \tfrac{1}{\lambda})^{-2\bar{\gamma}_2/\gamma_2}
	\]
	if~$\lambda$ is small enough.
	Next, let~$F'$ be the event that~$F$ occurs, and moreover, the root transmits the infection to a neighbour in~$\tilde R$ before recovering. Then,
	\[
		\mathbb P(F') = \tfrac{\lambda}{1+\lambda}\cdot \mathbb P(F) \ge \tfrac{1}{4} \cdot   \lambda^{1+\bar{\gamma}_2 \nu_{\mathsf S}}\cdot (\log \tfrac{1}{\lambda})^{-2\bar{\gamma}_2/\gamma_2}.
	\]
	
	Conditioning on~$F'$ only involves information on~$\mathcal P_2 \cap \tilde R$, and on the graphical construction in the subgraph induced by the root and~$\mathcal P_2 \cap \tilde R$, up to the time when the transmission mentioned in the definition of~$F'$ happens.  We can then use Proposition~\ref{prop_all_lower_bound_new2} (with a spatial shift, and re-rooting at the neighbour that gets infected) gives
	\[
		\mathbb P(\text{infection from root survives}\mid F') \ge 1/2,
	\]
	which completes the proof.
\end{proof}

\begin{proposition}[Lower bound with ``one step to direct spreader'' strategy]\label{prop_strat_direct}
	For any~$(\gamma_1,\gamma_2) \in \mathcal S$ and~$a > 0$, there exists~$c > 0$ such that if~$\lambda$ is small enough, the contact process on~$\mathcal G$ started from the root infected survives with probability at least~$c\lambda^{1+\bar{\gamma}_2 \nu_{\mathsf D}}$.
\end{proposition}
\begin{proof}
	\textit{Graph events.} Define
	\[
		\hat{u}_k:= 2^{-k} \lambda^{\mu_{\mathsf D}},\qquad \hat v_k:= 2^{-k} \lambda^{\nu_{\mathsf D}},\qquad \hat \ell_k:=2^{(\gamma_1 + \gamma_2)(k-1)}\cdot \lambda^{-\gamma_1 \mu_{\mathsf D}} \cdot \lambda^{-\gamma_2 \nu_{\mathsf D}},\qquad k \in \mathbb N_0.
	\]
	
	Then, define the regions
	$
		\mathcal R:= [0,(2\hat v_0)^{-\gamma_2}] \times [\hat v_0,2\hat v_0]
		$
	and
	\[
		\mathcal R_k:=[0,\hat \ell_k] \times [0,\hat u_k],\qquad \mathcal R_k':=[0,\hat \ell_k] \times [0,\hat v_k],\qquad k \in \mathbb N_0.
	\]
	
	Using the definitions of~$\mu_{\mathsf D}$ and~$\nu_{\mathsf D}$, we have
	\[\bar{\gamma}_1 \mu_{\mathsf D} - \gamma_2 \nu_{\mathsf D} = -a \qquad \text{and} \qquad -\gamma_1 \mu_{\mathsf D} + \bar{\gamma}_2 \nu_{\mathsf D} = -1;\]
	this implies that
	\begin{equation}\label{eq_magic_Delta}
		\mathrm{Leb}(\mathcal R_k) = 2^{-(\gamma_1+\gamma_2)+\Delta \cdot k } \cdot \lambda^{-a},\qquad \mathrm{Leb}(\mathcal R_k')= 2^{-(\gamma_1+\gamma_2)+\Delta \cdot k } \cdot \lambda^{-1}.
	\end{equation}
	
	Inspecting the dimensions of the boxes~$\mathcal R_k$ and~$\mathcal R_k'$, it is easy to check that
	\begin{equation}\label{eq_all_neighbours1}
		x \in \mathcal P_2 \cap \mathcal R,\; y \in \mathcal P_1 \cap \mathcal R_0 \quad \Longrightarrow \quad x \sim y
	\end{equation}
	and
	\begin{equation}\label{eq_all_neighbours2}
x \in \mathcal P_1 \cap \mathcal R_k,\;y \in \mathcal P_2 \cap \mathcal R_k',\; z \in \mathcal P_1 \cap \mathcal R_{k+1}\quad \Longrightarrow \quad x \sim y \sim z.
	\end{equation}

	We let~$E_1$ be the event that the root of~$\mathcal G$ has a neighbour in~$\mathcal R$. Also let
	\[
		E_2:=\bigcap_{k=0}^\infty \{|\mathcal P_1 \cap \mathcal R_k| \ge \tfrac12\mathrm{Leb}(\mathcal R_k),\;|\mathcal P_2 \cap \mathcal R_k'| \ge \tfrac12\mathrm{Leb}(\mathcal R_k')\}.
	\]
	
	We have
	\begin{align*}
		\mathbb P(E_1) \ge \mathbb P(\mathcal P_2 \cap \mathcal R \neq \varnothing) &= 1 - \exp\{-\mathrm{Leb}(\mathcal R)\} \\
		&= 1-\exp\{- 2^{-\gamma_2} (\hat v_0)^{\bar{\gamma}_2} \} \ge 2^{-\gamma_2-1} (\hat v_0)^{\bar{\gamma}_2} =2^{-\gamma_2-1} \cdot \lambda^{\bar{\gamma}_2 \nu_{\mathsf D}}.
	\end{align*}
	
	Moreover, using a Chernoff bound,
	\begin{align*}
		&\mathbb P(E_2^c) \le \sum_{k=0}^\infty (\exp\{- \tfrac12 (1-\log 2)\cdot \mathrm{Leb}(\mathcal R_k) \} + \exp\{- \tfrac12 (1-\log 2)\cdot \mathrm{Leb}(\mathcal R_k') \})\\
		&= \sum_{k=0}^\infty (\exp\{- \tfrac12 (1-\log 2)\cdot 2^{-(\gamma_1+\gamma_2) + \Delta \cdot k} \cdot \lambda^{-a} \} + \exp\{- \tfrac12 (1-\log 2)\cdot 2^{-(\gamma_1+\gamma_2) + \Delta \cdot k} \cdot \lambda^{-1}  \}),
	\end{align*}
	where the equality follows from~\eqref{eq_magic_Delta}.
	This shows that as~$\lambda \to 0$,~$\mathbb P(E_2^c)$ tends to zero faster than any power of~$\lambda$. Therefore, for~$\lambda$ small enough we have
	\begin{equation}\label{eq_prob_good_events_dir}
		\mathbb P(E_1 \cap E_2) \ge \mathbb P(E_1) - \mathbb P(E_2^c) \ge \tfrac12 \mathbb P(E_1) \ge 2^{-\gamma_2-2} \cdot \lambda^{\bar{\gamma}_2 \nu_{\mathsf D}}.
	\end{equation}

	\textit{Infection survival probability.} Assume that~$E_1$ and~$E_2$ occur. By the definition of~$E_1$, we can take~$z=(y,h) \in \mathcal R$ neighbouring the root. We make the infection spread as follows:
	\begin{itemize}
		\item The root infects~$z$ before recovering with probability~$\frac{\lambda}{1+\lambda}$. 
		\item Let~$N_0$ denote the number of neighbours of~$z$ in~$\mathcal R_0$. By~\eqref{eq_all_neighbours1} and the definition of~$E_2$, we have~$N_0 \ge |\mathcal P_1 \cap \mathcal R_0| \ge \tfrac12 \mathrm{Leb}(\mathcal R_0)$. With probability above~$\lambda^a N_0/(1+\lambda^a N_0)$,~$z$ transmits the infection to some element of~$\mathcal R_0$ (call it~$x_0$) before recovering.
		\item 
			Let~$N_0'$ denote the number of neighbours of~$x_0$ in~$\mathcal R_0'$. We have~$N_0' \ge |\mathcal P_2 \cap \mathcal R_0'| \ge \tfrac12 \mathrm{Leb}(\mathcal R_0')$. With probability above~$\lambda N_0'/(1+\lambda N_0')$,~$x_0$ transmits the infection to some element of~$\mathcal R_0'$ (call it~$y_0$) before recovering.
		\item 
	Let~$N_1$ denote the number of neighbours of~$y_0$ in~$\mathcal R_1$. We have~$N_1 \ge |\mathcal P_1 \cap \mathcal R_1| \ge \tfrac12 \mathrm{Leb}(\mathcal R_1)$. With probability above~$\lambda^a N_1/(1+\lambda N_1)$,~$y_0$ transmits the infection to some element of~$\mathcal R_1$ before recovering.
	\end{itemize}

Proceeding in this manner, we see that the infection starting from the root survives forever with probability larger than
	\[
		\frac{\lambda}{1+\lambda} \cdot \prod_{k=0}^\infty \frac{\tfrac12 \lambda^a \cdot  \mathrm{Leb}(\mathcal R_k)}{1+\tfrac12 \lambda^a \cdot  \mathrm{Leb}(\mathcal R_k)} \cdot \frac{\tfrac12 \lambda \cdot  \mathrm{Leb}(\mathcal R_k')}{1+\tfrac12 \lambda \cdot  \mathrm{Leb}(\mathcal R_k')}.	\]
		
	Using~\eqref{eq_magic_Delta}, it is straightforward to check that this is larger than~$c \lambda$, for some constant~$c > 0$. Combining this with~\eqref{eq_prob_good_events_dir} gives the desired bound.
\end{proof}

\begin{proof}[Proof of Proposition~\ref{prop_lower_bound}]
	By putting together Propositions~\ref{prop_strategy_rott_star},~\ref{prop_strategy_bridge},~\ref{prop_lower_one_to_star}, and~\ref{prop_strat_direct}, we see that by choosing~$C > 0$ large enough, for small~$\lambda$ we have
	\[
		\theta(\lambda) \ge (\log\tfrac{1}{\lambda})^{-C}\cdot \lambda^{\min\{\mu_{\mathsf S},\; 1+\bar{\gamma}_2 \nu_{\mathsf S},\; 1+\bar{\gamma}_2 \nu_{\mathsf B},\; 1+\bar{\gamma}_2 \nu_{\mathsf D} \}}.
	\]
	
	By Proposition~\ref{prop_opt}, we have~$A_\star \in \{\mu_{\mathsf S},\; 1+\bar{\gamma}_2 \nu_{\mathsf S},\; 1+\bar{\gamma}_2 \nu_{\mathsf B},\; 1+\bar{\gamma}_2 \nu_{\mathsf D}\}$, concluding the proof.
\end{proof}

\section{Upper bound}
\label{s_upper_bound}
In this section, we prove the upper bound. The main point is to show that, up to logarithmic
factors, only three mechanisms contribute to the survival probability: the root is already a
target, the infection reaches a target in one step, or it reaches a target in two steps.

\subsection{Ordered traces and martingale bound}
As in~\cite{link}, we use a martingale argument to bound the probability of individual
infection paths. However, because of the two-type structure, the existing results in the
literature do not apply directly. Since the proof is short and instructive, we include it here.


For a given graph $G$, we let $\mathbf{\Gamma}=\mathbf{\Gamma}(G)$ denote the set of all \emph{vertex paths}, i.e., (possibly finite) sequences of the form $(\Gamma_0, \Gamma_1, \dots)$ where $\Gamma_i \in V$ and $\Gamma_i \sim \Gamma_{i+1}$ for each~$i$. If~$\Gamma=(\Gamma_0,\ldots,\Gamma_\ell)$, then the \emph{length} of~$\Gamma$ is~$\ell$.

Given an infection path $g: I \to V$, its \emph{ordered trace} is the vertex path $\Gamma\in \mathbf{\Gamma}$ obtained by recording the vertices visited by $g$ in the order they are visited (possibly with repetitions). More precisely, if~$g$ is given by
\[
	g(t) = \begin{cases}
		v_0 &\text{for } t \in [t_0,t_1),\\
		v_1 &\text{for } t \in [t_1,t_2),\\
		\ldots\\
		v_\ell&\text{for } t \in [t_\ell,t_{\ell+1}],
	\end{cases}
\]
with~$t_0 < t_1 < \cdots < t_{\ell+1}$ and~$v_i \neq v_{i+1}$ for each~$i$, then the ordered trace of~$g$ is~$\Gamma=(v_0,\ldots,v_{\ell})$. Note that this sequence may have repeated vertices, but they are never consecutive.
 
\begin{lemma}[Martingale bounds on infection paths]
\label{lem:martingale}
Assume that~$G$ is a bipartite graph, and consider the contact process on~$G$ in which the infection rates of types~$1$ and~$2$ are~$\lambda_1$ and~$\lambda_2$, respectively. Assume that~$\lambda_1,\lambda_2 < 1/4$.
	Let $\Gamma=(\Gamma_0,\ldots,\Gamma_\ell)$ be a vertex path in~$G$ with initial vertex~$\Gamma_0$ of type~$1$. Then, the probability that there exist~$t > 0$ and an infection path~$g:[0,t]\to V$ having $\Gamma$ as its ordered trace is at most $(2\lambda_1)^{\lceil|\Gamma|/2\rceil}(2\lambda_2)^{\lfloor|\Gamma|/2\rfloor}$.
\end{lemma}

The symmetric statement with $(2\lambda_1)^{\lfloor|\Gamma|/2\rfloor}(2\lambda_2)^{\lceil|\Gamma|/2\rceil}$ is of course also true with initial vertex of type~$2$.

\begin{proof}
This martingale argument is similar to the one introduced in \cite[Lemma 5.1]{link}, but with alternating weights $2\lambda_1$ and $2\lambda_2$ along the path. 

	Fix a vertex path~$\Gamma=(\Gamma_0,\ldots, \Gamma_\ell)$. For each~$t \ge 0$, let~$X_t$ denote the largest index~$i \in \{0,\ldots,\ell\}$ such that there exists an infection path~$g:[0,t]\to V$ such that~$g(0)=\Gamma_0$ and~$g(t)=\Gamma_i$ (with~$X_t=-\infty$ if no such~$i$ exists).

	Using that an even length path must end at type $1$, and an odd length path at type $2$, we construct the process
\[
	M_t=\begin{cases} (4\lambda_1\lambda_2)^{-X_t/2}, & X_t \text{ even},\\
2\lambda_1(4\lambda_1\lambda_2)^{-(X_t-1)/2}, & X_t \text{ odd}.\\ \end{cases}
\]

	The process~$(M_t)_t$ is shown to be a supermartingale with respect to the natural filtration~$(\mathcal F_t)_t$, as follows:
	\[
		\left.\frac{\mathrm{d}}{\mathrm{d}s}\mathbb E[M_{t+s} \mid \mathcal F_t]\right|_{s=0+} \le \lambda_1 M_t \left( \frac{1}{2\lambda_1}-1  \right) + M_t \left( 2 \lambda_2 - 1 \right) \le M_t \left(-\frac12 - \lambda_1 + 2\lambda_2\right) < 0,
	\]
and the same bounds with swapped indices when it has odd length. 
Hence if we define $\tau$ as the first time when $\{X_t=\ell\}$, by optional stopping,
\[
	1=\mathbb E[M_0]\ge \mathbb E[M_\tau;\tau<\infty]=(2\lambda_1)^{\lfloor|\Gamma|/2\rfloor}(2\lambda_2)^{\lceil|\Gamma|/2\rceil}\mathbb P(\tau<\infty) \]
and so we have the claimed bound on $\mathbb P(X_\tau=|\pi|)=\mathbb P(\tau<\infty)$.
\end{proof}

\subsection{Infection paths classified by length and cardinality}
We classify infection paths according to their length and the number of distinct vertices they
visit, and then bound the probability of each class.

For the rest of this section, we fix~$(\gamma_1,\gamma_2) \in \mathcal S$ and~$a > 0$. We will often assume that~$\lambda$ is small enough.
Define
\[
	\delta:=\gamma_1 \wedge \gamma_2 \wedge  (\Delta_1+\Delta_2)
\]
and fix~$b > 17/\delta$. Let
\begin{align*}
	\mathfrak{u}_1:=\lambda^{\mu_\star} \cdot (\log\tfrac{1}{\lambda})^{b},\qquad \mathfrak{u}_2:=\lambda^{\nu_\star} \cdot (\log\tfrac{1}{\lambda})^{b}.
\end{align*}

We label vertices of~$\mathcal G$ as \emph{blue} or \emph{red} as follows:
\begin{itemize}
	\item a vertex of type 1 is blue if its mark is larger than~$\mathfrak{u}_1$, and red otherwise;
	\item a vertex of type 2 is blue if its mark is larger than~$\mathfrak{u}_2$, and red otherwise.
\end{itemize}

We now state the two key propositions that give upper bounds to the probabilities of infection paths.
Their proofs are deferred to Section~\ref{ss_combinatorial} below.
\begin{proposition}\label{prop_bounds_together}
	There exists~$C > 0$ such that the following holds for~$\lambda$ small enough.
	Let~$k,\ell \in \mathbb N$ with~$k \ge 2$ and~$\ell \ge k-1$. Let~$P(k,\ell)$ be the probability that there is an infection path in~$\mathcal G$ started at the root and having ordered trace~$\Gamma$ satisfying the following properties:
	\begin{itemize}
		\item $\Gamma$ has length~$\ell$ and visits~$k$ distinct vertices;
		\item all vertices of~$\Gamma$ are blue, except for the last one, which is red.
	\end{itemize}
	
	Then, defining
	\begin{equation}\label{eq_def_of_M_f}
		M(\lambda,k,\ell):=\binom{\ell +1}{k}  k^{\ell+1-k} \cdot  (2 \lambda^{a \wedge 1})^{(\ell-2k)\vee 0} \cdot \left(\log \frac{1}{\lambda}\right)^{-16\left\lfloor \frac{k-3}{2} \right\rfloor},
	\end{equation}
	we have
\begin{equation*}
	P(k,\ell) \le \begin{cases} 
		C M(\lambda,k,\ell)\cdot  \lambda^{1+a-\Delta_2 \nu_{\star}+\bar{\gamma}_1\mu_{\star}} &\text{if $\ell$ is even};\\
	C M(\lambda,k,\ell)\cdot  \lambda^{1+\bar{\gamma}_2 \nu_\star} &\text{if $\ell$ is odd}.\end{cases}
\end{equation*}
\end{proposition}

\begin{proposition}\label{prop_bounds_together2}
	There exists~$C > 0$ such that the following holds for~$\lambda$ small enough.
	Let~$k,\ell \in \mathbb N$ with~$k \ge 2$ and~$\ell \ge k-1$. Let~$Q(k,\ell)$ be the probability that there is an infection path in~$\mathcal G$ started at the root, having ordered trace~$\Gamma$ satisfying the following properties:
	\begin{itemize}
		\item $\Gamma$ has length~$\ell$ and visits~$k$ distinct vertices;
		\item the last vertex of~$\Gamma$ is distinct from all other vertices;
		\item all vertices of~$\Gamma$ are blue.
	\end{itemize}
	
	Then, taking~$M(\lambda,k,\ell)$ as in~\eqref{eq_def_of_M_f},
\begin{equation*}
	Q(k,\ell) \le \begin{cases}
		CM(\lambda,k,\ell)\cdot \lambda^{1+a}&\text{if $\ell$ is even};\\
		CM(\lambda,k,\ell)\cdot \lambda&\text{if $\ell$ is odd}.
	\end{cases}
\end{equation*}
\end{proposition}

In giving an upper bound to~$M(\lambda,k,\ell)$, the following will be useful:
\begin{lemma}\label{lem_bound_B}
	For any~$B > 0$, for~$\lambda$ small enough and every~$k \in \{3,\ldots, \lfloor B \log \frac{1}{\lambda}\rfloor\}$, we have
	\[
		\sum_{\ell=k-1}^\infty {\ell+1 \choose k}k^{\ell+1-k} \cdot (2\lambda^{a \wedge 1})^{(\ell-2k)\vee0} \le (4k)^{3k}.
	\]
\end{lemma}
\begin{proof}
We write the sum as
	\[
		\sum_{j=k}^\infty {j \choose k}k^{j-k} \cdot (2\lambda^{a \wedge 1})^{(j-1-2k)\vee0}.
	\]
	
We bound~${j \choose k} \le j^{k \wedge (j-k)}$.
For~$j\in \{k,2k\}$, we bound
\[
	j^{k \wedge (j-k)} k^{j-k}\cdot  (2\lambda^{a \wedge 1})^{(j-1-2k)\vee 0} = j^{j -k} k^{j-k} \le (2k)^kk^k \le (2k)^{2k},
\]
so that
\[
	\sum_{j=k}^{2k}j^{k \wedge (j-k)}  k^{j-k}\cdot  (2 \lambda^{a \wedge 1})^{(j-1-2k)\vee 0} \le k(2k)^{2k}.
\]

For~$j > 2k$, we bound
\[
	j^k \le (2(j-k))^k = 2^k(j-k)^k \le 2^kk^{j-k},
\]
where we have used the inequality~$x^y \ge y^x$, valid whenever~$y \ge x \ge e$ (because $z\mapsto\tfrac{\log z}{z}$ is decreasing on $[e,\infty)$).  Then, still assuming~$j > 2k$,
	\begin{align*}
		j^{k \wedge (j-k)} k^{j-k}\cdot  (2\lambda^{a \wedge 1})^{(j-1-2k)\vee 0} &= j^k k^{j-k}\cdot  (2\lambda^{a \wedge 1})^{j-1-2k} \\
		&\le 2^k k^{2(j-k)}\cdot (2\lambda^{a \wedge 1})^{j-1-2k} = 2^k k^{2+2k}\cdot  (2k^{2}\lambda^{a \wedge 1})^{j-1-2k}.
	\end{align*}
	
This gives
\[
	\sum_{j=2k+1}^\infty j^{k \wedge (j-k)} k^{j-k}\cdot (2 \lambda^{a \wedge 1})^{(j-1-2k)\vee 0} \le 2^k k^{2k+2} \sum_{j=2k+1}^\infty (2k^2\lambda^{a \wedge 1})^{j-1-2k} \le 2\cdot 2^k k^{2k+2},
\]
since~$2k^{2}\lambda^{a \wedge 1} \ll 1$.
We have thus proved that
	\[\sum_{j=k}^\infty {j \choose k}k^{j-k} \cdot (2\lambda^{a \wedge 1})^{(j-1-2k)\vee 0}\le 2^{2k}k^{2k+1}+ 2^{k+1} k^{2k+2} \le (4k)^{3k}. \qedhere\]
\end{proof}

\begin{proof}[Proof of Proposition~\ref{prop_upper_bound}]
	Fix~$B > 0$ large, to be chosen later (in a way that will not depend on~$\lambda$), and define~$k_\star:=\lfloor B \log\tfrac{1}{\lambda}\rfloor$.

	Almost surely, there is no infinite infection path that only visits finitely many vertices of~$\mathcal G$. Based on this observation, it readily follows that
	\[
		\theta(\lambda) \le \mathbb P(\text{root has mark }\le \mathfrak{u}_1) + \sum_{k=2}^{k_\star}\sum_{\ell= k-1}^\infty P(k,\ell) + \sum_{\ell=k_\star-1}^\infty Q(k_\star,\ell).
	\]
	
	The first probability on the right-hand side above is~$\mathfrak{u}_1$.   Next, by Proposition~\ref{prop_bounds_together} and Lemma~\ref{lem_bound_B}, with 
	the bound~$\lfloor (k-3)/2 \rfloor \ge k/4$ for~$k \ge 6$, we have
\begin{align*}
	&\sum_{k=1}^{k_\star}\sum_{\ell=k-1}^\infty P(k,\ell) \\
	&\le \left(\sum_{k=1}^5 (4k)^{3k} + \sum_{k=4}^{k_\star} \left((4k)^{3} \cdot \left(\log \tfrac{1}{\lambda} \right)^{-4} \right)^{k} \right)\cdot \lambda^{(1+a-\Delta_2 \nu_{\star}+\bar{\gamma}_1\mu_{\star}) \wedge (1+\bar{\gamma}_2 \nu_\star)}.
\end{align*}

Now note that, for~$k \le k_\star$,
\[
	(4k)^3\cdot  \left( \log\tfrac{1}{\lambda}\right)^{-4} \le 8B^3 \left( \log\tfrac{1}{\lambda}\right)^{-1} \ll 1,
\]
so 
\[
	\sum_{k=1}^{k_\star}\sum_{\ell=k-1}^\infty P(k,\ell) \le \left(\sum_{k=1}^5 (4k)^{3k} + 2\right)\cdot \lambda^{(1+a-\Delta_2 \nu_{\star}+\bar{\gamma}_1\mu_{\star}) \wedge (1+\bar{\gamma}_2 \nu_\star)}.
\]

	Next, using Proposition~\ref{prop_bounds_together2} and Lemma~\ref{lem_bound_B}, we bound
\begin{align*}
	\sum_{\ell=k_\star-1}^\infty Q(k_\star,\ell) \le \left((4k_\star)^{3}\cdot \left( \log \tfrac{1}{\lambda}\right)^{-4} \right)^{k_\star} \cdot \lambda \le 2^{-k_\star} \cdot \lambda \le \lambda^{\frac{B}{2}\log 2+1}.
\end{align*}

We have proved that
	\[
		\theta(\lambda) \le \mathfrak{u}_1 + C\lambda^{(1+a-\Delta_2 \nu_{\star}+\bar{\gamma}_1\mu_{\star}) \wedge (1+\bar{\gamma}_2 \nu_\star)} + \lambda^{\frac{B}{2}\log 2+1}.
	\]
	
	Recalling that~$\mathfrak{u}_1 := \lambda^{\mu_\star} (\log \tfrac{1}{\lambda})^b$ and~$A_\star:=\mu_\star \wedge (1+\bar{\gamma}_2 \nu_\star)\wedge (1+a-\Delta_2 \nu_\star+\bar{\gamma}_1 \mu_\star)$, the result follows by taking~$B$ large enough.
\end{proof}

\subsection{Combinatorial paths and discovery trees}
\label{ss_combinatorial}
We now develop the combinatorial estimates needed for the proofs of
Propositions~\ref{prop_bounds_together} and~\ref{prop_bounds_together2}.
We define a  \emph{combinatorial path} as a finite sequence of nonnegative integers, starting with~$0$ and then $1$ and from that point only introducing new integers when they are $1$ larger than the previous maximum.
The following simple combinatorial bound will be useful. The proof is omitted.

%
%
\begin{lemma}\label{lem_count_combi}
	There are at most~${\ell+1 \choose k}k^{\ell+1-k}$ combinatorial paths of length~$\ell$ with~$k$ distinct vertices.
\end{lemma}

We associate a particular combinatorial path to paths in~$\mathcal G$, in the following way. Let~$\gamma=(\gamma_0,\ldots,\gamma_\ell)$ be some finite path in~$\mathcal G$ started from the root, and let~$k$ denote the number of distinct vertices visited by~$\gamma$. Define
\[
	I:=\{0\} \cup \{i \in \{1,\ldots,\ell\}:\; \gamma_i \notin \{\gamma_0,\ldots,\gamma_{i-1}\}\},
\]
the set of indices when the path~$\gamma$ visits a new vertex.  Let~$0=i_0 < \cdots < i_{k-1}$ denote the elements of~$I$ in increasing order. Note that~$\{\gamma_0,\ldots,\gamma_\ell\} = \{\gamma_{i_0},\ldots,\gamma_{i_{k-1}}\}$.
Define~$\varphi:\{\gamma_0,\ldots, \gamma_\ell\} \to \mathbb N_0$ by setting~$\varphi(\gamma_{i_n})=n$ for~$n=0,\ldots,k-1$. The combinatorial path~$\pi$ associated to $(\gamma_0,\ldots,\gamma_\ell)$ is then $(\varphi(\gamma_0),\ldots,\varphi(\gamma_\ell))$.

Let~$\pi=(\pi_0,\ldots,\pi_\ell)$ be a combinatorial path of length~$\ell$ with~$k$ distinct integers. The \emph{discovery tree of~$\pi$} is the tree~$T=T_\pi$ defined as follows. The set of vertices of~$T_\pi$ is~$\{\pi_0,\ldots,\pi_\ell\}$. Given distinct vertices~$m,n$, the edge~$\{m,n\}$ is included if the path jumps from~$m$ to~$n$ at its first visit to~$n$. More formally, the set of edges is
\[
	 \{\{\pi_j,\pi_{j+1}\}:\; j \in \{i \ge 0:\;\pi_{i+1} \notin \{\pi_0,\ldots,\pi_i\}\}\}.
\]

We let~$\deg_{T}(\cdot)$ and~$\mathrm{dist}_T(\cdot,\cdot)$ denote the degree and graph distance in~$T$, respectively.

These concepts are illustrated in Figure~\ref{fig_sk} below.
\begin{figure}[H]
	\begin{center}
		\fbox{\includegraphics[width = .8\textwidth]{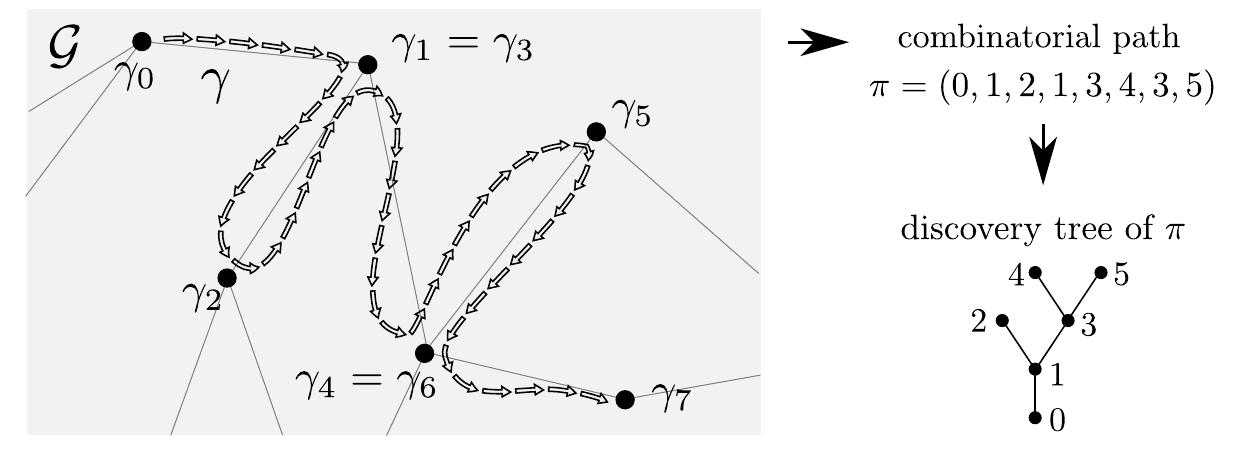}}
		\caption{Illustration of the combinatorial path~$\pi$ associated to a graph path~$\gamma$, and the discovery tree of~$\pi$}\label{fig_sk}
	\end{center}
\end{figure}

We define
\begin{align*}
	&\mathscr U_n := \int_{\mathfrak{u}_1}^1 u^{-\gamma_1\cdot n}\;\mathrm{d}u,\qquad\mathscr V_n := \int_{\mathfrak{u}_2}^1 v^{-\gamma_2\cdot n}\;\mathrm{d}v,\qquad n \in \mathbb N,
\end{align*}
The following two statements are consequences of Mecke's formula (see~\cite[Theorem~4.4]{last2018lectures}).
\begin{lemma}\label{lem_second_just_embed}
	Let~$\pi=(\pi_0,\ldots,\pi_\ell)$ be a combinatorial path of length~$\ell$ with~$k$ distinct vertices; assume that~$\pi_\ell \notin\{\pi_0,\ldots,\pi_{\ell-1}\}$. Let~$T$ be the discovery tree of~$\pi$.  Then, the expected number of paths~$\Gamma$ in~$\mathcal G$ that start at the root, have combinatorial path equal to~$\pi$, and have all vertices blue, except for the last, which is red, is:
			\[
				\prod_{\substack{m \in T\backslash \{\pi_\ell\}:\\ \mathrm{dist}_{T}(\pi_0,m) \text{ is even}}} \mathscr U_{\deg_T(m)} \cdot\prod_{\substack{m \in T\backslash \{\pi_\ell\}:\\ \mathrm{dist}_{T}(\pi_0,m) \text{ is odd}}} \mathscr V_{\deg_T(m)}\cdot \int_0^{\mathfrak{u}_1} u^{-\gamma_1}\;\mathrm{d}u
			\]
			if~$\ell$ is even, and
			\[
				\prod_{\substack{m \in T\backslash \{\pi_\ell\}:\\ \mathrm{dist}_{T}(\pi_0,m) \text{ is even}}} \mathscr U_{\deg_T(m)} \cdot\prod_{\substack{m \in T\backslash \{\pi_\ell\}:\\ \mathrm{dist}_{T}(\pi_0,m) \text{ is odd}}} \mathscr V_{\deg_T(m)}\cdot \int_0^{\mathfrak{u}_2} v^{-\gamma_2}\;\mathrm{d}v 
			\]
			if~$\ell$ is odd.
\end{lemma}

\begin{lemma}\label{lem_second_just_embed_blue}
	Let~$\pi=(\pi_0,\ldots,\pi_\ell)$ be a combinatorial path of length~$\ell$ with~$k$ distinct vertices, and let~$T$ be the discovery tree of~$\pi$.  Then, the expected number of paths~$\Gamma$ in~$\mathcal G$ that start at the root, have combinatorial path equal to~$\pi$, and only visit blue vertices, is:
			\[
				\prod_{\substack{m \in T:\\ \mathrm{dist}_{T}(\pi_0,m) \text{ is even}}} \mathscr U_{\deg_T(m)} \cdot\prod_{\substack{m \in T:\\ \mathrm{dist}_{T}(\pi_0,m) \text{ is odd}}} \mathscr V_{\deg_T(m)}.
			\]
\end{lemma}

Let~$T$ be a finite tree with root~$o$ and a distinguished leaf~$m^* \neq o$.  We define the \emph{tree weight function}
\begin{align*}
	F(T,o,m^*):= &\prod_{\substack{m \in T \backslash \{m^*\}:\\ \mathrm{dist}_T(o,m) \text{ is even}}} (2\lambda)^{(\deg_T(m)-1)\vee 1}\cdot \mathscr{U}_{\deg_T(m)} \\
	&\times \prod_{\substack{m \in T \backslash \{m^*\}:\\ \mathrm{dist}_T(o,m) \text{ is odd}}} (2\lambda^a)^{(\deg_T(m)-1)\vee 1}\cdot \mathscr{V}_{\deg_T(m)}.
\end{align*}

\begin{lemma}\label{lem_bound_on_F}
There exists~$C > 0$ such that, for any finite tree~$T$ with root~$o$ and a distinguished leaf~$m^* \neq o$, we have
	\[
		F(T,o,m^*) \le C\big(\log\tfrac{1}{\lambda}\big)^{-16\left\lfloor \frac{k-3}{2} \right\rfloor}\cdot ((\mathscr V_2\cdot \lambda^{1+a}) \vee \lambda).
	\]
\end{lemma}
We postpone the proof of this lemma to Section~\ref{ss_tree_weights} below.

\begin{lemma} \label{lem_use_bound_on_F1}
Let~$\pi=(\pi_0,\ldots,\pi_\ell)$ be a combinatorial path of length~$\ell$ with~$k$ distinct vertices; assume that~$\pi_\ell \notin\{\pi_0,\ldots,\pi_{\ell-1}\}$. Let~$T$ be the discovery tree of~$\pi$. Let~$\mathcal P(\pi)$ be the probability that there exists an infection path in~$\mathcal G$ started at the root and having an ordered trace~$\Gamma$ satisfying the following properties:
\begin{itemize}
\item the combinatorial path associated to~$\Gamma$ is~$\pi$;
\item all vertices of~$\Gamma$ are blue, except for the last one, which is red.
\end{itemize}
Then,
	\[
		\mathcal P(\pi) \le \begin{cases}
			(2\lambda^{a \wedge 1})^{(\ell-2k)\vee 0} \cdot F(T,\pi_0,\pi_\ell) \cdot \int_0^{\mathfrak{u}_1} u^{-\gamma_1}\;\mathrm{d}u &\text{if $\ell$ is even;}\\[.2cm]
			(2\lambda^{a \wedge 1})^{(\ell-2k)\vee 0} \cdot F(T,\pi_0,\pi_\ell) \cdot \int_0^{\mathfrak{u}_2} v^{-\gamma_2}\;\mathrm{d}v&\text{if $\ell$ is odd.}
		\end{cases}
	\]
\end{lemma}

\begin{proof}
For~$i \in \mathbb N_0$, define
\[\alpha(i):=\begin{cases}1&\text{if $i$ is even};\\
a&\text{if $i$ is odd}. \end{cases}\]

Using Mecke's formula and Lemma \ref{lem:martingale}, we have
$
\mathcal P(\pi) \le \mathcal E(\pi)\cdot \prod_{i=0}^{\ell-1} 2\lambda^{\alpha(i)},
$
where~$\mathcal E(\pi)$ is the expected number of paths in~$\mathcal G$ that start at the root, have~$\pi$ as the associated combinatorial path, and have all vertices blue except for the last one, which is red.

Define
	\[
		J:=\{i \in \{0,\ldots, \ell-1\}:\;\pi_{i+1} \notin \{\pi_0,\ldots,\pi_i\}\}.
	\]
	
	For each leaf vertex~$m$ of~$T$ with~$m\neq \pi_\ell$, we can choose an index~$\iota(m) \in \{0,\ldots,\ell-1\}$ such that~$\pi_{\iota(m)} = m$. Define
	\[
		J':=\{\iota(m):\; m \text{ is a leaf vertex of $T$ with~$m \neq \pi_\ell$}\}.
	\]
	
	Using Lemma~\ref{lem_second_just_embed} and the definition of~$F(T,\pi_0,\pi_\ell)$, we now have the equality
	\[
		\prod_{i \in J \cup J'} \lambda^{\alpha(i)} \cdot \mathcal E(\pi) = \begin{cases} F(T,\pi_0,\pi_\ell) \cdot  \int_0^{u^\star} u^{-\gamma_1}\;\mathrm{d}u & \text{if $\ell$ is even};\\[.2cm]
		F(T,\pi_0,\pi_\ell) \cdot  \int_0^{v^\star} v^{-\gamma_2}\;\mathrm{d}v & \text{if $\ell$ is odd}.\end{cases}
	\]
	
The proof is concluded by bounding
	\[\prod_{i \in \{0,\ldots,\ell-1\} \backslash (J \cup J')} 2\lambda^{\alpha(i)} \le (2\lambda^{a \wedge 1})^{| \{0,\ldots,\ell-1\} \backslash (J \cup J')|} \le (2\lambda^{a \wedge 1})^{(\ell-2k)\vee 0},
	\]
	since~$|J| \le k$ and~$|J'| \le k$.
\end{proof}

\begin{proof}[Proof of Proposition~\ref{prop_bounds_together}]
	The statement follows from Lemma~\ref{lem_count_combi}, Lemma~\ref{lem_bound_on_F}, and Lemma~\ref{lem_use_bound_on_F1}.
\end{proof}

\begin{lemma} \label{lem_use_bound_on_F12}
Let~$\pi=(\pi_0,\ldots,\pi_\ell)$ be a combinatorial path of length~$\ell$ with~$k$ distinct vertices; assume that~$\pi_\ell \notin\{\pi_0,\ldots,\pi_{\ell-1}\}$. Let~$T$ be the discovery tree of~$\pi$. Let~$\mathcal Q(\pi)$ be the probability that there exists an infection path in~$\mathcal G$ started at the root and having an ordered trace~$\Gamma$ satisfying the following properties:
\begin{itemize}
\item the combinatorial path associated to~$\Gamma$ is~$\pi$;
\item all vertices of~$\Gamma$ are blue.
\end{itemize}
Then,
	\[
		\mathcal Q(\pi) \le 
			(2\lambda^{a \wedge 1})^{(\ell-2k)\vee 0} \cdot F(T,\pi_0,\pi_\ell). 
	\]
\end{lemma}

The proof is similar to that of Lemma~\ref{lem_use_bound_on_F1}, with the application of Lemma~\ref{lem_second_just_embed} being replaced by an application of Lemma~\ref{lem_second_just_embed_blue}. We skip the details.

\begin{proof}[Proof of Proposition~\ref{prop_bounds_together2}]
	The statement follows from Lemma~\ref{lem_count_combi}, Lemma~\ref{lem_bound_on_F}, and Lemma~\ref{lem_use_bound_on_F12}.
\end{proof}

\subsection{Bound on tree weight function}\label{ss_tree_weights}
It remains to prove Lemma~\ref{lem_bound_on_F}. We start with the following auxiliary lemma. Recall that we have defined~$\delta:=\gamma_1 \wedge \gamma_2 \wedge (\Delta_1+\Delta_2)$ and then fixed~$b > 17/\delta$.
\begin{lemma}
	For~$\lambda$ small enough, the following inequalities hold:
	\begin{align}
		\label{eq_just_stars}&\mathfrak{u}_1^{\gamma_1} \wedge \mathfrak{u}_2^{\gamma_2} \ge (\log\tfrac{1}{\lambda})^{\delta b} \cdot \lambda^{1+a};\\[.2cm]
		\label{eq_just_one_and_twoa}&\mathfrak{u}_1^{\Delta_1}\cdot \mathfrak{u}_2^{\gamma_2} \ge (\log\tfrac{1}{\lambda})^{\delta b} \cdot \lambda^{1+2a};\\[.2cm]
		\label{eq_just_two_and_a}&\mathfrak{u}_1^{\gamma_1}\cdot \mathfrak{u}_2^{\Delta_2} \ge (\log\tfrac{1}{\lambda})^{\delta b} \cdot \lambda^{2+a};\\[.2cm]
		\label{eq_just_directs}&\mathfrak{u}_1^{\Delta_1}\cdot \mathfrak{u}_2^{\Delta_2} \ge (\log\tfrac{1}{\lambda})^{\delta b} \cdot \lambda^{1+a}.
	\end{align}
\end{lemma}

\begin{proof}

	We write $f_{\circ}:=\gamma_1 \cdot \mu_{\circ}$ and $g_{\circ}:=\gamma_2 \cdot \nu_{\circ}$ for~$\circ \in \{\mathsf S, \mathsf B, \mathsf D, \star\}$.

	For~\eqref{eq_just_stars}, using~$f_\star \le f_{\mathrm{S}}=1+a$, we have
	\[
		\mathfrak{u}_1^{\gamma_1} = \left( \log\tfrac{1}{\lambda}\right)^{\gamma_1 b} \cdot \lambda^{f_\star} \ge \left( \log\tfrac{1}{\lambda}\right)^{\delta b} \cdot \lambda^{f_{\mathrm{S}}}  = \left( \log\tfrac{1}{\lambda}\right)^{\delta b} \cdot \lambda^{1+a};  
	\]
	the lower bound for~$\mathfrak{u}_2^{\gamma_2}$ in~\eqref{eq_just_stars} is treated in the same way, using~$g_\star \le g_{\mathrm{S}}=1+a$.

	For~\eqref{eq_just_one_and_twoa}, let us first assume that~$\gamma_1 \le 1/2$, which implies that~$\Delta_1 = 0$. In that case, the inequality~\eqref{eq_just_one_and_twoa} follows immediately from~\eqref{eq_just_stars}. Now assume that~$\gamma_1 > 1/2$, so that~$\Delta_1 = 2\gamma_1-1>0$.
	We then have
	\begin{align*}
		\mathfrak{u}_1^{\Delta_1}\cdot \mathfrak{u}_2^{\gamma_2} &= \left(\log \tfrac{1}{\lambda} \right)^{(\Delta_1+\gamma_2)b}\cdot \lambda^{\frac{\Delta_1}{\gamma_1}f_\star + g_\star}\\[.2cm]
		&\ge \left(\log \tfrac{1}{\lambda} \right)^{\delta b}\cdot \lambda^{\frac{\Delta_1}{\gamma_1}f_{\mathrm{S}}+ g_{\mathrm{B}}} = \left(\log \tfrac{1}{\lambda} \right)^{\delta b}\cdot \lambda^{1+2a}, 
	\end{align*}
	where the last equality follows from
	\[
		\frac{\Delta_1}{\gamma_1}f_{\mathrm{S}}+ g_{\mathrm{B}} = \frac{2\gamma_1 - 1}{\gamma_1} (1+a) + \frac{\bar{\gamma}_1}{\gamma_1} + \frac{1}{\gamma_1} a = 1+2a.
	\]

	Turning to~\eqref{eq_just_two_and_a}, the inequality is again immediate from~\eqref{eq_just_stars} if~$\gamma_2 \le 1/2$, so we assume that~$\gamma_2 > 1/2$, which implies that~$\Delta_2 = 2\gamma_2-1>0$. We then bound
	\begin{align*}
		\mathfrak{u}_1^{\gamma_1}\cdot \mathfrak{u}_2^{\Delta_2} &= \left(\log \tfrac{1}{\lambda} \right)^{(\gamma_1+\Delta_2)b}\cdot \lambda^{ f_\star + \frac{\Delta_2}{\gamma_2} g_\star}\\[.2cm]
		&\ge \left(\log \tfrac{1}{\lambda} \right)^{\delta b}\cdot \lambda^{ f_{\mathrm{B}}+ \frac{\Delta_2}{\gamma_2} g_{\mathrm{S}}} = \left(\log \tfrac{1}{\lambda} \right)^{\delta b}\cdot \lambda^{2+a}, 
	\end{align*}
	since
	\[
		f_{\mathrm{B}}+ \frac{\Delta_2}{\gamma_2} g_{\mathrm{S}} = \frac{1}{\gamma_2} + \frac{\bar{\gamma}_2}{\gamma_2}a + \frac{2\gamma_2-1}{\gamma_2}(1+a) = 2+a.
	\]

	We now prove~\eqref{eq_just_directs}. In case~$\gamma_1 \le 1/2$, we have~$\Delta_1=0$, and also~$\gamma_2 \ge 1/2$, so~$\Delta_2 = 2\gamma_2 - 1$, and then,
	\[
		\mathfrak{u}_1^{\Delta_1}\cdot \mathfrak{u}_2^{\Delta_2} = \mathfrak{u}_2^{2\gamma_2-1} \ge \mathfrak{u}_2^{\gamma_2} \stackrel{\eqref{eq_just_stars}}{\ge} \left( \log \tfrac{1}{\lambda}\right)^{\delta b} \cdot \lambda^{1+a}.
	\]
	
	The case where~$\gamma_2 \le 1/2$ is treated similarly. We now assume that~$\gamma_1,\gamma_2>1/2$, so that~$\Delta_1,\Delta_2 > 0$. Then,
	\begin{align*}
		\mathfrak{u}_1^{\Delta_1}\cdot \mathfrak{u}_2^{\Delta_2} = \left( \log \tfrac{1}{\lambda}\right)^{(\Delta_1+\Delta_2)b}\cdot \lambda^{\Delta_1 f_\star + \Delta_2 g_\star} \ge  \left( \log \tfrac{1}{\lambda}\right)^{\delta b}\cdot \lambda^{\Delta_1 f_{\mathrm{D}}+ \Delta_2 g_{\mathrm{D}}}=
		\left( \log \tfrac{1}{\lambda}\right)^{\delta b}\cdot \lambda^{1+a},
	\end{align*}
	since
	\[
		\Delta_1 f_{\mathrm{D}}+ \Delta_2 g_{\mathrm{D}} = \frac{2\gamma_1-1}{\gamma_1} \left(\frac{\gamma_1\gamma_2}{\Delta} + \frac{\gamma_1 \bar{\gamma}_2}{\Delta}a \right) + \frac{2\gamma_2-1}{\gamma_2} \left( \frac{\bar{\gamma}_1\gamma_2}{\Delta} + \frac{\gamma_1\gamma_2}{\Delta}a\right) = 1+a.
	\]
	This concludes the proof of the four claimed bounds.
\end{proof}

Let~$\mathscr T$ denote the set of triples~$(T,o,m^*)$, where~$T$ is a finite tree,~$o$ is a root vertex of~$T$ and~$m^* \neq o$ is a distinguished leaf of~$T$. We will sometimes abuse notation and write~$T$ instead of the triple~$(T,o,m^*)$.

We now define three \emph{reduction operations} that may be applied to the tree~$T$ of the triple~$(T,o,m^*) \in \mathscr T$, producing a smaller tree~$T'$, which still contains the root~$o$ and the distinguished leaf~$m^*$.\\[.2cm]
	\textbf{Op1}: \emph{Trimming a spare leaf}. Assume~$x,y$ are vertices of~$T\backslash \{m^*\}$ such that~$y$ is a leaf,~$y$ is a child, and not the only child of~$x$. Then, remove~$y$ (and the edge~$\{x,y\}$) from~$T$.\\[.2cm]
\textbf{Op2}: \emph{Trimming a spare segment of length~$2$}. Assume~$x,y,z$ are vertices of~$T\backslash \{m^*\}$ such that~$z$ is a leaf,~$z$ is the only child of~$y$, and~$y$ is a child, and not the only child, of~$x$. Then, remove~$y$ and~$z$ (and the edges~$\{x,y\}$ and~$\{y,z\}$) from~$T$.\\[.2cm]
\textbf{Op3}: \emph{Collapsing a line segment of length 2 connecting non-leaves}. Assume~$x,y,z$ are vertices of~$T\backslash \{m^*\}$ with~$x \sim y \sim z$, such that~$x$ and~$z$ are not leaves, and~$x$ and~$z$ are the only two neighbours of~$y$. Then, ``remove~$y$ and collapse~$x$ and~$z$''; formally, remove~$x,y,z$ (and the edges~$\{x,y\}$ and~$\{y,z\}$) from~$T$, add a new vertex~$w$, and make~$w$ a neighbour of all the vertices that were neighbours of~$x$ and~$z$, other than~$y$.

These operations have two important features. First, as illustrated in Figure~\ref{fig_trim}, they allow us to reduce a tree into a line segment with extremities~$o$ and~$m^*$, and length equal to either~$1$ or~$2$. Second, as Lemma~\ref{lem_factor_reduce} will show, each time one of the operations is applied, it produces a tree in which the weight~$F$ is at least a multiplicative factor~$(\log \tfrac{1}{\lambda})^{16}$ larger than the previous one. By iterating this, we will show that large trees have small weights; see the proof of Lemma~\ref{lem_bound_on_F} below.
\begin{figure}[H]
	\begin{center}
		\fbox{\includegraphics[width = .6\textwidth]{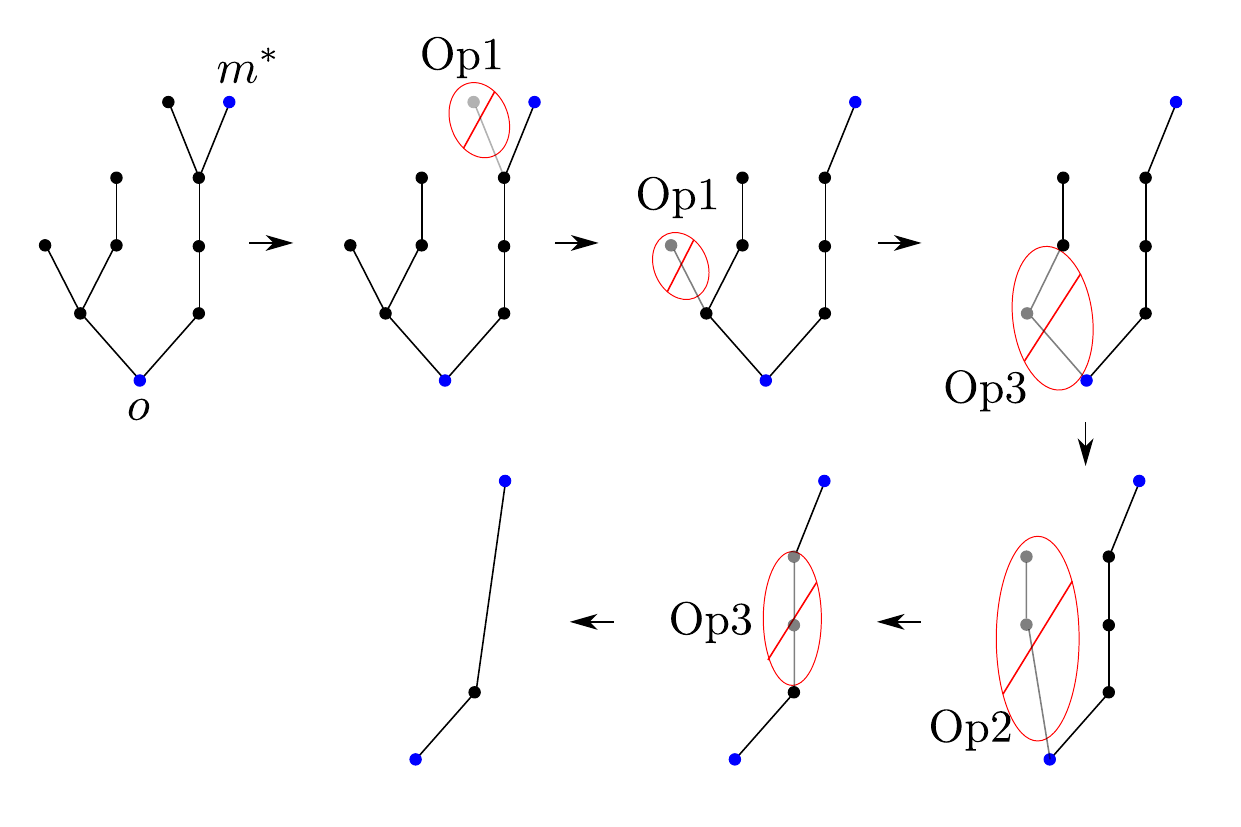}}
		\caption{Example of application of three operations to reduce a tree to a line segment of length two.}\label{fig_trim}
	\end{center}
\end{figure}

\begin{lemma}\label{lem_factor_reduce}
	Assume that~$(T',o,m^*)$ is obtained from~$(T,o,m)$ by one of the three operations above. Then,
	\[
		F(T,o,m^*) \le \left(\log\tfrac{1}{\lambda} \right)^{-16}\cdot  F(T',o,m^*).
	\]
\end{lemma}
\begin{proof}
	We can choose constants~$C_{\gamma_1}>0$ and~$C_{\gamma_2}>0$, only depending on~$\gamma_1$ and~$\gamma_2$, respectively, such that for~$\lambda$ small enough,
	\begin{equation}\label{eq_just_for_scru}
		C_{\gamma_1}^{-1} \le \frac{\mathscr U_n}{\mathfrak{u}_1^{(1-\gamma_1 n)\wedge 0}} \le C_{\gamma_1} \log\left(\tfrac{1}{\lambda}\right), \qquad C_{\gamma_2}^{-1} \le \frac{\mathscr V_n}{\mathfrak{u}_2^{(1-\gamma_2 n)\wedge 0}} \le C_{\gamma_2} \log\left(\tfrac{1}{\lambda}\right)
	\end{equation}
	for all~$n \in \mathbb N$; the~$\log$ factors are included to cover the cases where zero exponents appear in the integrals. In the bounds that follow, we will denote by~$C$ a product of powers of these constants, which may change from line to line.

	Assume~$T'$ is obtained from~$T$ using Op1. Let~$x$ and~$y$ be as in the description of~Op1. If~$\mathrm{dist}_T(o,x)$ is even, we have
	\[
		\frac{F(T,o,m^*)}{F(T',o,m^*)} = \lambda^{1+a} \cdot \frac{\mathcal{U}_{\deg_T(x)}}{\mathcal{U}_{\deg_T(x)-1}} \cdot \mathcal{V}_1 \stackrel{\eqref{eq_just_for_scru}}{\le}  \lambda^{1+a}\cdot C\log\tfrac{1}{\lambda}\cdot \mathfrak{u}_1^{-\gamma_1} \stackrel{\eqref{eq_just_stars}}{\le} \left( \log \tfrac{1}{\lambda}\right)^{-16}
	\]
	when~$\lambda$ is small.
	Similarly, if~$\mathrm{dist}_T(o,x)$ is odd, we have
	\[
		\frac{F(T,o,m^*)}{F(T',o,m^*)} = \lambda^{1+a} \cdot \frac{\mathcal{V}_{\deg_T(x)}}{\mathcal{V}_{\deg_T(x)-1}} \cdot \mathcal{U}_1 \le \left( \log \tfrac{1}{\lambda}\right)^{-16}.
	\]

	Now, assume~$T'$ is obtained from~$T$ using Op2, and let~$x,y,z$ be as in the description of that operation. If~$\mathrm{dist}_T(o,x)$ is even, we have
	\[
		\frac{F(T,o,m^*)}{F(T',o,m^*)} = \lambda^{2+a} \cdot \frac{\mathcal{U}_{\deg_T(x)}}{\mathcal{U}_{\deg_T(x)-1}} \cdot \mathcal{V}_2 \cdot \mathcal{U}_1 \stackrel{\eqref{eq_just_for_scru}}{\le}\lambda^{2+a} \cdot C\log\tfrac{1}{\lambda}\cdot \mathfrak{u}_1^{-\gamma_1} \cdot \mathfrak{u}_2^{-\Delta_2} \stackrel{\eqref{eq_just_two_and_a}}{\le} \left( \log \tfrac{1}{\lambda}\right)^{-16}.
	\]
	
	Similarly, if~$\mathrm{dist}_T(o,x)$ is odd,
	\[
		\frac{F(T,o,m^*)}{F(T',o,m^*)} = \lambda^{1+2a} \cdot \frac{\mathcal{V}_{\deg_T(x)}}{\mathcal{V}_{\deg_T(x)-1}} \cdot \mathcal{U}_2 \cdot \mathcal{V}_1 \stackrel{\eqref{eq_just_for_scru}}{\le}\lambda^{1+2a} \cdot C\log\tfrac{1}{\lambda}\cdot \mathfrak{u}_2^{-\gamma_2} \cdot \mathfrak{u}_1^{-\Delta_1} \stackrel{\eqref{eq_just_one_and_twoa}}{\le} \left( \log \tfrac{1}{\lambda}\right)^{-16}.
	\]
	
	Finally, assume~$T'$ is obtained from~$T$ using Op3, and let~$x,y,z$ be as in the description of that operation. If~$\mathrm{dist}_T(o,x)$ is even, we have
	\[
		\frac{F(T,o,m^*)}{F(T',o,m^*)} = \lambda^{1+a} \cdot \frac{\mathscr U_{\deg_T(x)}\cdot \mathscr U_{\deg_T(z)}}{\mathscr U_{\deg_T(x)+\deg_T(z)-2}} \cdot \mathscr V_2 \le \lambda^{1+a} \cdot C \log\tfrac{1}{\lambda}\cdot \mathfrak{u}_1^{-\Delta_1}\cdot \mathfrak{u}_2^{-\Delta_2} \stackrel{\eqref{eq_just_directs}}{\le} \left( \log \tfrac{1}{\lambda}\right)^{-16}.
	\]
	
	Similarly, if~$\mathrm{dist}_T(o,x)$ is odd, then
	\[
		\frac{F(T,o,m^*)}{F(T',o,m^*)} = \lambda^{1+a} \cdot \frac{\mathscr V_{\deg_T(x)}\cdot \mathscr V_{\deg_T(z)}}{\mathscr V_{\deg_T(x)+\deg_T(z)-2}} \cdot \mathscr U_2 \le \lambda^{1+a} \cdot C \log\tfrac{1}{\lambda}\cdot \mathfrak{u}_2^{-\Delta_2}\cdot \mathfrak{u}_1^{-\Delta_1} \stackrel{\eqref{eq_just_directs}}{\le} \left( \log \tfrac{1}{\lambda}\right)^{-16}.\qedhere
	\]
\end{proof}

\begin{proof}[Proof of Lemma~\ref{lem_bound_on_F}]
	First assume that~$\mathrm{dist}_T(o,m^*)$ is even. By successive applications of Op1, Op2, and Op3, we can reduce~$T$ to a line segment of length~$2$ between~$o$ and~$m^*$. Calling this final graph~$T'$, we have~$F(T',o,m^*)=C\mathscr V_2\lambda^{1+a}$ for some~$C > 0$. Each of the reduction operations reduces the number of vertices in the tree by at most two, and there are three vertices in the end, so the number of operations carried out is at least~$\lfloor (k-3)/2 \rfloor$. The desired bound then follows from repeatedly applying Lemma~\ref{lem_factor_reduce}.

	The case where~$\mathrm{dist}_T(o,m^*)$ is odd is treated in the same way, except that the final graph~$T'$ which arises from the reduction is just the two vertices~$o,m^*$, and the edge connecting them; then,~$F(T',o,m^*)\le C\lambda$.
\end{proof}

\begin{appendix}
\section*{}\label{appn} 

	Our first goal in this section is to find~$(\mu_\star,\nu_\star)$ explicitly.
	\subsection{Target optimisation}
\begin{proposition}\label{prop_target_opt}
	Assume that~$(\gamma_1,\gamma_2) \in \mathcal S$ and~$a > 0$.
\begin{itemize}
	\item[$\mathrm{(I)}$] If~$\gamma_2 \le \frac12$ (and consequently~$\gamma_1 > \frac12$), we have
		\[
			(\mu_\star,\nu_\star) = \begin{cases}
				(\mu_{\mathsf S}, \nu_{\mathsf B})&\text{if } a \le \frac{\bar{\gamma}_2}{\Delta_2};\\[.2cm]
				(\mu_{\mathsf S}, \nu_{\mathsf S})&\text{if } a > \frac{\bar{\gamma}_2}{\Delta_2}.
			\end{cases}
		\]
	\item[$\mathrm{(II)}$] If~$\gamma_1 \le \frac12$ (and consequently~$\gamma_2 > \frac12$), we have
		\[
			(\mu_\star,\nu_\star) = \begin{cases}
				(\mu_{\mathsf S}, \nu_{\mathsf S})&\text{if } a \le \frac{\Delta_1}{\bar{\gamma}_1};\\[.2cm]
				(\mu_{\mathsf B}, \nu_{\mathsf S})&\text{if } a > \frac{\Delta_1}{\bar{\gamma}_1}.
			\end{cases}
		\]

	\item[$\mathrm{(III)}$] If~$\gamma_1,\gamma_2 > \frac12$ and~$\frac{1}{\gamma_1}+\frac{1}{\gamma_2} > 3$, then
		\[
			(\mu_\star,\nu_\star) = \begin{cases}
				(\mu_{\mathsf S}, \nu_{\mathsf B})&\text{if } a \le \frac{\Delta_1}{\bar{\gamma}_1};\\[.2cm]
				(\mu_{\mathsf S}, \nu_{\mathsf S})&\text{if } \frac{\Delta_1}{\bar{\gamma}_1} < a \le \frac{\bar{\gamma}_2}{\Delta_2};\\[.2cm]
				(\mu_{\mathsf B}, \nu_{\mathsf S})&\text{if } a > \frac{\bar{\gamma}_2}{\Delta_2}.
			\end{cases}
		\]
	\item[$\mathrm{(IV)}$] If~$\gamma_1, \gamma_2 > \frac12$ and~$\frac{1}{\gamma_1}+\frac{1}{\gamma_2} \le 3$, then
		\[
			(\mu_\star,\nu_\star) = \begin{cases}
				(\mu_{\mathsf S}, \nu_{\mathsf B})&\text{if } a \le \frac{\bar{\gamma}_1 \bar{\gamma}_2}{\gamma_1\gamma_2 + \gamma_2 - 1};\\[.2cm]
				(\mu_{\mathsf D}, \nu_{\mathsf D})&\text{if } \frac{\bar{\gamma}_1 \bar{\gamma}_2}{\gamma_1\gamma_2 + \gamma_2 - 1} < a \le \frac{\gamma_1\gamma_2 + \gamma_1 - 1}{\bar{\gamma}_1 \bar{\gamma}_2};\\[.2cm]
				(\mu_{\mathsf B}, \nu_{\mathsf S})&\text{if } a > \frac{\gamma_1\gamma_2 + \gamma_1 - 1}{\bar{\gamma}_1 \bar{\gamma}_2}.
			\end{cases}
		\]
\end{itemize}
\end{proposition}

This is illustrated in Figure~\ref{fig_first}.  

\begin{figure}[H]
	\begin{center}
		\fbox{\includegraphics[width = .8\textwidth]{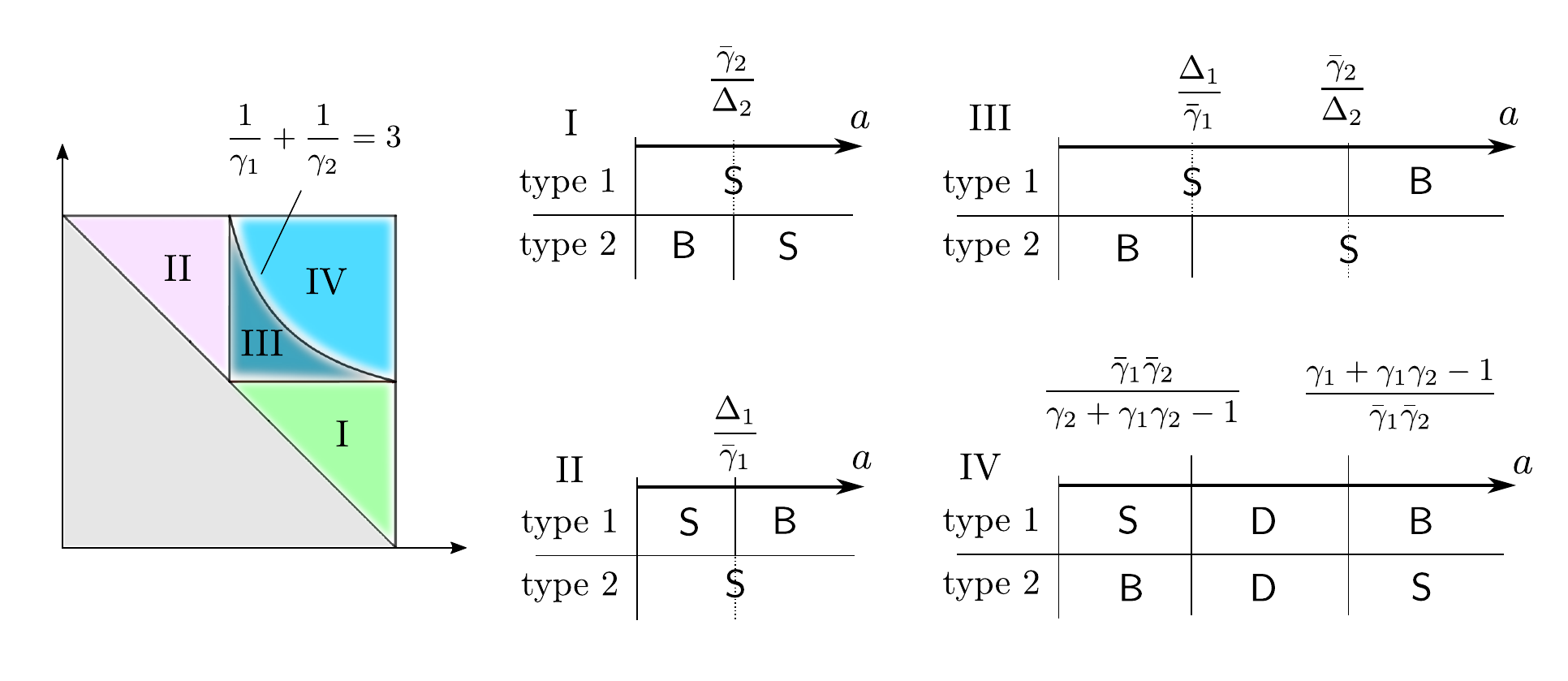}}
		\caption{Minimisers given by Proposition~\ref{prop_target_opt}, as functions of~$\gamma_1,\gamma_2, a$.}\label{fig_first}
	\end{center}
\end{figure}

It will be convenient to define~$f_{\circ}:=\gamma_1 \cdot \mu_{\circ}$ for~$\circ \in \{\mathsf S, \mathsf B, \mathsf D\}$, so that
	\[
		f_{\mathsf S}(\gamma_1,\gamma_2,a) = 1+a,\qquad f_{\mathsf B}(\gamma_1,\gamma_2, a) = \frac{1}{\gamma_2} + \frac{\bar{\gamma}_2}{\gamma_2} a,\qquad f_{\mathsf D}(\gamma_1,\gamma_2,a) = \frac{\gamma_1\gamma_2}{\Delta} + \frac{\gamma_1\bar{\gamma}_2}{\Delta}a.
	\]
	
	Note that although~$f_{\mathsf S}$ does not depend on~$\gamma_1,\gamma_2$ and~$f_{\mathsf B}$ does not depend on~$\gamma_1$, we write all three functions as functions of~$(\gamma_1,\gamma_2,a)$.

\begin{lemma}\label{lem_first_easy_case}
	If~$(\gamma_1,\gamma_2,a) \in \mathcal S \times (0,\infty)$ with~$\gamma_2 \le 1/2$, then~$\min\{f_{\mathsf S},f_{\mathsf B},f_{\mathsf D}\} = f_{\mathsf S}$.
\end{lemma}
\begin{proof}
We have
	$f_{\mathsf B}(\gamma_1,\gamma_2,0+) = \frac{1}{\gamma_2} > 1 = f_{\mathsf S}(\gamma_1,\gamma_2,0+)$
	and
	\[
		\frac{\partial f_{\mathsf B}}{\partial a} = \frac{\bar{\gamma}_2}{\gamma_2} > 1 = \frac{\partial f_{\mathsf S}}{\partial a} \quad \text{ when } \gamma_2 \le \frac12.
	\]
	
	This shows that~$f_{\mathsf S} < f_{\mathsf B}$ when~$\gamma_2 \le 1/2$. 
	Using~$0<\bar{\gamma}_1 \bar{\gamma}_2 = \gamma_1\gamma_2 - \Delta$, we can verify that
\begin{equation}
	\label{eq_bridge_zero}
	f_{\mathsf D}(\gamma_1,\gamma_2,0+) = \frac{\gamma_1\gamma_2}{\Delta}>1=f_{\mathsf S}(\gamma_1,\gamma_2, 0+).
\end{equation}

	Next, we compute
	\begin{equation*}
		\frac{\partial^2 f_{\mathsf D}}{\partial \gamma_1 \partial a} = \frac{\partial}{\partial \gamma_1} \left(\frac{\gamma_1\bar{\gamma}_2}{\Delta} \right) = \bar{\gamma}_2 \cdot \frac{\Delta - \gamma_1}{\Delta^2} = \bar{\gamma}_2 \cdot \frac{\gamma_2-1}{\Delta^2} < 0,
	\end{equation*}
	and
	\begin{equation*}
		\lim_{\gamma_1 \uparrow 1}\frac{\partial f_{\mathsf D}}{\partial a} =\lim_{\gamma_1 \uparrow 1}  \frac{\gamma_1 \bar{\gamma}_2}{\Delta} = \frac{\bar{\gamma}_2}{\gamma_2}.
	\end{equation*}
	
	Consequently,
	\begin{equation}
	\label{eq_new_compBD}
		\frac{\partial f_{\mathsf D}}{\partial a} \ge \frac{\bar{\gamma}_2}{\gamma_2} \qquad \text{ for all }(\gamma_1,\gamma_2,a) \in \mathcal S \times (0,\infty).
	\end{equation}
	
	In the case~$\gamma_2 \le 1/2$, we also have~$\frac{\bar{\gamma}_2}{\gamma_2} > 1$. This gives
	\begin{equation}
		\label{eq_bridge_01}
		\gamma_2 \le 1/2 \quad \Longrightarrow \quad \frac{\partial f_{\mathsf D}}{\partial a} > 1 = \frac{\partial f_{\mathsf S}}{\partial a}.
	\end{equation}
	
	Combining~\eqref{eq_bridge_zero} and~\eqref{eq_bridge_01} gives the desired statement.
\end{proof}

	A direct computation shows that, when~$\gamma_2 > 1/2$,
\begin{equation}\label{eq_direct_comp}
		 f_{\mathsf S}(\gamma_1, \gamma_2, \tfrac{\bar{\gamma}_2}{\Delta_2}) = f_{\mathsf B}(\gamma_1, \gamma_2, \tfrac{\bar{\gamma}_2}{\Delta_2}) = \tfrac{\gamma_2}{\Delta_2},
\end{equation}
with
\begin{equation}\label{eq_direct_comp0}
	\min\{f_{\mathsf S},f_{\mathsf B}\} = \begin{cases} f_{\mathsf S}&\text{if } a \le \tfrac{\bar{\gamma_2}}{\Delta_2};\\[.2cm]
	f_{\mathsf B}&\text{if } a > \tfrac{\bar{\gamma_2}}{\Delta_2}.\end{cases}
\end{equation}

We now define
	\[
		\kappa(\gamma) := \frac{\gamma}{3\gamma-1},\qquad \gamma \in (1/2,1).
	\]
	
Using that~$\kappa$ is decreasing, with~$\kappa(1/2)=1$ and~$\kappa(1)=1/2$, we have~$\kappa(\gamma_2)+\gamma_2 > 1$; hence,
\[(\kappa(\gamma_2),\gamma_2) \in \mathcal S \text{ for all } \gamma_2 \in (1/2,1).\]

Again using the fact that~$\kappa$ is decreasing, together with~$\frac{1}{\gamma}+\frac{1}{\kappa(\gamma)} = 3$, we have
\begin{equation}\label{eq_gamma_kappa}
		\frac{1}{\gamma_1}+\frac{1}{\gamma_2} < 3 \quad \text{if and only if} \quad \gamma_1 > \kappa(\gamma_2).
\end{equation}

\begin{lemma}\label{lem_to_postpone}
	If~$(\gamma_1,\gamma_2,a) \in \mathcal S \times (0,\infty)$ with~$\gamma_2 > 1/2$, then
	\begin{equation*}
		f_{\mathsf D}(\gamma_1,\gamma_2,\tfrac{\bar{\gamma}_2}{\Delta_2}) \begin{cases}
			> \tfrac{\gamma_2}{\Delta_2} &\text{if } \gamma_1 \in (\bar{\gamma}_2, \kappa(\gamma_2));\\[.2cm]
			= \tfrac{\gamma_2}{\Delta_2} &\text{if } \gamma_1  = \kappa({\gamma}_2);\\[.2cm]
			< \tfrac{\gamma_2}{\Delta_2} &\text{if } \gamma_1  \in (\kappa({\gamma}_2),1).
		\end{cases}
	\end{equation*}
\end{lemma}
\begin{proof}
	We write, for arbitrary~$\gamma_1 \in (\bar{\gamma}_2,1)$:
	\[
		f_{\mathsf D}(\gamma_1,\gamma_2,\tfrac{\bar{\gamma}_2}{\Delta_2}) = \frac{\gamma_1\gamma_2}{\Delta} + \frac{\gamma_1 \bar{\gamma}_2^2}{\Delta\cdot \Delta_2} = \frac{\gamma_1}{\Delta}\cdot \frac{3\gamma_2^2-3\gamma_2+1}{\Delta_2} = \frac{1}{\gamma_2 (\frac{1}{\gamma_1} + \frac{1}{\gamma_2} - \frac{1}{\gamma_1\gamma_2})} \cdot \frac{3\gamma_2^2 - 3\gamma_2 + 1}{\Delta_2}.
	\]
	
	Using~$\frac{1}{\kappa(\gamma_2)} + \frac{1}{\gamma_2} = 3$ and~$\kappa(\gamma_2)\cdot \gamma_2 = \frac{\gamma_2^2}{3\gamma_2-1}$, when~$\gamma_1= \kappa(\gamma_2)$ the right-hand side above becomes
	\[
		\frac{1}{\gamma_2(3-\frac{3\gamma_2-1}{\gamma_2^2})} \cdot  \frac{3\gamma_2^2 - 3\gamma_2+1}{\Delta_2} = \frac{\gamma_2}{\Delta_2},
	\]
	proving that
	\begin{equation*}
		f_{\mathsf D}(\kappa(\gamma_2),\gamma_2,\tfrac{\bar{\gamma}_2}{\Delta_2}) = \frac{\gamma_2}{\Delta_2}.
	\end{equation*}
	
	The proof is then concluded by noting that
	\[
		\frac{\partial f_{\mathsf D}}{\partial \gamma_1}  = (\gamma_2+\bar{\gamma}_2 a)\cdot \frac{\partial}{\partial \gamma_1} \frac{\gamma_1}{\Delta} = (\gamma_2 + \bar{\gamma}_2 a) \cdot \frac{\gamma_2-1}{\Delta^2} < 0. \qedhere
	\]
\end{proof}

\begin{lemma}\label{lem_second_inter_case}
	If~$(\gamma_1,\gamma_2) \in \mathcal S$ with~$\gamma_2 > 1/2$ and~$\gamma_1 \in (\bar{\gamma}_2, \kappa(\gamma_2)]$, then
	\[
		\min\{f_{\mathsf S},f_{\mathsf B},f_{\mathsf D}\} = \begin{cases}
			f_{\mathsf S}&\text{if } a \le \tfrac{\bar{\gamma}_2}{\Delta_2};\\[.2cm]
			f_{\mathsf B}&\text{if } a > \tfrac{\bar{\gamma}_2}{\Delta_2}.
		\end{cases}
	\]
\end{lemma}
\begin{proof}
	Fix~$(\gamma_1,\gamma_2) \in \mathcal S$ as in the statement. By~\eqref{eq_bridge_zero},~\eqref{eq_direct_comp} and Lemma~\ref{lem_to_postpone}, we have
	\[
		f_{\mathsf S}(\gamma_1,\gamma_2,0+) < f_{\mathsf D}(\gamma_1,\gamma_2,0+) \qquad \text{and}\qquad f_{\mathsf S}(\gamma_1,\gamma_2,\tfrac{\bar{\gamma}_2}{\Delta_2}) \le f_{\mathsf D}(\gamma_1,\gamma_2,\tfrac{\bar{\gamma}_2}{\Delta_2}).
	\]
	
	Since both~$a \mapsto f_{\mathsf S}(\gamma_1,\gamma_2,a)$ and~$a \mapsto f_{\mathsf D}(\gamma_1,\gamma_2,a)$ are affine functions of~$a$, we obtain~$f_{\mathsf S}(\gamma_1,\gamma_2,a) < f_{\mathsf D}(\gamma_1,\gamma_2,a)$ for all~$a > 0$. Together with~\eqref{eq_direct_comp0}, this implies that for~$a \in (0,\tfrac{\bar{\gamma}_2}{\Delta_2}]$, the minimum of the three functions is~$f_{\mathsf S}$.
	Next,~\eqref{eq_direct_comp}, Lemma~\ref{lem_to_postpone} and~\eqref{eq_new_compBD} give
	\[
		f_{\mathsf B}(\gamma_1,\gamma_2,\tfrac{\bar{\gamma}_2}{\Delta_2}) \le  f_{\mathsf D}(\gamma_1,\gamma_2,\tfrac{\bar{\gamma}_2}{\Delta_2}) \qquad \text{and}\qquad \frac{\partial f_{\mathsf B}}{\partial a} <  \frac{\partial f_{\mathsf D}}{\partial a}.
	\]
	
	Together with~\eqref{eq_direct_comp0}, this implies that for~$a > \tfrac{\bar{\gamma}_2}{\Delta_2}$, the minimum of the three functions is~$f_{\mathsf B}$.
\end{proof}

\begin{lemma}\label{lem_third_last_case}
	Assume that~$(\gamma_1,\gamma_2) \in \mathcal S$ with~$\gamma_2 > 1/2$ and~$\gamma_1 > \kappa(\gamma_2)$. Then, letting
	\begin{equation}\label{eq_two_stars}
		a^*_1=a^*_1(\gamma_1,\gamma_2):=\frac{\bar{\gamma}_1 \bar{\gamma}_2}{\gamma_2 + \gamma_1\gamma_2 -1 }\qquad\text{and}\qquad a^*_2 = a^*_2(\gamma_1,\gamma_2):= \frac{\gamma_1 + \gamma_1\gamma_2 -1 }{\bar{\gamma}_1 \bar{\gamma}_2},
	\end{equation}
	we have
	\[
		\min\{f_{\mathsf S},f_{\mathsf B},f_{\mathsf D}\} = \begin{cases}
			f_{\mathsf S}&\text{if } a \le a^*_1;\\
			f_{\mathsf D}&\text{if } a^*_1< a \le a^*_2;\\
			f_{\mathsf B}&\text{if } a > a^*_2.
		\end{cases}
	\]
\end{lemma}
\begin{proof}
	Fix~$(\gamma_1,\gamma_2)$ as in the statement. By~\eqref{eq_bridge_zero},~\eqref{eq_direct_comp} and Lemma~\ref{lem_to_postpone}, we have
	\[
		f_{\mathsf S}(\gamma_1,\gamma_2,0+) < f_{\mathsf D}(\gamma_1,\gamma_2,0+) \qquad \text{and}\qquad f_{\mathsf S}(\gamma_1,\gamma_2,\tfrac{\bar{\gamma}_2}{\Delta_2}) \ge f_{\mathsf D}(\gamma_1,\gamma_2,\tfrac{\bar{\gamma}_2}{\Delta_2}).
	\]
	
	Hence, the two functions~$a \mapsto f_{\mathsf S}(\gamma_1,\gamma_2,a)$ and~$a \mapsto f_{\mathsf D}(\gamma_1,\gamma_2,a)$ coincide for some~$a \in [0,\tfrac{\bar{\gamma}_2}{\Delta_2}]$; by equating them, we see that this happens at~$a_1^*$. We then have~$f_{\mathsf S} \le f_{\mathsf D}$ for~$a \in (0,a^*_1]$ and~$f_{\mathsf D} \le f_{\mathsf S}$ for~$a \in [a^*_1, \tfrac{\bar{\gamma}_2}{\Delta_2}]$. We also know from~\eqref{eq_direct_comp0} that~$f_{\mathsf S} < f_{\mathsf B}$ for~$a \in (0,\tfrac{\bar{\gamma}_2}{\Delta_2})$. This implies that
		\[
			\min\{f_{\mathsf S},f_{\mathsf B},f_{\mathsf D}\} = \begin{cases} f_{\mathsf S}&\text{if } a \in (0,a^*_1];\\[.2cm]
			f_{\mathsf D}&\text{if } a \in (a^*_1,\tfrac{\bar{\gamma}_2}{\Delta_2}].\end{cases}
		\]

		Next, by~\eqref{eq_direct_comp}, Lemma~\ref{lem_to_postpone}, and~\eqref{eq_new_compBD}, we have
	\[
		f_{\mathsf B}(\gamma_1,\gamma_2,\tfrac{\bar{\gamma}_2}{\Delta_2}) \ge f_{\mathsf D}(\gamma_1,\gamma_2,\tfrac{\bar{\gamma}_2}{\Delta_2})\qquad \text{and}\qquad \frac{\partial f_{\mathsf B}}{\partial a} < \frac{\partial f_{\mathsf D}}{\partial a} .
	\]
	
		Hence, the two functions~$a \mapsto f_{\mathsf B}(\gamma_1,\gamma_2,a)$ and~$a \mapsto f_{\mathsf D}(\gamma_1,\gamma_2,a)$ coincide for some~$a \in [\tfrac{\bar{\gamma}_2}{\Delta_2},\infty)$; by equating them, we see that this happens at~$a^*_2$. Then,~$f_{\mathsf D} \le f_{\mathsf B}$ for~$a \in [\tfrac{\bar{\gamma}_2}{\Delta_2},a^*_2]$ and~$f_{\mathsf B} \le f_{\mathsf D}$ for~$a \ge a^*_2$.
		We also know from~\eqref{eq_direct_comp0} that~$f_{\mathsf B} < f_{\mathsf S}$ for~$a > \tfrac{\bar{\gamma}_2}{\Delta_2}$. Hence,
		\[
			\min\{f_{\mathsf S},f_{\mathsf B},f_{\mathsf D}\} = \begin{cases} f_{\mathsf D}&\text{if } a \in (\tfrac{\bar{\gamma}_2}{\Delta_2},a_2^*];\\[.2cm]
			f_{\mathsf B}&\text{if } a \in (a^*_2,\infty).\end{cases} \qedhere
		\]
\end{proof}

\begin{proof}[Proof of Proposition~\ref{prop_target_opt}]

Using the relation
	\begin{equation}
		\label{eq_symmetry}
		\nu_\circ(\gamma_1,\gamma_2,a) = a\cdot \mu_\circ(\gamma_2,\gamma_1,\tfrac{1}{a}),\qquad \circ \in \{\mathsf S,\mathsf B, \mathsf D\},
	\end{equation}
	Lemma~\ref{lem_first_easy_case}, Lemma~\ref{lem_second_inter_case} and Lemma~\ref{lem_third_last_case} respectively imply that:
	\begin{itemize}
		\item[($\mathrm a$)] if~$\gamma_1 \le 1/2$, then~$\nu_\star = \nu_{\mathsf S}$;
		\item[($\mathrm b$)] if~$\gamma_1 > 1/2$ and~$\gamma_2 \in (\bar{\gamma}_1,\kappa(\gamma_1)]$, then
			\[
				\nu_\star = \begin{cases}
					\nu_{\mathsf B} &\text{if } a \le \Delta_1/\bar{\gamma}_1;\\[.2cm]
					\nu_{\mathsf S} &\text{if } a > \Delta_1/\bar{\gamma}_1;
				\end{cases}
			\]
		\item[($\mathrm c$)] if~$\gamma_1 > 1/2$ and~$\gamma_2 \in (\kappa(\gamma_2),1]$, then, for~$a_1^*$ and~$a_2^*$ as in~\eqref{eq_two_stars},
			\[
				\nu_\star = \begin{cases}
					\nu_{\mathsf B} &\text{if } a \le a_1^*;\\[.2cm]
					\nu_{\mathsf D} &\text{if } a_1^* < a \le a_2^*;\\[.2cm]
					\nu_{\mathsf S} &\text{if } a > a_2^*.
				\end{cases}
			\]
	\end{itemize}

	We now consider all the cases that appear in the statement of Proposition~\ref{prop_target_opt}. It is useful to note that~\eqref{eq_gamma_kappa} implies that
	\begin{align*}
		\{(\gamma_1,\gamma_2) \in \mathcal S: \gamma_2 > 1/2,\; \gamma_1 > \kappa(\gamma_2)\} &= 
		\{(\gamma_1,\gamma_2) \in \mathcal S: \gamma_1 > 1/2,\; \gamma_2 > \kappa(\gamma_1)\} \\
		&= \left\{(\gamma_1,\gamma_2) \in \mathcal S:  \frac{1}{\gamma_1} + \frac{1}{\gamma_2} < 3\right\}.
	\end{align*}	
	\begin{itemize}
		\item
		Assume that~$\gamma_2 \le 1/2$ (and consequently,~$\gamma_1 > 1/2$ and~$\gamma_2 < \kappa(\gamma_1)$). Then, Lemma~\ref{lem_first_easy_case} implies that~$\mu_\star = \mu_{\mathsf S}$, and case ($\mathrm b$) above implies that~$\nu_\star=\nu_{\mathsf B}$ if~$a \le \Delta_1/\bar{\gamma}_1$, and~$\nu_\star = \nu_{\mathsf S}$ if~$a > \Delta_1/\bar{\gamma}_1$.
		\item
		Similarly, if~$\gamma_1 \le 1/2$, case~($\mathrm a$) above implies that~$\nu_\star = \nu_{\mathsf S}$, and Lemma~\ref{lem_second_inter_case} implies that~$\mu_\star = \mu_{\mathsf S}$ if~$a \le \bar{\gamma}_2/\Delta_2$, and~$\mu_\star = \mu_{\mathsf B}$ if~$a > \bar{\gamma}_2/\Delta_2$.
	\item Now assume that~$\gamma_1,\gamma_2 > 1/2$ and~$\tfrac{1}{\gamma_1}+\tfrac{1}{\gamma_2} > 3$. Note that the latter inequality gives~$\tfrac{\Delta_1}{\bar{\gamma}_1} \le \frac{\bar{\gamma}_2}{\Delta_2}$. Then, Lemma~\ref{lem_second_inter_case} and case~($\mathrm b$) above imply that~$(\mu_\star,\nu_\star)$ is equal to~$(\mu_{\mathsf S},\nu_{\mathsf B})$ when~$a \le \frac{\Delta_1}{\bar{\gamma}_1}$, equal to~$(\mu_{\mathsf S},\nu_{\mathsf S})$ when~$\frac{\Delta_1}{\bar{\gamma}_1} < a \le \frac{\bar{\gamma}_2}{\Delta_2}$, and equal to~$(\mu_{\mathsf B},\nu_{\mathsf S})$ when~$a > \frac{\bar{\gamma}_2}{\Delta_2}$.
	\item Finally, assume that~$\gamma_1,\gamma_2 > 1/2$ with~$\tfrac{1}{\gamma_1}+\tfrac{1}{\gamma_2} \le 3$. Then, Lemma~\ref{lem_third_last_case} and case ($\mathrm c$) above imply that~$(\mu_\star,\nu_\star)$ is equal to~$(\mu_{\mathsf S},\nu_{\mathsf B})$ if~$a \le a_1^*$, equal to~$(\mu_{\mathsf D},\nu_{\mathsf D})$ if~$a_1^* < a \le a_2^*$, and equal to~$(\mu_{\mathsf B},\nu_{\mathsf S})$ if~$a > a_2^*$. \qedhere
	\end{itemize}

\end{proof}

\subsection{Optimisation: proof of Proposition~\ref{prop_opt}}
\label{ss_proof_opt}

Recall that
\[A_\star:=\mu_\star \wedge (1+\bar{\gamma}_2 \nu_\star)\wedge (1+a-\Delta_2 \nu_\star+\bar{\gamma}_1 \mu_\star).\] 

We first compare the first two terms in the minimum (corresponding, respectively, to zero step and one step of the infection started at the root of type~1). Later, we will show that the third term (corresponding to two steps) is never smaller than the minimum between the first two.

In all that follows, we assume that $(\gamma_1,\gamma_2,a) \in \mathcal S \times (0,\infty)$.

\subsubsection{Comparison between zero and one step}
\begin{lemma}\label{lem_work_Astar}
Let
\begin{align*}
	\mathcal J = &\left\{(\gamma_1,\gamma_2,a): \gamma_1 \ge \frac12,\;\gamma_2 \le \frac{\gamma_1}{1+\gamma_1} \right\} \\
	&\cup \left\{ (\gamma_1,\gamma_2,a): \gamma_1 \ge \frac12,\;\frac{\gamma_1}{1+\gamma_1} \le \gamma_2 \le \frac12,\; a \le \frac{\gamma_1-\gamma_2}{\gamma_2-\gamma_1+\gamma_1\gamma_2}\right\}.
\end{align*}
We have
\[\mu_\star \le  1+ \bar{\gamma}_2 \nu_\star \quad \text{if }(\gamma_1,\gamma_2,a) \in \mathcal J\]
and
\[\mu_\star \ge  1+ \bar{\gamma}_2 \nu_\star\quad \text{if }(\gamma_1,\gamma_2,a) \in (\mathcal S \times (0,\infty)) \backslash \mathcal J.\]
\end{lemma}
\begin{proof}
For convenience of reference, we list the values:
\begin{align*}
\begin{array}{lp{.6cm}lp{.6cm}l} \mu_{\mathsf S} = \frac{1}{\gamma_1} + \frac{1}{\gamma_1}\cdot a,&&\mu_{\mathsf B}= \frac{1}{\gamma_1 \gamma_2} + \frac{\bar{\gamma}_2}{\gamma_1\gamma_2}\cdot a,&& \mu_{\mathsf D}= \frac{\gamma_2}{\Delta}+\frac{\bar{\gamma}_2}{\Delta}\cdot a,\\[.4cm]
1+\bar{\gamma}_2 \nu_{\mathsf{S}} = \frac{1}{\gamma_2} + \frac{\bar{\gamma}_2}{\gamma_2}\cdot a,&&1+\bar{\gamma}_2 \nu_{\mathsf B}= \frac{\gamma_1 \gamma_2 + \bar{\gamma}_1  \bar{\gamma}_2}{\gamma_1\gamma_2} + \frac{\bar{\gamma}_2}{\gamma_1 \gamma_2}\cdot a,&& 1+\bar{\gamma}_2 \nu_{\mathsf D}= \frac{\gamma_1\gamma_2}{\Delta}+ \frac{\gamma_1 \bar{\gamma}_2}{\Delta}\cdot a.
\end{array}
\end{align*}

Although~$\mu_{\mathsf S},\mu_{\mathsf B},\mu_{\mathsf D}, \nu_{\mathsf S},\nu_{\mathsf B}, \nu_{\mathsf D}$ are functions of~$\gamma_1,\gamma_2,a$, we will either omit all three variables as we have just done above, or write only~$a$ in the argument (for instance,~$\mu_{\mathsf S}(0+)$ denotes the value of~$\mu_{\mathsf S}$ when~$a$ approaches~$0$ from the right, and~$\gamma_1,\gamma_2$ are either unimportant or clear from the context). In the same spirit, we will use a prime symbol ($\mu_{\mathsf S}',\mu_{\mathsf B}',\mu_{\mathsf D}',\nu_{\mathsf S}', \nu_{\mathsf B}', \nu_{\mathsf D}'$)  to indicate the partial derivative with respect to~$a$; note that the values of these partial derivatives depend on~$\gamma_1$ and~$\gamma_2$ but not on~$a$, since these functions are all affine functions of~$a$ when~$\gamma_1,\gamma_2$ are fixed.

To inform the comparisons we have to carry out, we recall that the pair~$(\mu_\star,\nu_\star)$ is equal to one of the four pairs:~$(\mu_{\mathsf S},\nu_{\mathsf S})$,~$(\mu_{\mathsf S},\nu_{\mathsf B})$,~$(\mu_{\mathsf B},\nu_{\mathsf S})$,~$(\mu_{\mathsf D},\nu_{\mathsf D})$.
 We proceed through a series of claims.

\begin{claim}
We have~$\mu_{\mathsf D} \ge 1+\bar{\gamma}_2 \nu_{\mathsf D}$ for all~$(\gamma_1,\gamma_2,a) \in \mathcal S \times (0,\infty)$.
\end{claim}
\begin{proof}
This follows from noting that
\[\mu_{\mathsf D}(0+) = \frac{\gamma_2}{\Delta} \ge \frac{\gamma_1 \gamma_2}{\Delta} =  1+\bar{\gamma}_2 \nu_{\mathsf D}(0+),\qquad \mu_{\mathsf D}' = \frac{\bar{\gamma}_2}{\Delta} \ge \frac{\gamma_1\bar{\gamma}_2}{\Delta}=  \bar{\gamma}_2 \nu_{\mathsf D}'.\qedhere \]
\end{proof}

The above claim implies that~$\mu_\star \ge 1+\bar{\gamma}_2 \nu_\star$ wherever~$(\mu_\star,\nu_\star)=(\mu_{\mathsf D},\nu_{\mathsf D})$.

\begin{claim}
We have~$\mu_{\mathsf B} \ge 1+\bar{\gamma}_2 \nu_{\mathsf S}$ for all~$(\gamma_1,\gamma_2,a) \in \mathcal S \times (0,\infty)$.
\end{claim}
\begin{proof}
We have
\[
\mu_{\mathsf B}(0+) = \frac{1}{\gamma_1\gamma_2} \ge \frac{1}{\gamma_2} = 1 + \bar{\gamma}_2 \nu_{\mathsf S}(0+), \qquad   \mu_{\mathsf B}' = \frac{\bar{\gamma}_2}{\gamma_1\gamma_2} \ge \frac{\bar{\gamma}_2}{\gamma_2} = \bar{\gamma}_2 \nu_{\mathsf S}'. \qedhere
\]
\end{proof}
This claim implies that~$\mu_\star \ge 1+\bar{\gamma}_2 \nu_\star$ wherever~$(\mu_\star,\nu_\star)=(\mu_{\mathsf B},\nu_{\mathsf S})$.

\begin{claim}\label{cl_SB_target}
We have
\[
\mu_{\mathsf S} \ge 1+\bar{\gamma}_2 \nu_{\mathsf B} \quad \text{if } (\gamma_1,\gamma_2,a) \in \mathcal S \times (0,\infty) \text{ with } \gamma_2 \ge \frac12
\]
and
\[
\mu_{\mathsf S} \le 1+\bar{\gamma}_2 \nu_{\mathsf B} \quad \text{if } (\gamma_1,\gamma_2,a) \in \mathcal S \times (0,\infty) \text{ with } \gamma_2 \le \frac12.
\]
\end{claim}
\begin{proof}
Direct algebra shows that
\[
\mu_{\mathsf S}(0+) = \frac{1}{\gamma_1} \ge \frac{\gamma_1 \gamma_2 + \bar{\gamma}_1  \bar{\gamma}_2}{\gamma_1\gamma_2} = 1+\bar{\gamma}_2 \nu_{\mathsf B}(0+) \qquad \text{if and only if}\qquad \gamma_2 \ge \frac12
\]
and
\[
\mu_{\mathsf S}' = \frac{1}{\gamma_1} \ge \frac{\bar{\gamma}_2}{\gamma_1 \gamma_2} = \bar{\gamma}_2\nu_{\mathsf B}' \qquad \text{if and only if} \qquad \gamma_2 \ge \frac12. \qedhere
\]
\end{proof}
We write~$\{(\gamma_1,\gamma_2,a): (\mu_\star,\nu_\star)=(\mu_{\mathsf S},\nu_{\mathsf B}) \} = \mathcal J_{\mathsf{SB},1} \cup\mathcal J_{\mathsf{SB},2} \cup \mathcal J_{\mathsf{SB},3}$, where
\begin{align*}
&\mathcal J_{\mathsf{SB},1}:=\left\{(\gamma_1,\gamma_2,a): \gamma_1 \ge 1/2,\; \gamma_2 \le 1/2,\; a \le \frac{2\gamma_1-1}{\bar{\gamma}_1}\right\},\\[.2cm]
&\mathcal J_{\mathsf{SB},2}:=\left\{(\gamma_1,\gamma_2,a): \gamma_1 \ge  1/2,\; \gamma_2 \ge 1/2,\; H(\gamma_1,\gamma_2) \le \frac23,\; a \le \frac{2\gamma_1-1}{\bar{\gamma}_1}\right\},\\[.2cm]
&\mathcal J_{\mathsf{SB},3}:=\left\{(\gamma_1,\gamma_2,a): \gamma_1, \gamma_2 \ge 1/2,\; H(\gamma_1,\gamma_2) \ge \frac23,\; a \le \frac{\bar{\gamma}_1 \bar{\gamma}_2}{\gamma_1 \gamma_2 + \gamma_2 -1}\right\}.
\end{align*}
Then, Claim~\ref{cl_SB_target} implies that~$\mu_\star \ge 1+\bar{\gamma}_2 \nu_\star$ on~$\mathcal J_{\mathsf{SB},2} \cup \mathcal J_{\mathsf{SB},3}$ and~$\mu_\star \le 1+\bar{\gamma}_2 \nu_\star $ on~$\mathcal J_{\mathsf{SB},1}$.

It remains to treat the star-star case. For this, we will need two claims.
\begin{claim}\label{cl_tricky_star_star}
We have
\begin{equation}\label{eq_tricky_SS1}
\mu_{\mathsf S} \ge 1+\bar{\gamma}_2 \nu_{\mathsf S} \quad \text{if  either } \gamma_2 \ge \gamma_1 \text{ or } \left[\frac{\gamma_1}{1+\gamma_1} < \gamma_2 \le \gamma_1 \text{ and } a \ge \frac{\gamma_1-\gamma_2}{\gamma_2-\gamma_1+\gamma_1 \gamma_2}\right]
\end{equation}
and
\begin{equation}\label{eq_tricky_SS2}
\mu_{\mathsf S} \le 1+\bar{\gamma}_2 \nu_{\mathsf S} \quad \text{if  either }  \left[\frac{\gamma_1}{1+\gamma_1} < \gamma_2 \le \gamma_1 \text{ and } a \le \frac{\gamma_1-\gamma_2}{\gamma_2-\gamma_1+\gamma_1 \gamma_2}\right] \text{ or } \gamma_2 \le \frac{\gamma_1}{1+\gamma_1}.
\end{equation}
\end{claim}
(Note that~$\frac{\gamma_1-\gamma_2}{\gamma_2-\gamma_1+\gamma_1 \gamma_2} \ge 0$ when~$\frac{\gamma_1}{1+\gamma_1} < \gamma_2 \le \gamma_1$).
\begin{proof}
The statement again follows from comparing intercepts and slopes:
\[
	\mu_{\mathsf S}(0+) = \frac{1}{\gamma_1}  \ge \frac{1}{\gamma_2} = 1+\bar{\gamma}_2 \nu_{\mathsf S}(0+) \qquad \text{if and only if } \qquad \gamma_2 \ge \gamma_1
\]
and
\[
\mu_{\mathsf S}' = \frac{1}{\gamma_1} \ge \frac{\bar{\gamma}_2}{\gamma_2} = \bar{\gamma}_2 \nu_{\mathsf S}' \qquad \text{if and only if}\qquad \gamma_2 \ge \frac{\gamma_1}{1+\gamma_1}.
\]
\end{proof}

\begin{claim}
For~$\gamma_1 \ge \frac12$, we have
\begin{equation}\label{eq_exponentsSS1}
\frac{\gamma_1-\gamma_2}{\gamma_2-\gamma_1+\gamma_1\gamma_2} \ge \frac{2\gamma_1 - 1}{\bar{\gamma}_1} \qquad \text{if } \gamma_2 \in \left(\frac{\gamma_1}{1+\gamma_1},\;\frac12\right]
\end{equation}
and
\begin{equation}\label{eq_exponentsSS2}
\frac{\gamma_1-\gamma_2}{\gamma_2-\gamma_1+\gamma_1\gamma_2} \le \frac{2\gamma_1 - 1}{\bar{\gamma}_1} \qquad \text{if } \gamma_2 \in \left[\frac12, \; \gamma_1\right].
\end{equation}
\end{claim}
\begin{proof}
Equality holds when~$\gamma_2=\frac12$, and 
\[
\frac{\mathrm{d}}{\mathrm{d}\gamma_2} \left(\frac{\gamma_1-\gamma_2}{\gamma_2-\gamma_1+\gamma_1\gamma_2} \right) = -\frac{\gamma_1^2}{(\gamma_2-\gamma_1 +\gamma_1\gamma_2)^2} < 0. \qedhere
\]
\end{proof}

Now, recall that~$\{(\gamma_1,\gamma_2,a): (\mu_\star,\nu_\star)=(\mu_{\mathsf S},\nu_{\mathsf S}) \} = \mathcal J_{\mathsf{SS},1} \cup\mathcal J_{\mathsf{SS},2} \cup \mathcal J_{\mathsf{SS},3}$, where
\begin{align*}
&\mathcal J_{\mathsf{SS},1}:= \left\{(\gamma_1,\gamma_2,a): \gamma_1 \le \frac12,\; \gamma_2 \ge \frac12,\; a \ge \frac{\bar{\gamma}_2}{2\gamma_2-1}\right\},\\[.2cm]
&\mathcal J_{\mathsf{SS},2}:= \left\{(\gamma_1,\gamma_2,a): \gamma_1,\gamma_2 \ge \frac12,\; H(\gamma_1,\gamma_2) \le \frac23,\; \frac{2\gamma_1-1}{\bar{\gamma}_1} \le a \le \frac{\bar{\gamma}_2}{2\gamma_2-1}\right\},\\[.2cm]
&\mathcal J_{\mathsf{SS},3}:= \left\{(\gamma_1,\gamma_2,a): \gamma_1 \ge \frac12,\; \gamma_2 \le \frac12, \; a \ge \frac{2\gamma_1-1}{\bar{\gamma}_1}\right\}.
\end{align*}

On~$\mathcal J_{\mathsf{SS},1}$ we have~$\gamma_2 \ge \gamma_1$, so applying~\eqref{eq_tricky_SS1} directly, we see that~$\mu_\star \ge 1+\bar{\gamma}_2\nu_\star$ on~$\mathcal J_{\mathsf{SS},1}$. On~$\mathcal J_{\mathsf{SS},2}$, we either have~$\gamma_2 \ge \gamma_1$ (which again gives~$\mu_\star \ge \nu_\star$ by~\eqref{eq_tricky_SS1} directly) or we have both~$\frac12 \le \gamma_2 \le \gamma_1$ and~$a \ge \frac{2\gamma_1-1}{\bar{\gamma}_1}$; in this case,~\eqref{eq_exponentsSS2} gives~$a \ge \frac{\gamma_1-\gamma_2}{\gamma_1\gamma_2+\gamma_2-1}$, and then,~\eqref{eq_tricky_SS1} gives~$\mu_\star \ge 1+\bar{\gamma}_2\nu_\star$. Next, note that~\eqref{eq_exponentsSS1} implies that~$\frac{2\gamma_1-1}{\bar{\gamma}_1} \le \frac{\gamma_1-\gamma_2}{\gamma_2-\gamma_1+\gamma_1\gamma_2}$ on~$\mathcal J_{\mathsf{SS},3}$. We define
\begin{align*}
& \mathcal J_{\mathsf{SS},3,1}:=\mathcal J_{\mathsf{SS},3} \cap \left\{(\gamma_1,\gamma_2,a):\frac{\gamma_1}{1+\gamma_1} \le \gamma_2 \le \frac12,\; \frac{2\gamma_1-1}{\bar{\gamma}_1} \le a \le \frac{\gamma_1- \gamma_2}{\gamma_2-\gamma_1+\gamma_1\gamma_2}  \right\},\\[.2cm]
& \mathcal J_{\mathsf{SS},3,2}:=\mathcal J_{\mathsf{SS},3} \cap \left\{(\gamma_1,\gamma_2,a):\frac{\gamma_1}{1+\gamma_1} \le \gamma_2 \le \frac12,\; a \ge \frac{\gamma_1- \gamma_2}{\gamma_2-\gamma_1+\gamma_1\gamma_2}  \right\},\\[.2cm]
&\mathcal J_{\mathsf{SS},3,3}:=\mathcal J_{\mathsf{SS},3} \cap \left\{(\gamma_1,\gamma_2,a): \gamma_2 \le \frac{\gamma_1}{1+\gamma_1}\right\}.
\end{align*}

Now,~\eqref{eq_tricky_SS2} implies that~$\mu_\star \le 1+\bar{\gamma}_2 \nu_\star$ on~$\mathcal J_{\mathsf{SS},3,1} \cup \mathcal J_{\mathsf{SS},3,3}$, and~\eqref{eq_tricky_SS1} implies that and~$\mu_\star \ge 1+\bar{\gamma}_2 \nu_\star$ on~$\mathcal{J}_{\mathsf{SS},3,2}$. This concludes the treatment of all regimes.

Putting everything together, we have proved that~$\mu_\star \le 1+\bar{\gamma}_2 \nu_\star$ on~$\mathcal J_{\mathsf{SB},1} \cup \mathcal J_{\mathsf{SS},3,1} \cup \mathcal J_{\mathsf{SS},3,3}$, and~$\mu_\star \ge 1+\bar{\gamma}_2 \nu_\star$ everywhere else.
\end{proof}

\subsubsection{Two steps are never optimal}
We define the functions
\begin{equation*}
		\Phi(\mu,\nu):=-\gamma_1 \cdot \mu + \bar{\gamma}_2 \cdot \nu +1, \qquad \Psi(\mu,\nu):=\bar{\gamma}_1\cdot \mu - \gamma_2 \cdot \nu + a, \qquad \mu, \nu \in \mathbb R.
\end{equation*}

Note that we omit the dependence on~$\gamma_1,\gamma_2,a$. Also, we will mostly care about these functions for~$\mu,\nu > 0$, but it is convenient to allow these variables to be negative.
Next, define~$\varphi(\mu)$ and~$\psi(\nu)$ as the solutions of
\[
	\Phi(\mu,\varphi(\mu))=0\qquad \text{and}\qquad \Psi(\psi(\nu),\nu)=0,
\]
that is,
\begin{equation*}
	\varphi(\mu) := \frac{\gamma_1}{\bar{\gamma}_2} \mu - \frac{1}{\bar{\gamma}_2},\qquad \psi(\nu):= \frac{\gamma_2}{\bar{\gamma}_1} \nu - \frac{1}{\bar{\gamma}_1}a.
\end{equation*}

It is readily checked that
\begin{equation}
	\label{eq_bridge_equation}
	\mu_{\mathsf B} = \varphi^{-1}(\nu_{\mathsf S}) \qquad \text{and} \qquad \nu_{\mathsf B} = \psi^{-1}(\mu_{\mathsf S}).
\end{equation}

Noting that~$\psi^{-1}(\mu) = \frac{\bar{\gamma}_1}{\gamma_2}\mu + \frac{1}{\gamma_2} a$,
we have~$\varphi(0) < \psi^{-1}(0)$ and~$\varphi' > (\psi^{-1})'$ (the latter inequality follows from~$\gamma_1+\gamma_2 > 1$). This implies that there is a (unique) value of~$\mu$ for which~$\varphi(\mu)=\psi^{-1}(\mu)$; it is easy to check that this value is~$\mu_{\mathsf D}$. Similarly,~$\nu_{\mathsf D}$ is the (unique) value of~$\nu$ satisfying~$\psi(\nu)=\varphi^{-1}(\nu)$. Hence, we have
\begin{equation}
\label{eq_direct_conv}
\varphi(\mu_{\mathsf D}) = \nu_{\mathsf D},\qquad \psi(\nu_{\mathsf D}) = \mu_{\mathsf D},\qquad \psi(\nu_{\mathsf D}) = \varphi^{-1}(\nu_{\mathsf D}).
\end{equation}

It will be useful to note that
\begin{equation}\label{eq_useful_varphi}
	\varphi(\mu) < \psi^{-1}(\mu) \quad \text{ for all } \mu \in (-\infty, \mu_{\mathsf D}).
\end{equation}

\begin{lemma}\label{lem_positive_PP}
	For any~$(\gamma_1,\gamma_2) \in \mathcal S$ and~$a > 0$, we have
	\[\Phi(\mu_\star,\nu_\star) \ge 0 \qquad \text{and} \qquad \Psi(\mu_\star,\nu_\star) \ge 0.\]
\end{lemma}
\begin{proof}
	The two inequalities are proved in the same way, so we only prove the first.
	 Recall that we have~$\Phi(\mu_\star,\varphi(\mu_\star))=0$ by the definition of~$\varphi$. Moreover,~$\nu \mapsto \Phi(\mu,\nu)$ is increasing. Hence, it suffices to prove that
	\begin{equation*}
	\varphi(\mu_\star) \le \nu_\star\qquad \text{for all } (\gamma_1,\gamma_2) \in \mathcal S,\; a > 0.
	\end{equation*}
	We verify this inequality for each of the four possible pairs given by Proposition~\ref{prop_target_opt}:
	\begin{itemize}
		\item Case~$(\mu_\star,\nu_\star)=(\mu_{\mathsf D},\nu_{\mathsf D})$: In this case we have~$\varphi(\mu_\star)=\nu_\star$ by~\eqref{eq_direct_conv}.
		\item Case~$(\mu_\star,\nu_\star)=(\mu_{\mathsf B},\nu_{\mathsf S})$: Here we have~$\varphi(\mu_\star)=\varphi(\mu_{\mathsf B})=\nu_{\mathsf S} = \nu_\star$, where the second equality follows from the definition of~$\mu_{\mathsf B}$.
		\item Case~$(\mu_\star,\nu_\star)=(\mu_{\mathsf S},\nu_{\mathsf S})$: We have~$\mu_\star=\min\{\mu_{\mathsf S},\mu_{\mathsf B},\mu_{\mathsf D}\} \le \mu_{\mathsf B}$; since~$\varphi(\mu_{\mathsf B}) = \nu_{\mathsf S}$ (by the definition of~$\mu_{\mathsf B}$) and~$\varphi$ is increasing, we obtain~$\varphi(\mu_{\star}) \le \nu_{\mathsf S} = \nu_\star$.
		\item Case~$(\mu_\star,\nu_\star)=(\mu_{\mathsf S},\nu_{\mathsf B})$: We have~$\mu_\star=\min\{\mu_{\mathsf S},\mu_{\mathsf B},\mu_{\mathsf D}\} \le \mu_{\mathsf D}$; then, by~\eqref{eq_useful_varphi}, we have~$\varphi(\mu_\star) \le \psi^{-1}(\mu_\star)$. Moreover,~$\nu_\star=\nu_{\mathsf B} = \psi^{-1}(\mu_{\mathsf S})=\psi^{-1}(\mu_\star)$. \qedhere
	\end{itemize}
\end{proof}

\begin{lemma}\label{lem_can_discard_two} For any~$(\gamma_1,\gamma_2) \in \mathcal S$ and~$a > 0$, we have
	\[1+a-\Delta_2 \nu_\star+\bar{\gamma}_1 \mu_\star\ge \min\left\{\mu_\star,\;1+\bar{\gamma}_2 \nu_\star \right\}.\]
	Consequently,
	\begin{equation*}
		A_\star = \min\left\{\mu_\star,\; 1+\bar{\gamma}_2 \nu_\star \right\}.
	\end{equation*}
\end{lemma}
\begin{proof}
	First assume that~$\gamma_2 < \frac12$. We then have~$\Delta_2 = 0\vee (2\gamma_2-1)= 0$ and, by Proposition~\ref{prop_target_opt},~$\mu_\star = \mu_{\mathsf S}$. Then,
	\[1+a-\Delta_2 \nu_\star + \bar{\gamma}_1 \mu_\star = 1+a+\bar{\gamma}_1 \cdot \left( \frac{1}{\gamma_1} + \frac{1}{\gamma_1} a\right) = \frac{1}{\gamma_1} + \frac{1}{\gamma_1}a = \mu_\star.\]

	Now assume that~$\mu_2 \ge \frac12$, so that~$\Delta_2 = 2\gamma_2 - 1$. We then have
	\[
		1+a-\Delta_2 \nu_\star+\bar{\gamma}_1 \mu_\star = 1 + \bar{\gamma}_2 \nu_\star + \Psi(\mu_\star,\nu_\star) \ge 1 + \bar{\gamma}_2 \nu\star
	\]
	by Lemma~\ref{lem_positive_PP}.
\end{proof}

\begin{proof}[Proof of Proposition~\ref{prop_opt}]
	The statement follows readily by combining Proposition~\ref{prop_target_opt},  Lemma \ref{lem_work_Astar} and Lemma \ref{lem_can_discard_two}.
\end{proof}

\end{appendix}

\section*{Acknowledgements}
This research was partially carried out during CH's visit to the University of Warwick funded by AUFF visit grant AUFF-E-2024-6-18. 


\bibliographystyle{abbrv}
\bibliography{lit}

\end{document}